\documentclass[11pt]{article}
      
      \usepackage{dsfont}
      \usepackage{amsmath, amsfonts, amsthm, amssymb}
      \usepackage{authblk, color, bm, graphicx, empheq, enumitem}
      \usepackage[margin = 1.5cm, font = small, labelfont = bf, labelsep = period]{caption}
      \usepackage{xspace}
      \usepackage[utf8]{inputenc}
      \usepackage[round]{natbib}
      \usepackage{mathrsfs}
      \usepackage{tikz}
      \usetikzlibrary{arrows.meta}
      \usetikzlibrary{calc}

      \usepackage{fullpage}[2cm]
      \newtheorem{thm}{Theorem}
      \newtheorem{coro}[thm]{Corollary}
      \newtheorem{prop}[thm]{Proposition}
      \newtheorem{lem}[thm]{Lemma}

      \newtheorem{defi}{Definition}
      
      \newtheorem{ex}{Example}
      \newtheorem{rem}{Remark}

      \newtheorem{app_thm}{Theorem}[section]
      
      \newtheorem{app_prop}[app_thm]{Proposition}
      \newtheorem{app_lem}[app_thm]{Lemma}

      \newtheorem{app_defi}{Definition}[section]
      
      \newtheorem{app_ex}{Example}[section]
      \newtheorem{app_rem}{Remark}[section]

      \newcommand{\eg}{\textit{e.g.}}
      \newcommand{\ie}{\textit{i.e.}}

      \newcommand{\set}[1]{\mathcal{#1}}
      \newcommand{\dom}{\mathop{\rm{dom}}}
      \newcommand{\epi}{\mathop{\rm{epi}}}
      \newcommand{\R}{\mathbb R}
      \newcommand{\subj}{\rm{s.t.}}
      \newcommand{\supp}{\mathcal Z}
      \newcommand{\amb}[1]{\mathscr{#1}}
      \newcommand{\EP}[1]{{\mathbb{E}}_{\mathbb{P}}}
      \newcommand{\EPS}[1]{{\mathbb{E}}_{\mathbb{P}^{\rm S}}}
      \newcommand{\EPSi}[1]{{\mathbb{E}}_{\mathbb{P}^{\rm S}_i}}
      \newcommand{\bP}{{\mathbb{P}}}
      \newcommand{\bmt}[1]{\tilde{\bm{#1}}}

      \newcommand{\rev}[1]{{\color{black} #1}}

      \title{ A Unified Theory of Robust and Distributionally Robust Optimization via the Primal-Worst-Equals-Dual-Best Principle}

      \author[1]{{Jianzhe Zhen}
      }
      \author[2]{Daniel Kuhn}
      \author[3]{Wolfram Wiesemann}

      \affil[1]{\small \textit{School of Economics and Management, University of Chinese Academy of Sciences, China}}
      \affil[2]{\small \textit{College of Management of Technology, École Polytechnique Fédérale de Lausanne, Switzerland}}
      \affil[3]{\small \textit{Imperial College Business School, Imperial College London, United Kingdom}}

\begin{document}

      \maketitle

      \begin{abstract}
      Robust and distributionally robust optimization are modeling paradigms for decision-making under uncertainty where the uncertain parameters are only known to reside in an uncertainty set or are governed by any probability distribution from within an ambiguity set, respectively, and a decision is sought that minimizes a cost function under the most adverse outcome of the uncertainty. In this paper, we develop a rigorous and general theory of robust and distributionally robust nonlinear optimization using the language of convex analysis. Our framework is based on a generalized `primal-worst-equals-dual-best' principle that establishes strong duality between a semi-infinite primal worst and a non-convex dual best formulation, both of which admit finite convex reformulations.
      This principle offers an alternative formulation for robust optimization problems that obviates the need to mobilize the machinery of abstract semi-infinite duality theory to prove strong duality in distributionally robust optimization. We illustrate the modeling power of our approach through convex reformulations for distributionally robust optimization problems whose ambiguity sets are defined through general optimal transport distances, which generalize earlier results for Wasserstein ambiguity sets.

      \textbf{Keywords:} (Distributionally) Robust Optimization, Convex Analysis, Optimal Transport.
      \end{abstract}

      \section{Introduction}

      Mathematical optimization problems frequently require decisions to be taken under partial or complete lack of information about key problem parameters: the topology of a truss needs to be designed before the magnitudes and directions of the external forces acting upon it are known, a portfolio of financial assets needs to be built without knowledge of the future asset price movements, the energy production of a power plant needs to be fixed several hours before the demands, outputs of intermittent generators and plant/line failures are known, and the classifying hyperplane of a support vector machine needs to be selected under incomplete knowledge of the data generating distribution. These (and many other) applications have in common that key problem parameters are not only to be considered random, but they are also governed by probability distributions that are at least partially unknown.

      In the last 25 years, robust and distributionally robust optimization have emerged as promising techniques to model, analyze and optimize decisions under risk (where some problem parameters constitute random variables) and ambiguity (where the underlying distributions are only partially known). Robust optimization assumes that the uncertain problem parameters can take on any value from within a pre-specified \emph{uncertainty set}, whereas distributionally robust optimization models the uncertain parameters as random variables whose underlying probability distribution can be any distribution from within a pre-specified \emph{ambiguity set}. In both cases, the decision maker seeks to determine the best decision in view of the worst realization of the uncertainty; this is often depicted as a game \mbox{between the decision maker and an adversary nature that is `in charge' of the uncertainty.}

      The vast majority of research in robust and distributionally robust optimization focuses on well-structured conic optimization problems such as linear, second-order cone and semi-definite programs. While the presence of structure simplifies the exposition and ensures computational tractability, it requires similar arguments to be redeveloped for different problem classes, and it obfuscates the view on the underlying principles in their full generality. This concern has been noted by several researchers, and various attempts have been made to extend the theory of robust and distributionally robust optimization to general convex optimization problems. Since the main focus of these works is a computational one, however, mathematical subtleties that emerge from this generalization are often either incorrectly addressed or disregarded altogether.

      \rev{
      In this paper, we develop a general theory of robust and distributionally robust optimization from first principles, using the language of convex analysis. Section~\ref{sec:co} first revisits classical duality results for convex optimization and derives explicit dual optimization problems involving the conjugates of the objective and constraint functions. Our analysis allows the objective and constraint functions of the primal problem to be arbitrary extended real-valued proper, closed and convex functions. This generality is crucial since we will dualize problems that involve implicit constraints or conjugates, but it significantly complicates our analysis. We then leverage the results of Section~\ref{sec:co} to build a unified theory of robust (Section~\ref{sec:rco}) and distributionally robust (Sections~\ref{sec:drco} and~\ref{sec:extensions}) optimization problems as well as modern data-driven optimization problems (Section~\ref{sec:dot}), where the objective function, the constraints as well as the uncertainty or ambiguity set are described in terms of generic convex functions. At the heart of our framework lies a generalized `primal-worst-equals-dual-best' principle that establishes strong duality between a semi-infinite primal worst and a non-convex dual best formulation, \mbox{both of which can be reformulated as finite convex optimization problems.}

      Our key contributions may be summarized as follows.
      \begin{enumerate}
        \item[\emph{(i)}] We propose a unified theory of robust (Section~\ref{sec:rco}) and distributionally robust optimization with moment-based (Sections~\ref{sec:drco} and~\ref{sec:extensions}) and optimal transport-based (Section~\ref{sec:dot}) ambiguity sets. In particular, we derive easily verifiable conditions for strong duality in distributionally robust optimization from first principles of finite-dimensional convex analysis, as opposed to the abstract moment conditions traditionally required by semi-infinite duality theory.
        \item[\emph{(ii)}] Classical texts on robust optimization either study robust programs with linear (or quadratic) constraint functions and conic inequalities \citep{ben09,bn98} or they study robust programs with nonlinear (convex-concave) constraint functions but classical inequalities \citep{bgnv05}. By catering both for nonlinear functions and conic inequalities, we significantly enlarge the pool of robust and distributionally robust programs that admit finite convex reformulations. We also derive convex reformulations of distributionally robust programs with general optimal transport-based ambiguity sets. The flexibility to shape the transportation cost function allows modelers to control the likelihood that the uncertain parameters will fall into particular regions of the sample space.
        \item[\emph{(iii)}] We carefully account for subtle technical issues that have often been neglected in the related literature but are crucial for a rigorous treatment of extended real-valued functions. 
      \end{enumerate}}

      Robust optimization problems are traditionally solved by dualizing the embedded maximization over all possible uncertainty realizations in the primal worst problem \citep{el97, eol98, bn98, bn99, bn00, bn02, bs04}. If the embedded maximization problems represent linear conic optimization problems, then the duals can be constructed explicitly in terms of the original problem data \citep{ben09}. If the embedded maximization problems constitute generic convex optimization problems, on the other hand, they may not admit explicit duals. Instead, the dual objective function is only implicitly defined as the infimum of the Lagrangian function with respect to the uncertain parameters. Using Fenchel duality, \cite{bdv15} show that the optimal values of the embedded maximization problems can be expressed as differences between the support function of the uncertainty set and the partial conjugates of the constraint (or objective) functions with respect to the uncertain parameters. Instead of the support function of the uncertainty set, our reformulation explicitly involves the conjugates of the functions characterizing the uncertainty set and may therefore be easier to implement and automate. We refer to \citet{ben09}, \citet{bbc11} and \citet{gmt14} for comprehensive reviews of robust optimization and its manifold applications.

      Alternatively, robust optimization problems can be studied from the perspective of the dual best problem. For bounded uncertainty sets, \cite{bbt09} show that the optimal value of the primal worst problem coincides with that of the dual best problem if a Slater condition holds. \citet{gbbd14} show how the non-convexities of the dual best problem can be eliminated in robust optimization problems with linear objective and linear conic constraint functions over uncertainty sets described by general convex functions. This result was later extended by \citet{gd15} to general robust convex optimization problems. Our results extend the works of \citet{bbt09} and \citet{gd15} to robust optimization problems with unbounded uncertainty sets, which will prove essential when we apply our results to uncertainty quantification and distributionally robust optimization problems.

      Tractable reformulations for distributionally robust optimization problems can be derived in different ways. \cite{bn00}, \cite{bs04}, \cite{ce06}, \cite{Nemirovski_Shapiro_2006}, \cite{css07}, \cite{xu2012optimization} and \cite{bgk18a} rely on classical probability bounds (such as Hoeffding's inequality or Bernstein bounds) or statistical hypothesis tests to derive tractable reformulations. In contrast, \cite{ElGhaoui_Oks_Oustry_2003}, \cite{bp05}, \cite{Delage_Ye_2010}, \cite{xm12} and \cite{wks14} dualize the uncertainty quantification problem embedded in the distributionally robust optimization problem and apply techniques from standard robust optimization to replace the semi-infinite dual with a finite reformulation. To ensure strong duality between the primal uncertainty quantification problem and its semi-infinite dual, this literature stream usually relies on results from semi-infinite duality theory that are very general but can be tedious to verify and that are prone to misinterpretations \citep{Isii59, Isii62, s01}. A popular condition is to check whether the bounds on the generalized moments imposed by the ambiguity set belong to the interior of some moment cone generated by all non-negative measures (not only the probability measures) on the prescribed support set; \rev{see, {\em e.g.}, \citet[Propostion~3.4]{s01}}. Despite being convex, this moment cone usually lacks an explicit description in terms of simple convex constraints. In contrast, our conditions for strong duality, which are based on our generalized primal-worst-equals-dual-best principle, are typically easy to verify, both theoretically and algorithmically (\emph{e.g.}, via the solution of a convex optimization problem).

      Using the primal-worst-equals-dual-best principle to construct finite reductions of uncertainty quantification problems
      was first proposed for the subclass of chance constrained programs over restricted classes of ambiguity sets by \citet{JIDRCO, Hanasusanto_Roitch_Kuhn_Wiesemann_2017}. These papers, however, still rely on semi-infinite duality  theory (and its aforementioned shortcomings) to ensure strong duality.
      \citet{httom15} derive a finite reduction similar to ours by applying the Richter-Rogosinsky theorem \citep[Theorem~7.37]{shapiro2009lectures} and a subsequent induction argument directly to the primal uncertainty quantification problem. Since the focus of that work is on uncertainty quantification, however, it does not study the dual of the uncertainty quantification problem, which is essential to obtain tractable reformulations for distributionally robust optimization problems.

      \paragraph*{Notation.}
      We set $\overline{\mathbb R}=[-\infty,+\infty]$. \rev{The calligraphic letters~$\set I$, $\set J$, $\set K$, $\set L$ and the corresponding capital Roman letters $I$, $J$, $K$, $L$ are reserved for finite index sets and their respective cardinalities, \ie, $\set{I}= \{1,\dots, I\}$ etc. The subscript~$0$ indicates that the index set additionally includes $0$, \ie, $\set{I}_0 =   \{0, \ldots , I\}$ etc. We use ri$(\set{X})$ to denote the relative interior of a set $\set{X} \subseteq \R^{d_{\bm x}}$. }
      
      \section{Convex Optimization} \label{sec:co}

      In this section we adapt existing duality results to generic convex optimization problems with extended real-valued objective and constraint functions. This flexibility allows us to work with optimization problems whose objective and constraints involve conjugates. We also discuss regularity conditions under which the primal and the dual problems are solvable. As we will see later on, solvability is essential for the existence of worst-case scenarios and worst-case distributions in robust and distributionally robust optimization, respectively. By themselves, the results in this section are not new, however they are dispersed throughout the literature, and they often miss subtle but---in view of our applications in later sections---crucial regularity conditions.

	\rev{Throughout the paper we use the language convex analysis. Thus, we adopt the usual definitions of the domain $\dom(f)$, the epigraph $\epi(f)$, the conjugate~$f^*$ and the biconjugate~$f^{**}$ of an extended real-valued function~$f:\mathbb R^{d_{\bm x}}\rightarrow \overline \R$. As usual, we call $f$ proper and closed if its epigraph is a nonempty closed set that contains no vertical line. In addition, we use~$\delta_{\set X}$ and~$\delta^*_{\set X}$ to denote the indicator function and the support function of a set~$\set X\subseteq \R^{d_{\bm x}}$. If $f$ is  proper, closed and convex, then we define its convex perspective~$\underline f:\mathbb R^{d_{\bm x}}\times \overline \R\rightarrow \overline \R $ through $\underline f(\bm x,t)=tf(\bm x/t)$ if~$t>0$; $=\delta^* _{\dom(f^*)}(\bm x)$ if~$t=0$. This definition ensures that~$\underline f$ is proper, closed and convex. To avoid clutter, we henceforth write somewhat informally $tf(\bm x/t)$ instead of~$\underline f(\bm x,t)$ even if~$t=0$. Rigorous definitions of the above key concepts of convex analysis and a nuanced discussion of the inherent subtleties are provided in Appendix~\ref{app_before_A} in the Electronic Companion. Next, we} introduce Slater conditions for both sets and optimization problems. This distinction will enable us to characterize the uncertainty sets whose associated robust optimization problems are amenable to finite convex reformulations using the machinery of strong convex duality.

      \begin{defi}[Slater Condition for Sets]\label{def:Slater_sets}
      The vector $\bm{x}^{\textnormal{S}}$ is a Slater point of the set $\mathcal{X}$ represented by $\mathcal{X} = \{ \bm x \in \R^{d_{\bm x}} \ | \ f_i (\bm x) \le 0 \;\; \forall i \in \mathcal{I}, \ h_j (\bm x) = 0 \;\; \forall j \in \mathcal{J} \}$ if
      \emph{\emph{(i)}} $\bm{x}^{\textnormal{S}} \in {\rm ri} (\dom (f_i))$ as well as~$\bm{x}^{\textnormal{S}} \in {\rm ri} (\dom (h_j))$ for all $i\in\set I$ and $j\in\set J$;
      \emph{(ii)} $\bm{x}^{\textnormal{S}} \in \mathcal{X}$;
      and \emph{(iii)} $f_i (\bm{x}^{\textnormal{S}}) < 0$ for all $i \in \mathcal{I}$ such that $f_i$ is nonlinear. The Slater point $\bm{x}^{\textnormal{S}}$ is strict if $f_i (\bm{x}^{\textnormal{S}}) < 0$ for all $i \in \mathcal{I}$.
      \end{defi}

      Note that the definition of a Slater point depends on the representation of the set $\mathcal{X}$. In fact, $\mathcal{X} = \{ 0 \}$ has a strict Slater point if represented as $\mathcal{X} = \{ x \in \R \ | \ x = 0 \}$, whereas the alternative representation $\mathcal{X} = \{ x \in \R \ | \ x^2 \le 0 \}$ does not admit a Slater point.

      \begin{defi}[Slater Condition for Optimization Problems]\label{def:Slater_problems}
      The vector $\bm{x}^{\textnormal{S}}$ is a Slater point of the minimization problem $\inf \{ f_0 (\bm{x}) \ | \ \bm{x} \in \mathcal{X} \}$, where $\mathcal{X}$ is represented as in Definition~\ref{def:Slater_sets}, if $\bm{x}^{\textnormal{S}}$ is a Slater point of $\mathcal{X}$ and $\bm{x}^{\textnormal{S}} \in  {\rm ri} (\dom (f_0))$. \mbox{The Slater point $\bm{x}^{\textnormal{S}}$ is strict if it is a strict Slater point of $\mathcal{X}$.}
      \end{defi}

      For a maximization problem, we replace the requirement $\bm{x}^{\textnormal{S}} \in  {\rm ri} (\dom (f_0))$ in Definition~\ref{def:Slater_problems} with $\bm{x}^{\textnormal{S}} \in  {\rm ri} (\dom (-f_0))$. Consider now a generic nonlinear optimization problem of the following form.

      \begin{samepage}
      \begin{empheq}[box=\fbox]{equation}\label{eq:p-co}
      \begin{array}{l@{\quad}l@{\qquad}l}
      \displaystyle \inf & \displaystyle f_0(\bm x) \\
      \displaystyle \subj & \displaystyle f_i(\bm x) \le 0 & \displaystyle \forall i \in \set{I} \\
      & \displaystyle \bm x \text{ free}
      \end{array}
      \tag{P}
      \end{empheq}
      $\mspace{310mu}$ \textbf{Primal Problem}
      \end{samepage}

      \noindent Here, the objective and constraint functions are extended real-valued functions $f_i:\R^{d_{\bm x}}\rightarrow  \overline{\R}$ for $i\in \set I_0$. In the remainder of this section, we assume that problem~\eqref{eq:p-co} is convex, that is, we assume that its objective and constraint functions satisfy the following regularity condition.

      \begin{itemize}
      \item[\textbf{(F)}] The function $f_i$ is proper, closed and convex for each $i\in \set{I}_0$.
      \end{itemize}

      We now introduce the problem dual to~\eqref{eq:p-co}.

      \begin{samepage}
      \begin{empheq}[box=\fbox]{equation}\label{eq:d-co}
      \begin{array}{l@{\quad}l@{\qquad}l}
      \displaystyle \sup & \multicolumn{2}{l}{\displaystyle \mspace{-8mu} -f_0^* \left( \bm w_0 \right)- \sum_{i\in \set{I}}  \lambda_i f_i^* \left( \bm w_i /  \lambda_i \right)} \\
      \displaystyle \subj & \displaystyle \sum_{i\in \set{I}_0} \bm w_i = \bm 0 \\
      & \displaystyle \bm w_i \textnormal{ free}, \; i\in \set{I}_0, \;\;  \bm \lambda \ge \bm 0 \\
      \end{array}
      \tag{D}
      \end{empheq}
      $\mspace{320mu}$ \textbf{Dual Problem}
      \end{samepage}

      \noindent The dual problem~\eqref{eq:d-co} maximizes the (negative) infimal convolution of the conjugate objective function as well as the perspectives of the conjugate constraint functions. 

      \begin{thm}[Weak Duality] \label{prop:weak-duality}
      The infimum of~\eqref{eq:p-co} is larger or equal to the supremum of~\eqref{eq:d-co}.
      \end{thm}

      Theorem~\ref{prop:weak-duality} implies that~\eqref{eq:d-co} is necessarily infeasible whenever~\eqref{eq:p-co} is unbounded, and that~\eqref{eq:p-co} is necessarily infeasible whenever~\eqref{eq:d-co} is unbounded.

      \begin{thm}[Strong Duality] \label{prop:strong-duality}
      The following statements hold.
      \begin{enumerate}
      \item[(i)] If~\eqref{eq:p-co} or~\eqref{eq:d-co} admits a Slater point, then the infimum of~\eqref{eq:p-co} coincides with the supremum of~\eqref{eq:d-co}, and~\eqref{eq:d-co} or~\eqref{eq:p-co} is solvable, respectively.
      \item[(ii)] If the feasible region of~\eqref{eq:p-co} or~\eqref{eq:d-co} is nonempty and bounded, then the infimum of~\eqref{eq:p-co} coincides with the supremum of~\eqref{eq:d-co}, and~\eqref{eq:p-co} or~\eqref{eq:d-co} is solvable, respectively.
      \end{enumerate}
      \end{thm}
\rev{A discussion of the explicit convex duality theory presented here is provided in Appendix~\ref{app:om} in the Electronic Companion.}

      \section{Robust Convex Optimization} \label{sec:rco}

      Consider now the parametric optimization problem
      \begin{samepage}
      \begin{empheq}[box=\fbox]{equation} \label{eq:ps-ro}
      \begin{array}{l@{\quad}l@{\qquad}l}
      \inf &  f_0(\bm x, \bm z_0) \\
      \subj &   f_i  (\bm x, \bm z_i) \le 0  & \displaystyle \forall i\in \set{I}\\
      &  \bm x \text{ free}
      \end{array}
      \tag{P-S}
      \end{empheq}
      $\mspace{265mu}$ \textbf{Primal Scenario Problem}
      \end{samepage}
      \newline
      whose objective and constraint functions $f_i:\R^{d_{\bm x}} \times \R^{d_{\bm z}} \rightarrow \overline{\R}$ \rev{depend on} uncertain parameters $\bm z_i$, $i\in \set{I}_0$. \rev{As this optimization problem is parameterized by the joint scenario $(\bm z_0,  \dots, \bm z_I)$ of all uncertain parameters, we henceforth refer to it as the (primal) {\em scenario problem}. In the remainder of this section, we assume that~\eqref{eq:ps-ro} is convex. Even more, we assume that its objective and constraint functions display a saddle structure in the sense of the following regularity condition.}
      \begin{itemize}
      \item[\textbf{(RF)}] The function $f_i(\bm x, \bm z_i)$ is proper, closed and convex in  $\bm x$ for every fixed $\bm z_i$, and $-f_i(\bm x, \bm z_i)$ is proper, closed and convex in $\bm z_i$ for every fixed $\bm x$ across all $i\in\mathcal I_0$.
      \end{itemize}

      Assumption \textbf{(RF)}\rev{, which is a \underline{r}obust pendant of assumption~\textbf{(F)} from Section~\ref{sec:co},} implies that $f_i$ is real-valued for every $i \in \set I_0$. Indeed, as $f_i(\bm x, \bm z_i)$ is proper in $\bm x$ for every~$\bm z_i$, we have $f_i(\bm x, \bm z_i) > -\infty$ for every $\bm x$ and $\bm z_i$. Similarly, as  $- f_i(\bm x, \bm z_i) $ is proper in~$\bm z_i$ for every $\bm x$, we have $f_i(\bm x, \bm z_i) < +\infty$ for every $\bm x$ and $\bm z_i$.  In particular, the assumption \textbf{(RF)} thus implies that ${\rm ri}(\dom (f_i)) = {\rm ri}(\dom (-f_i))= \R^{d_{\bm x}} \times \R^{d_{\bm z}}$ for all $i \in \mathcal{I}_0$. As any convex function is continuous on the relative interior of its domain, this implies that $f_i(\bm x,\bm z_i)$ is continuous in $\bm x$ for all fixed $\bm z_i$ and continuous in $\bm z_i$ for all fixed $\bm x$.

\rev{In the remainder of the paper we will sometimes have to evaluate conjugates and perspectives of a bivariate function $f_i(\bm x, \bm z_i)$ with respect to only one of its two arguments. Specifically, the partial conjugates of $f_i$ with respect to its first argument~$\bm x$ and with respect to its second argument~$\bm z_i$ will henceforth be denoted by~$f_i^{*1}$ and~$f_i^{*2}$, respectively. Similarly, the partial perspectives will be denoted by~$tf_i(\bm x/t,\bm z_i)$ and~$tf_i(\bm x,\bm z_i/t)$, respectively. For details we refer to Definitions~\ref{def:partial-conjugate} and~\ref{def:partial-perspective} in Appendix~\ref{app_before_A} in the Electronic Companion.
}

      In analogy to Section~\ref{sec:co}, we \rev{can now construct} the problem dual to~\eqref{eq:ps-ro}.

      \begin{samepage}
      \begin{empheq}[box=\fbox]{equation*} \label{eq:ds-ro} \tag{D-S}
      \begin{array}{l@{\quad}l@{\qquad}l}
      \sup  &  \multicolumn{2}{l}{\displaystyle  \mspace{-8mu} - f_0^{*1} \left( \bm w_0, \bm z_0 \right) - \sum_{i\in\mathcal I}  \lambda_i f_i^{*1} (\bm w_i/\lambda_i, \bm z_i )} 	\\
      \subj &    \displaystyle \sum_{i\in \set{I}_0} \bm w_i = \bm 0 \\
      &  \bm w_i \textnormal{ free}, \; i\in \set{I}_0, \;\; \bm \lambda \ge \bm{0}
      \end{array}
      \end{empheq}
      $\mspace{275mu}$ \textbf{Dual Scenario Problem}
      \end{samepage}

      \noindent Since condition \textbf{(F)} from Section~\ref{sec:co} is satisfied, Theorem~\ref{prop:weak-duality} implies that the problems~\eqref{eq:ps-ro} and~\eqref{eq:ds-ro} satisfy a weak duality relationship. In addition, strong duality between~\eqref{eq:ps-ro} and~\eqref{eq:ds-ro} as well as primal and dual solvability hold under the relevant conditions of Theorem~\ref{prop:strong-duality}.

      The scenario problem~\eqref{eq:ps-ro} may have different \rev{minimizers} for different scenarios. Before the uncertainty is revealed, it is therefore unclear which of these \rev{minimizers} should be implemented. \rev{From now on} we assume that all uncertain parameters $\bm z_i$, $i\in \set{I}_0$, reside in the nonempty uncertainty~set
      \begin{equation}
      \label{eq:uncertainty-set-ro}
      \supp=\left\{ \bm z \in\R^{d_{\bm z}} \mid c_\ell(\bm z) \le 0 \;\; \forall \ell \in\mathcal L \right\}
      \end{equation}
      described by the constraint functions $c_\ell:\R^{d_{\bm z}}  \rightarrow  \overline{\R}$,  $\ell \in\mathcal L$. In the remainder of this section, we assume that these constraint functions obey the following regularity condition.

      \begin{itemize}
      \item[\textbf{(C)}] The function $c_\ell$ is proper, closed and convex for each $\ell \in\mathcal L$. 
      \end{itemize}
      Assumption~\textbf{(C)} immediately implies that the uncertainty set $\supp$ is closed.

      A popular approach to disambiguate the uncertain convex program~\eqref{eq:ps-ro} is to solve its robust counterpart, which seeks a decision that minimizes the worst-case objective across all $\bm z_0 \in \set{Z}$ and is feasible for all possible $\bm z_i \in \set{Z}$, $i\in \set{I}$. Formally, the robust counterpart can be expressed as follows.

      \begin{samepage}
      \begin{empheq}[box=\fbox]{equation*} \label{eq:pw-ro} \tag{P-W}
      \begin{array}{l@{\quad}l@{\qquad}l}
      \inf  &  \displaystyle \sup_{\bm z_0 \in \supp}f_0(\bm x, \bm z_0) \\
      \subj &  \displaystyle  \sup_{\bm z_i \in \supp}  f_i  (\bm x, \bm z_i) \le 0  & \displaystyle \forall i \in \set{I}\\
      &  \bm{x} \text{ free}
      \end{array}
      \end{empheq}
      $\mspace{275mu}$ \textbf{Primal Worst Problem}
      \end{samepage}

      Note that $\sup_{\bm z_i \in \supp}  f_i  (\bm x, \bm z_i) \le 0$ if and only if $f_i  (\bm x, \bm z_i) \le 0$ for all $\bm z_i \in \set{Z}$, $i\in \set{I}$. As it solves the uncertain primal problem~\eqref{eq:ps-ro} under the most pessimistic uncertainty realizations, the robust counterpart~\eqref{eq:pw-ro} is sometimes referred to as the {\em primal worst problem}.  Closely related to the primal worst is the {\it dual best problem}, which solves the dual scenario problem~\eqref{eq:ds-ro} under the most optimistic uncertainty realizations.

      \begin{samepage}
      \begin{empheq}[box=\fbox]{equation*} \label{eq:db-ro} \tag{D-B}
      \begin{array}{l@{\quad}l@{\qquad}l}
      \sup & \multicolumn{2}{l}{\displaystyle \mspace{-8mu} - f_0^{*1} \left(\bm w_0, \bm z_0 \right) -  \sum_{i\in\mathcal I}  \lambda_i f_i^{*1}   (\bm w_i/\lambda_i ,\bm z_i)} \\
      \subj  & \displaystyle   \sum_{i\in \set{I}_0} \bm w_i = \bm 0 \\
      &   \bm w_i \textnormal{ free}, \;\; \bm z_i \in \supp, \; i \in \set{I}_0, \;\;\bm \lambda \ge \bm 0
      \end{array}
      \end{empheq}
      $\mspace{290mu}$ \textbf{Dual Best Problem}
      \end{samepage}

      \noindent Note that in~\eqref{eq:db-ro} the uncertainty realizations $\bm z_i$, $i \in \set{I}_0$, are decision variables that can be chosen freely within the uncertainty set $\set Z$. As~\eqref{eq:db-ro} accommodates only finitely many decision variables and constraints, it is at least principally amenable to numerical solution. \rev{However,} \eqref{eq:db-ro} is generically non-convex as it involves partial perspectives of (jointly) convex functions\rev{; see Example~\ref{example:nonconvex} in Appendix~\ref{app:om} in the Electronic Companion}.

      The primal worst and dual best problems satisfy a weak duality relationship.

      \begin{thm}[Weak Duality for~\eqref{eq:pw-ro} and~\eqref{eq:db-ro}]
      \label{thm:primal-worst>dual-best-ro}
      The infimum of the primal worst problem~\eqref{eq:pw-ro} is larger or equal to the supremum of the dual best problem~\eqref{eq:db-ro}.
      \end{thm}

      As~\eqref{eq:pw-ro} involves embedded maximization problems and as~\eqref{eq:db-ro} is generally non-convex, both problems appear to be difficult to solve. In the following we will demonstrate, however, that under mild conditions both~\eqref{eq:pw-ro} and~\eqref{eq:db-ro} can be reformulated as polynomial-size convex programs~\eqref{eq:pw-cvx-ro} and~\eqref{eq:db-cvx-ro} that are amenable to solution with off-the-shelf solvers. While useful for computation, these reformulations will also allow us to prove strong duality between~\eqref{eq:pw-ro} and~\eqref{eq:db-ro}. To this end, we first summarize the relationships among the problems \eqref{eq:pw-ro},~\eqref{eq:pw-cvx-ro}, \eqref{eq:db-ro} and \eqref{eq:db-cvx-ro} in Figure~\ref{fig:weak-relations}.

	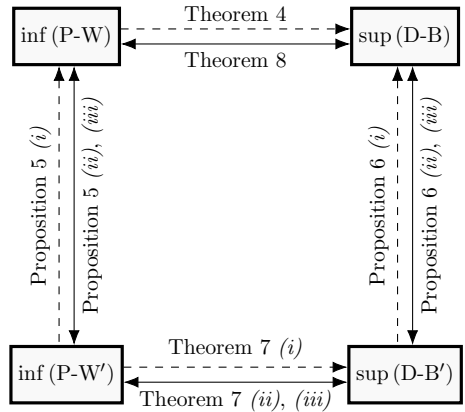
\begin{figure}[ht]
		\begin{center}
			\begin{tikzpicture}[scale=0.9, every node/.style={scale=0.7}, squarednode/.style={rectangle, draw=black, fill=gray!5, very thick, minimum width=15mm, minimum height=10mm}, node distance=3.05cm]
				\node[squarednode] (I1) {$\inf\eqref{eq:pw-ro}$ };
				\node (I2) [right of=I1] {};
				\node (I3) [right of=I2] {};
				\node[squarednode] (I4) [right of=I3] {$\sup \eqref{eq:db-ro}$};
				
				\node[squarednode] (C1) [below of=I1] {$\inf\eqref{eq:pw-cvx-ro}$};
				\node (C2) [right of=C1] {};
				\node (C3) [right of=C2] {};
				\node[squarednode] (C4) [below of=I4] {$\sup \eqref{eq:db-cvx-ro}$};


      \draw[-{Latex[length=2mm]},dashed] ($(I1.east)+(0mm,1mm)$) --  node[above]{\text{Theorem \ref{thm:primal-worst>dual-best-ro}}} ($(I4.west)+(0mm,1mm)$);
      \draw[{Latex[length=2mm]}-{Latex[length=2mm]}] ($(I1.east)+(0mm,-1mm)$) -- node[below]{\text{Theorem \ref{thm:p-w=d-b-ro}}} ($(I4.west)+(0mm,-1mm)$);

      \draw[{Latex[length=2mm]}-{Latex[length=2mm]}] ($(I1.south) +(1mm,0mm)$) -- node[right]{ \rotatebox{0}{\text{Proposition \ref{prop:p-w=p-w-cvx-ro}}~\emph{(ii)},~\emph{(iii)}}} ($(C1.north)+(1mm,0mm)$);
      \draw[{Latex[length=2mm]}-, dashed] ($(I1.south)+(-1mm,0mm)$) -- node[left]{ \rotatebox{0}{\text{Proposition \ref{prop:p-w=p-w-cvx-ro}~\emph{(i)}}}} ($(C1.north)+(-1mm,0mm)$);

      \draw[{Latex[length=2mm]}-, dashed] ($(I4.south)+(-1mm,0mm)$) -- node[left]{\rotatebox{0}{\text{Proposition \ref{prop:d-b=d-b-cvx-ro}~\emph{(i)}}}} ($(C4.north)+(-1mm,0mm)$);
      \draw[{Latex[length=2mm]}-{Latex[length=2mm]}] ($(I4.south)+(1mm,0mm)$) -- node[right]{\rotatebox{0}{\text{Proposition \ref{prop:d-b=d-b-cvx-ro}~\emph{(ii)},~\emph{(iii)}}}} ($(C4.north)+(1mm,0mm)$);

      \draw[{Latex[length=2mm]}-{Latex[length=2mm]}] ($(C1.east)+(0mm,-1mm)$) --  node[below]{\text{Theorem \ref{thm:p-w-cvx=d-b-cvx-ro}~\emph{(ii)},~\emph{(iii)}}} ($(C4.west)+(0mm,-1mm)$);
      \draw[-{Latex[length=2mm]},dashed] ($(C1.east)+(0mm,1mm)$) --   node[above]{\text{Theorem \ref{thm:p-w-cvx=d-b-cvx-ro}~\emph{(i)}}} ($(C4.west)+(0mm,1mm)$);

      \end{tikzpicture}
      \end{center}
      \vspace{-0.5cm}
      \caption{Illustration of the relationships among~\eqref{eq:pw-ro},~\eqref{eq:pw-cvx-ro}, \eqref{eq:db-ro} and \eqref{eq:db-cvx-ro}. The optimal values of these problems are non-increasing in the directions of the arcs. Dashed arcs represent universal inequalities, while solid arcs represent inequalities that hold under regularity conditions.}
      \label{fig:weak-relations}
      \end{figure}

      %

      We now show that~\eqref{eq:pw-ro} can be reduced to the following problem by dualizing the maximization problems that are embedded in~\eqref{eq:pw-ro}.

      \begin{samepage}
      \begin{empheq}[box=\fbox]{equation*}  \label{eq:pw-cvx-ro} \tag{P-W$'$}
      \begin{array}{l@{\quad}l@{\qquad}l}
      \inf  &   \displaystyle 	(- f_0)^{*2}\left(\bm x, \bm y_{00}  \right)  +   \sum_{\ell \in \set{L}}   \nu_{0\ell}c_\ell^*   \left(\bm y_{0\ell} /  \nu_{0\ell}\right)  \\
      \subj &   \displaystyle  (- f_i)^{*2}\left(\bm x,  \bm   y_{i0}  \right) +   \sum_{\ell \in \set{L}}  \nu_{i\ell} c_\ell^*  \left(\bm y_{i\ell} / \nu_{i\ell} \right) \le 0   & \displaystyle \forall i \in \set{I}\\
      & \displaystyle   \sum_{\ell \in \set{L}_0} \bm   y_{i\ell} = \bm 0 & \displaystyle \forall i \in \set{I}_0 \\
      & \multicolumn{2}{l}{\mspace{-10mu} \bm x \textnormal{ free}, \;\; \bm y_{i\ell} \textnormal{ free}, \; i \in \set{I}_0, \, \ell \in\mathcal L_0, \;\; \nu_{i \ell} \ge 0, \; i \in \set{I}_0, \, \ell \in \mathcal L}
      \end{array}
      \end{empheq}
      $\mspace{210mu}$ \textbf{Reformulated Primal Worst Problem}
      \end{samepage}

      Problem~\eqref{eq:pw-cvx-ro} is convex since its objective function and its constraints involve infimal convolutions of the conjugate objective and constraint functions $f_i$ as well as the perspectives of the conjugate constraint functions $c_\ell$ describing the uncertainty set (\emph{cf.}~Proposition~\ref{prop:partial-conjugate-dual}).

      %
      %
      %

      \begin{prop}[Convex Reformulation of~\eqref{eq:pw-ro}] \label{prop:p-w=p-w-cvx-ro}
      The following statements hold.
      \begin{enumerate}[label=(\roman*)]
      \item The infimum of~\eqref{eq:pw-ro} is smaller or equal to that of~\eqref{eq:pw-cvx-ro}.
      \item If $\supp$ admits a Slater point, then the infima of~\eqref{eq:pw-ro} and~\eqref{eq:pw-cvx-ro} coincide, and~\eqref{eq:pw-ro} is solvable if and only if~\eqref{eq:pw-cvx-ro} is solvable.
      \item If $\supp$ is compact and~\eqref{eq:pw-ro}  admits a strict Slater point, then the infima of~\eqref{eq:pw-ro} and~\eqref{eq:pw-cvx-ro} coincide, and~\eqref{eq:pw-ro} is solvable whenever~\eqref{eq:pw-cvx-ro} is solvable.
      \end{enumerate}
      \end{prop}

      Note that in Proposition~\ref{prop:p-w=p-w-cvx-ro}~\emph{(iii)}, the solvability of~\eqref{eq:pw-ro} does not imply the solvability of~\eqref{eq:pw-cvx-ro}. Nevertheless, one can construct a sequence of feasible solutions to~\eqref{eq:pw-cvx-ro} from a solution to~\eqref{eq:pw-ro} that asymptotically attain the same optimal value.

      Next, we argue that under mild conditions the non-convex dual best problem~\eqref{eq:db-ro} is equivalent to the following finite convex optimization problem.

      \begin{samepage}
      \begin{empheq}[box=\fbox]{equation*} \label{eq:db-cvx-ro} \tag{D-B$'$}
      \begin{array}{l@{\quad}l@{\qquad}l}
      \displaystyle \sup & \multicolumn{2}{l}{\displaystyle  \mspace{-8mu} - f_0^{*1}   \left( \bm w_0,\bm z_0 \right) - \sum_{i\in\mathcal I}  \lambda_i f_i^{*1}   (\bm w_i /\lambda_i  ,\bm \upsilon_i/\lambda_i  )} \\
      \subj &  \displaystyle \sum_{i\in \set{I}_0} \bm w_i = \bm 0 \\
      &  c_\ell  ( \bm z_0 ) \le 0 & \displaystyle \forall \ell \in\mathcal L \\
      &  \lambda_i c_\ell  ( \bm \upsilon_i/\lambda_i  ) \le 0  & \displaystyle \forall i \in \set{I}, \; \forall \ell \in\mathcal L  \\
      &  \multicolumn{2}{l}{\mspace{-10mu} \bm w_i \textnormal{ free}, \; i \in \set{I}_0, \;\; \bm z_0 \textnormal{ free}, \;\; \bm \lambda \ge \bm 0, \;\; \bm \upsilon_i \textnormal{ free}, \; i\in \set{I}}
      \end{array}
      \end{empheq}
      $\mspace{225mu}$ \textbf{Reformulated Dual Best Problem}
      \end{samepage}

      Intuitively, one can think of problem~\eqref{eq:db-cvx-ro} as being obtained from~\eqref{eq:db-ro} by multiplying the inequalities involving $\bm z_i$ by $\lambda_i$, $i\in \set I$ and via the variable substitution $\bm \upsilon_i \leftarrow \lambda_i \bm z_i$.  Since the perspective of a convex function is convex, the resulting model~\eqref{eq:db-cvx-ro} is indeed a convex problem.

      The equivalence of~\eqref{eq:db-ro} and~\eqref{eq:db-cvx-ro} was first established for
      robust linear programs with compact uncertainty sets \citep[Lemma~1]{gbbd14} and then generalized to robust nonlinear programs with compact uncertainty sets \citep[Theorem~1]{gd15}. 
      In the following, we relax the compactness condition and demonstrate that the nonlinear programs~\eqref{eq:db-ro} and~\eqref{eq:db-cvx-ro} remain equivalent if the uncertainty set admits a Slater point. This alternative result is useful for the analysis of distributionally robust optimization problems, which can often be reformulated as robust optimization problems with unbounded uncertainty sets (see Section~\ref{sec:drco}).



      \begin{prop}[Convex Reformulation of~\eqref{eq:db-ro}] \label{prop:d-b=d-b-cvx-ro}
      The following statements hold.
      \begin{enumerate}[label=(\roman*)]
      \item The supremum of~\eqref{eq:db-ro} is smaller or equal to that of~\eqref{eq:db-cvx-ro}.
      \item If~\eqref{eq:db-ro} admits a strict Slater point, then the suprema of~\eqref{eq:db-ro} and~\eqref{eq:db-cvx-ro} coincide, and~\eqref{eq:db-cvx-ro} is solvable whenever~\eqref{eq:db-ro} is solvable.
      \item If $\supp$ is bounded, then the suprema of~\eqref{eq:db-ro} and~\eqref{eq:db-cvx-ro} coincide, and~\eqref{eq:db-cvx-ro} is solvable if and only if~\eqref{eq:db-ro} is solvable.
      \end{enumerate}
      \end{prop}

      Note that in Proposition~\ref{prop:d-b=d-b-cvx-ro}~\emph{(ii)}, the solvability of~\eqref{eq:db-cvx-ro} does not imply the solvability of~\eqref{eq:db-ro}. Nevertheless, one can construct a sequence of feasible solutions to~\eqref{eq:db-ro} from a solution to~\eqref{eq:db-cvx-ro} that asymptotically attain the same optimal value.

      Having identified easily verifiable conditions under which~\eqref{eq:pw-ro} and~\eqref{eq:db-ro} are equivalent to their respective convex reformulations~\eqref{eq:pw-cvx-ro} and~\eqref{eq:db-cvx-ro}, we are now ready to prove that~\eqref{eq:pw-cvx-ro} and~\eqref{eq:db-cvx-ro} are dual to each other and thus enjoy various weak and strong duality relationships.

      \begin{thm}[Duality Results for~\eqref{eq:pw-cvx-ro} and~\eqref{eq:db-cvx-ro}]
      \label{thm:p-w-cvx=d-b-cvx-ro}
      The following statements hold.
      \begin{enumerate}[label=(\roman*)]
      \item The infimum of~\eqref{eq:pw-cvx-ro} is larger or equal to the supremum of~\eqref{eq:db-cvx-ro}.
      \item If~\eqref{eq:pw-cvx-ro} or \eqref{eq:db-cvx-ro} admits a Slater point, then the infimum of~\eqref{eq:pw-cvx-ro} coincides with the supremum of~\eqref{eq:db-cvx-ro}, and~\eqref{eq:db-cvx-ro} or~\eqref{eq:pw-cvx-ro} is solvable, respectively.
      \item If the feasible region of~\eqref{eq:pw-cvx-ro} or \eqref{eq:db-cvx-ro} is nonempty and bounded, then the infimum of~\eqref{eq:pw-cvx-ro} coincides with the supremum of~\eqref{eq:db-cvx-ro}, and \eqref{eq:pw-cvx-ro} or~\eqref{eq:db-cvx-ro} is solvable, respectively.
      \end{enumerate}
      \end{thm}

      We now demonstrate that the duality gap between~\eqref{eq:pw-ro} and~\eqref{eq:db-ro} vanishes provided that one out of two complementary regularity conditions holds. 
      
      \begin{thm}[Strong Duality for~\eqref{eq:pw-ro} and~\eqref{eq:db-ro}]
      The following statements hold. \label{thm:p-w=d-b-ro}
      \begin{enumerate}[label=(\roman*)]
      \item If~\eqref{eq:pw-ro} admits a strict Slater point and $\supp$ is bounded, then the infimum of~\eqref{eq:pw-ro} coincides with the supremum of~\eqref{eq:db-ro}, and~\eqref{eq:db-ro} is solvable.
      \item If the feasible region of~\eqref{eq:pw-ro} is nonempty and bounded and $\supp$ is bounded, then the infimum of~\eqref{eq:pw-ro} coincides with the supremum of~\eqref{eq:db-ro}, and~\eqref{eq:pw-ro} is solvable.
      \item If~\eqref{eq:db-ro} admits a strict Slater point, then its supremum coincides with the infimum of~\eqref{eq:pw-ro}, and~\eqref{eq:pw-ro} is solvable.
      \end{enumerate}
      \end{thm}

\rev{
\begin{ex}[Unbounded Uncertainty Sets]
    In contrast to earlier findings from the literature, our results in this section (such as Theorem~\ref{thm:p-w=d-b-ro}~\emph{(iii)}) do not require the uncertainty set $\mathcal{Z}$ to be bounded. Unbounded uncertainty sets commonly arise when a robust optimization problem involves nonlinear functions of the primitive uncertainties, such as demands that are modeled as functions of prices, the returns of derivative assets \citep{ZRK11:EJOR} or nonlinear decision rules \citep{GWK15:rules, BSZ19:adaptive}, and when these functions are linearized through liftings. While it is tempting to restrict an unbounded uncertainty set $\mathcal{Z}$ to a bounded subset $\mathcal{Z}' = \mathcal{Z} \cap [-B, B]^{d_{\bm{z}}}$ and subsequently apply existing results, this approach is plagued with practical challenges. Indeed, even in the benign case where $\mathcal{Z}$ is a polyhedron, verifying whether all vertices of $\mathcal{Z}$ are contained in $\mathcal{Z}'$---arguably a necessary but not sufficient condition for the validity of the revised uncertainty set $\mathcal{Z}'$---is not possible in polynomial time unless P = NP, see \citet[Corollary~1]{KLPS20:no_free_lunch}.
\end{ex}}

      \section{Uncertainty Quantification and Distributionally Robust Optimization} \label{sec:drco}

      An uncertainty quantification problem seeks a distribution that maximizes the expected value of a Borel measurable disutility function $g(\bmt z)$ over all probability distributions of the random vector~$\bmt z$ within a given set $\amb{P}$. As  uncertainty about the distribution of a random object is usually termed ambiguity, we henceforth refer to $\amb P$ as the {\em ambiguity set}.
      \begin{samepage}
      \begin{empheq}[box=\fbox]{align*}
      \tag{P-UQ}
      \label{eq:puq}
      \begin{array}{l}
      \displaystyle \sup_{\bP \in \amb{P}} \EP{g(\bmt{z})} \left[ g(\bmt{z}) \right]
      \end{array}
      \end{empheq}
      $\mspace{200mu}$   \textbf{Primal Uncertainty Quantification Problem}
      \end{samepage}

      \noindent To ensure that the expectation in~\eqref{eq:puq} is defined for all measurable disutility functions, we set $\EP{g(\bmt{z})} \left[ g(\bmt{z}) \right] = - \infty$ whenever the expectation of the positive and negative parts of $g(\bmt{z})$ are both infinite. This convention means that infeasibility dominates unboundedness. In the first part of this section, we assume that the disutility function $g:\R^{d_{\bm z}} \rightarrow \overline{\R}$ is decision-independent and representable as a pointwise maximum of $I \in \mathbb{N}$ component functions $g_i: \R^{d_{\bm z}} \rightarrow \overline{\R}$ for $i \in \set I$, that is, $g(\bm z) = \max_{i \in \set I } g_i( \bm z)$, which satisfy the following regularity condition.

      \begin{itemize}
      \item[\textbf{(G)}]  The function $-g_i$ is proper, closed and convex for each $i \in \mathcal I$.
      \end{itemize}

      Intuitively, the disutility function $g$ is thus a pointwise maximum of finitely many concave functions. Note that every piecewise affine continuous function can be represented in this way, and every continuous function on a compact set can be approximated arbitrarily closely by a piecewise affine continuous function. In addition, we assume that the ambiguity set $\amb{P}$ is nonempty and contains all distributions that satisfy $J \in \mathbb{N}$ moment conditions. Specifically, we assume that
      \begin{equation*}
      \amb{P}=\left\{ \bP \in \amb{P}_0(\set S) \; \left\vert \; \EP{\bm h(\bmt z)} \left[\bm h(\bmt z)\right] \leq \bm{\mu} \right. \right\},
      \end{equation*}
      where $\set S = \{ \bm z \in \R^{d_{\bm z}} \mid c_\ell (\bm z) \le 0 \;\; \forall \ell \in \set L \}$ is a nonempty support set of the same type as the uncertainty set $\set Z$ studied in Section~\ref{sec:rco} that satisfies assumption $\textbf{(C)}$, $\amb{P}_0(\set S)$ represents the family of all distributions supported on $\set S$, $\bm \mu \in \R^{J}$ is a vector of moment bounds, and the Borel measurable moment functions $h_j: \R^{d_{\bm z}} \rightarrow \overline{\R}$, $j \in \set J$, satisfy the following regularity condition.
      \begin{itemize}
      \item[\textbf{(H)}] The function $h_j$ is proper, closed and convex for each $j \in \set J$.
      \end{itemize}
      We set $\EP{h(\bmt{z})} \left[ h_j(\bmt{z}) \right] = \infty$ whenever the expectation of the positive and negative parts of $h_j(\bmt{z})$ are both infinite. This follows our convention that infeasibility dominates unboundedness, and it ensures that distributions under which the positive part of $h_j$ has an infinite expectation for some $j\in \set J$ are excluded from the ambiguity set and therefore infeasible in~\eqref{eq:puq}.

      Next, define
      \begin{align*}
      \bar{\set S}_i  = \left\{ \bm z \in \set S \mid  \bm z \in \dom (h_j) \;\; \forall j\in \set J, \; \bm z \in \dom(-g_i) \right\}
      \end{align*}
      for every $i\in\set I$. Note that $\bar{\set S}_i$ is convex by virtue of assumptions~\textbf{(G)} and~\textbf{(H)} but may fail to be closed. \rev{For example, if $\mathcal S=\mathbb R$,  $h_j(z) = 0$ for all $j\in\mathcal J$ and $-g_i(z) = 1/z$ for~$z>0$; $=+\infty$ for $z\leq 0$, then $\bar{\mathcal S}_i = (0, + \infty )$ is open.} Throughout this section, we impose the following regularity condition.
      \begin{itemize}
      \item[\textbf{(S)}] The set~$\bar{\set S}_i$ is nonempty for every $i \in \set I$.
      \end{itemize}
      Note that any distribution~$\bP$ feasible in~\eqref{eq:puq} must be supported on $\bar{\set S}=\cup_{i\in\set I}\bar{\set S}_i$. Indeed, $\bP$ must be supported on~${\rm dom} (h_j)$ for every $j\in\set J$ for otherwise~$\bP$ cannot satisfy the moment constraint~$\EP{\bm h(\bmt z)} \left[\bm h(\bmt z)\right] \leq \bm{\mu}$. Similarly, $\bP$ must be supported on~${\rm dom} (-g)=\cup_{i\in\set I} {\rm dom} (-g_i)$ for otherwise $\EP{g(\bmt{z})} [ g(\bmt{z})]=-\infty$. Thus, we will from now on refer to~$\bar{\set S}$ as the effective support set. Note that $\bar{\set S}$ is generically non-convex as it constitutes a finite union of convex sets, and it may fail to be closed.

      Assumption~\textbf{(S)} may be imposed without much loss of generality. To see this, note first that if $\bar{\set S}_i$ is empty for every $i\in\set I$, then the effective support set $\bar{\set S}$ is empty, and the uncertainty quantification problem~\eqref{eq:puq} is infeasible. We may thus assume that $\bar{\set S}$ is nonempty. In this case, if $\bar{\set S}_i=\emptyset$ for some $i\in\set I$, then $g_i(\bm z)=-\infty$ for all $\bm z\in \bar{\set S}$.
      This observation implies that the optimal value of~\eqref{eq:puq} does not change if we remove those components $g_i$ from $g$ for which $\bar{\set S}_i$ is empty.

      Note that~\eqref{eq:puq} constitutes a semi-infinite program with finitely many (moment) constraints and infinitely many decision variables because it optimizes over all probability distributions supported on the typically uncountable set $\set S$. The semi-infinite maximization problem~\eqref{eq:puq} admits a dual semi-infinite minimization problem~\eqref{eq:duq}, which involves only finitely many decision variables but infinitely many constraints parameterized by the elements of the effective support set~$\bar{\set S}$.

      \begin{samepage}
      \begin{empheq}[box=\fbox]{equation*}
      \tag{D-UQ}
      \label{eq:duq}
      \begin{array}{l@{\quad}l}
      \displaystyle \inf  &  \displaystyle \alpha  + \bm \mu^\top\bm \beta \\
      \subj & \displaystyle  \sup_{\bm z \in \bar{\set S}}  \left\{  g(\bm z) - \alpha - \bm h(\bm z)^\top \bm \beta  \right\}   \le 0 \\
      &  \alpha \text{ free}, \;\; \bm \beta\ge \bm 0
      \end{array}
      \end{empheq}
      $\mspace{190mu}$  \textbf{Dual Uncertainty Quantification Problem}
      \end{samepage}


      We first show that~\eqref{eq:puq} and~\eqref{eq:duq} satisfy a weak duality relationship.

      \begin{thm}[Weak Duality for~\eqref{eq:puq} and~\eqref{eq:duq}] \label{coro:cvxdro-ro-wd}
      The infimum of~\eqref{eq:duq} is larger or equal to the supremum of~\eqref{eq:puq}.
      \end{thm}

      Next, we consider the following restriction of problem~\eqref{eq:puq}.
      \begin{samepage}
      \begin{empheq}[box=\fbox]{equation*}  \tag{FR}
      \label{eq:fr}
      \begin{array}{l@{\quad}l@{\qquad}l}
      \sup  & \displaystyle \sum_{i \in \set{I}} \lambda_i g  (\bm z_i)  \\
      \subj & \displaystyle\sum_{i \in \set{I}} \lambda_i = 1, \qquad
      \sum_{i \in \set{I}} \lambda_i \bm{h} (\bm z_i) \le \bm \mu \\
      & \bm z_i \in \bar{\set S}, \; i \in \set{I}, \;\; \bm \lambda \ge \bm 0
      \end{array}
      \end{empheq}
      $\mspace{261mu}$  \textbf{Finite Reduction Problem}
      \end{samepage}

      \noindent Assumption \textbf{(G)} implies that $g(\bm z_i) < +\infty$ for all $\bm z_i \in \R^{d_{\bm z}}$, and assumption \textbf{(H)} implies that $h_j(\bm z_i) > -\infty$ for all $\bm z_i \in \R^{d_{\bm z}}$ and $j \in \set J$. The restriction $\bm z_i \in\bar{\set S}$, which further imposes that~$h_j (\bm{z}_i) < + \infty$ and $g (\bm{z}_i) > - \infty$, thus ensures that the products $\lambda_i g(\bm z_i)$ and $\lambda_i \bm{h} (\bm z_i)$ are well-defined even if $\lambda_i = 0$. Problem~\eqref{eq:fr} has intuitive appeal because it evaluates the worst-case expected disutility across all discrete $I$-point distributions $\bP \in \amb{P}$ with discretization points~$\bm{z}_i$ restricted to $\bar{\set S}$ and corresponding probabilities $\lambda_i$ for $i \in \set{I}$. Therefore, we henceforth refer to~\eqref{eq:fr} as a {\it finite reduction} of~\eqref{eq:puq}. As~\eqref{eq:fr} constitutes a restriction of~\eqref{eq:puq}, it provides a lower bound on~\eqref{eq:puq} and, by virtue of Theorem~\ref{coro:cvxdro-ro-wd}, on~\eqref{eq:duq}.


      We now show that \eqref{eq:duq} and \eqref{eq:fr} are instances of the primal worst and dual best robust optimization problems \eqref{eq:pw-ro} and~\eqref{eq:db-ro} studied in Section~\ref{sec:rco}, respectively, which we will call \emph{ambiguous} primal worst,  \eqref{eq:apw}, and \emph{ambiguous} dual best, \eqref{eq:adb}. While~\eqref{eq:apw} and~\eqref{eq:adb} appear to be difficult to solve, they again admit finite convex reformulations \eqref{eq:apw-cvx} and \eqref{eq:adb-cvx} that are instances of the problems \eqref{eq:pw-cvx-ro} and \eqref{eq:db-cvx-ro} in Section~\ref{sec:rco}, respectively. We can then use the results of Section~\ref{sec:rco} to derive conditions of strong duality between \eqref{eq:apw} and~\eqref{eq:adb}, which immediately imply equivalence between the uncertainty quantification problems \eqref{eq:puq} and \eqref{eq:duq} as well as the finite reduction \eqref{eq:fr}. These relationships are summarized in Figure~\ref{fig:weak-relations-dro}.
	
	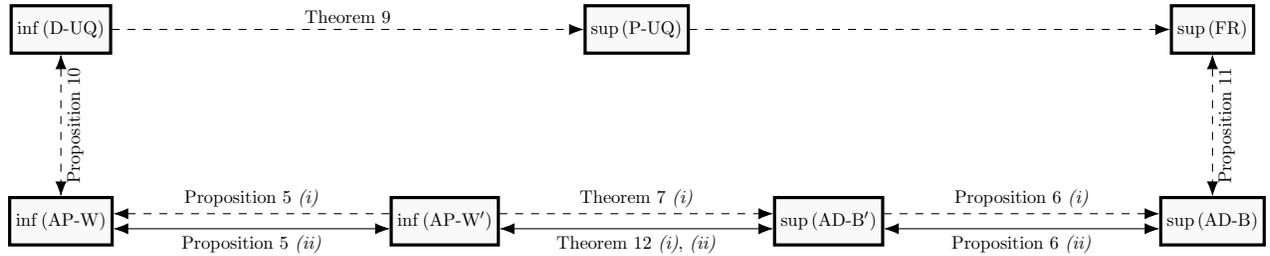
\begin{figure}[ht]
		\begin{center}
			\begin{tikzpicture}[scale=0.9, every node/.style={scale=0.7}, squarednode/.style={rectangle, draw=black, fill=gray!5, very thick, minimum width=15mm, minimum height=10mm}, node distance=3.05cm]

      \node[squarednode] (I1) {$\inf\eqref{eq:duq}$};
      \node (I2) [right of=I1] {};
      \node (I3) [right of=I2] {};
      \node[squarednode] (I4) [right of=I3] {$\sup\eqref{eq:puq}$};
      \node (I5) [right of=I4] {};
      \node (I6) [right of=I5] {};
      \node[squarednode] (I7) [right of=I6] {$\sup\eqref{eq:fr}$};

      \node[squarednode] (C1) [below of=I1] {$\inf\eqref{eq:apw}$ };
      \node[squarednode] (C2) [below of=I3] {$\inf\eqref{eq:apw-cvx}$};
      \node[squarednode] (C3) [below of=I5] {$\sup\eqref{eq:adb-cvx}$};
      \node[squarednode] (C4) [below of=I7] {$\sup\eqref{eq:adb}$};


      \draw[-{Latex[length=2mm]},dashed] ($(I1.east)+(0mm,0mm)$) --  node[above]{\text{Theorem~\ref{coro:cvxdro-ro-wd}}} ($(I4.west)+(0mm,0mm)$);

      \draw[-{Latex[length=2mm]},dashed] ($(I4.east)+(0mm,0mm)$) --  node[above]{\text{}} ($(I7.west)+(0mm,0mm)$);

      \draw[{Latex[length=2mm]}-{Latex[length=2mm]},dashed] ($(I1.south)+(0mm,0mm)$) -- node[right]{ \rotatebox{0}{Proposition~\ref{prop:d-uq=ap-w}}} ($(C1.north)+(0mm,0mm)$);

      \draw[{Latex[length=2mm]}-{Latex[length=2mm]},dashed] ($(I7.south)+(0mm,0mm)$) -- node[left]{ \rotatebox{0}{Proposition~\ref{prop:fr=ad-b}}} ($(C4.north)+(0mm,0mm)$);

      \draw[{Latex[length=2mm]}-{Latex[length=2mm]}] ($(C1.east)+(0mm,-1mm)$) --  node[below]{\text{Proposition~\ref{prop:p-w=p-w-cvx-ro}~\emph{(ii)}}} ($(C2.west)+(0mm,-1mm)$);
      \draw[{Latex[length=2mm]}-,dashed] ($(C1.east)+(0mm,1mm)$) --   node[above]{\text{Proposition~\ref{prop:p-w=p-w-cvx-ro}~\emph{(i)}}} ($(C2.west)+(0mm,1mm)$);

      \draw[{Latex[length=2mm]}-{Latex[length=2mm]}] ($(C2.east)+(0mm,-1mm)$) --  node[below]{\text{Theorem~\ref{thm:strong-duality-all-problems}~\emph{(i)},~\emph{(ii)}}} ($(C3.west)+(0mm,-1mm)$);
      \draw[-{Latex[length=2mm]}, dashed] ($(C2.east)+(0mm,+1mm)$) --  node[above]{\text{Theorem~\ref{thm:p-w-cvx=d-b-cvx-ro}~\emph{(i)}}} ($(C3.west)+(0mm,+1mm)$);

      \draw[{Latex[length=2mm]}-{Latex[length=2mm]}] ($(C3.east)+(0mm,-1mm)$) --  node[below]{\text{Proposition~\ref{prop:d-b=d-b-cvx-ro}~\emph{(ii)}}} ($(C4.west)+(0mm,-1mm)$);
      \draw[-{Latex[length=2mm]}, dashed] ($(C3.east)+(0mm,+1mm)$) --  node[above]{\text{Proposition~\ref{prop:d-b=d-b-cvx-ro}~\emph{(i)}}} ($(C4.west)+(0mm,+1mm)$);
      \end{tikzpicture}
      \end{center}
      \vspace{-0.5cm}
      \caption{Illustration of the relationships among~\eqref{eq:puq},~\eqref{eq:duq},~\eqref{eq:fr},~\eqref{eq:apw}, \eqref{eq:adb},~\eqref{eq:apw-cvx} and~\eqref{eq:adb-cvx}. The optimal values of these problems are non-increasing in the directions of the arcs. Dashed arcs represent universal inequalities, while the solid arcs represent inequalities that hold under regularity conditions.
      }
      \label{fig:weak-relations-dro}

      \end{figure}

      Although the dual uncertainty quantification problem \eqref{eq:duq} appears to be an instance of the primal worst problem \eqref{eq:pw-ro}, the `uncertainty set' $\bar{\set S}$ fails to satisfy the assumption~\textbf{(C)} from Section~\ref{sec:rco} for two reasons. Firstly, $\bar{\set S}$ is generally non-convex as it constitutes a union of~$I$ convex sets corresponding to the $I$ concave pieces of the disutility function $g$. Secondly, the domains of $h_j$ and $-g$ may not be closed, in which case $\bar{\set S}$ cannot be represented as an instance of $\mathcal{Z}$ that satisfies assumption \textbf{(C)}, which is closed by construction. In addition, the constraint function in~\eqref{eq:duq} violates the assumption \textbf{(RF)} since $g$ is not concave in $\bm{z}$. \rev{Hence, \eqref{eq:duq} fails to be an instance of~\eqref{eq:pw-ro}. Nevertheless, \eqref{eq:duq} admits an equivalent reformulation that {\em is} an instance of~\eqref{eq:pw-ro}. To see this, we introduce} separate augmented support sets $\set U_0=  \R^{d_{\bm z}} \times \R^{J} \times \R$ and
      \begin{align*}
      \set U_i  &= \{ (\bm z, \bm u, t) \in \R^{d_{\bm z}} \times \R^{J} \times \R  \mid c_\ell (\bm z) \le 0 \;\; \forall \ell \in \set L, \; \bm h(\bm z) \le  \bm u, \; g_i(\bm z) \ge t \} \quad \forall i \in \set I.
      \end{align*}
      We can then define the ambiguous primal worst problem as follows.
      \begin{samepage}
      \begin{empheq}[box=\fbox]{equation*}   \tag{AP-W}
      \label{eq:apw}
      \begin{array}{l@{\quad}l@{\qquad}l}
      \inf & \displaystyle \alpha  + \bm \mu^\top\bm \beta \\
      \subj & \displaystyle \sup_{ (\bm z_i, \bm u_i, t_i) \in  \set U_i}  \left\{  t_i - \alpha - \bm u_i^\top \bm \beta  \right\} \le 0   &  \forall i \in \set{I} \\
      & \alpha \text{ free},  \;\; \bm \beta\ge \bm 0
      \end{array}
      \end{empheq}
      $\mspace{220mu}$  \textbf{Ambiguous Primal Worst Problem}
      \end{samepage}

      The ambiguous primal worst problem~\eqref{eq:apw} \rev{can indeed be interpreted} as an instance of the primal worst problem~\eqref{eq:pw-ro} with the objective function
      \begin{subequations}  \label{eq:obj-con}
      \begin{equation}
      f_0( (\alpha, \bm \beta),(\bm z_0, \bm u_0, t_0)) \; = \; \alpha + \bm \mu^\top \bm \beta + \delta_{\R_+^J}(\bm \beta)
      \label{eq:obj}
      \end{equation}
      and the constraint functions
      \begin{equation}
      f_i( (\alpha, \bm \beta),(\bm z_i, \bm u_i, t_i)) \; =
      \; t_i - \alpha - \bm u_i^\top \bm \beta \label{eq:con}
      \end{equation}
      \end{subequations}
      for $i \in \set I$, where the optimization variables~$(\alpha, \bm \beta)$ correspond to~$\bm x$, and the uncertain parameters~$(\bm z_i, \bm u_i, t_i) \in \set U_i$ correspond to $\bm z_i \in \set Z$ for every~$i \in \set I_0$. Note that the objective function~\eqref{eq:obj} of~\eqref{eq:apw} is proper, convex and closed in the decision variables and constant in the uncertain parameters, while the constraint functions~\eqref{eq:con} of~\eqref{eq:apw} are bi-affine in the decision variables and the uncertain parameters. Thus, problem~\eqref{eq:apw} satisfies assumption\footnote{Strictly speaking, $f_0$ does not satisfy assumption~\textbf{(RF)} because $-f_0$ is not proper in the uncertain parameters for~$\bm \beta\not\geq \bm 0$. Since~$f_0$ is deterministic, however, maximizing~$f_0$ over the uncertain parameters will not yield a pathological worst-case objective function. The lack of properness of~$-f_0$ in this special case is therefore unproblematic. Details are omitted for brevity.} \textbf{(RF)}. In addition, as the support set $\set S$ satisfies assumption \textbf{(C)} and as the disutility function $g_i$ and the moment functions $h_j$, $j\in \set J$, satisfy the assumptions \textbf{(G)} and \textbf{(H)}, respectively, the functions defining the augmented support set $\set U_i$ all satisfy assumption~\textbf{(C)} for every $i \in \set I$. Although the effective support set $\bar{\set S}$ in~\eqref{eq:duq} and~\eqref{eq:puq} fails to satisfy assumption~\textbf{(C)}, its constituent sets $\bar{\set S}_i$, which are nonempty by assumption~\textbf{(S)}, are the projections of $\set U_i$ onto~$\R^{d_{\bm z}}$ for all $i \in \set I$:
      \begin{equation} \label{eq:Z and U_i}
      \bm z_i \in \bar{\set S}_i
      \quad \iff \quad
      \exists \bm{u}_i \in \mathbb{R}^J \text{ and } t_i \in \mathbb{R} \text{ such that } (\bm{z}_i, \bm{u}_i, t_i) \in \mathcal{U}_i.
      \end{equation}
      Indeed, $\bm{z}_i \in \bar{\set S}_i$ implies that $(\bm{z}_i, \bm{h} (\bm{z}_i), g (\bm{z}_i)) \in \mathcal{U}_i$, and the reverse implication holds since the condition on the right-hand side implies that $h_j (\bm{z}_i) < + \infty$ for all $j \in \mathcal{J}$ as well as $g_i (\bm{z}_i) > - \infty$.

      Using the equivalence~\eqref{eq:Z and U_i}, we can now show that problem~\eqref{eq:duq} is equivalent to~\eqref{eq:apw}.

      \begin{prop}[Equivalence of~\eqref{eq:duq} and~\eqref{eq:apw}] \label{prop:d-uq=ap-w}
      The infima of~\eqref{eq:duq} and~\eqref{eq:apw} coincide, and~\eqref{eq:duq} is solvable if and only if~\eqref{eq:apw} is solvable.
      \end{prop}

      %
      %
      \noindent	The partial conjugates of the objective and constraint functions~\eqref{eq:obj-con} with respect to the decision variables ($\alpha, \bm \beta$) are given by
      \begin{align*}
      f_0^{*1}( (v_{0}, \bm w_0),(\bm z_0, \bm u_0, t_0))   = \sup_{\alpha, \bm \beta \in \R_+^J} \left\{ v_{0} \alpha + \bm w_0^\top \bm \beta - \alpha - \bm \mu^\top \bm \beta  \right\} =
      \begin{cases}
      0 & \text{if } v_{0} = 1, \bm w_0 \le \bm \mu \\
      +\infty & \text{otherwise}
      \end{cases}
      \end{align*}
      and
      \begin{align*}
      f_i^{*1} \left( (v_{i}, \bm w_i),(\bm z_i, \bm u_i, t_i)\right)  = \sup_{\alpha, \bm \beta} \left\{ v_{i} \alpha + \bm w_i^\top \bm \beta - t_i + \alpha + \bm u_i^\top \bm \beta  \right\}  = \begin{cases}
      - t_i & \text{if } v_{i} = - 1, \bm w_i = - \bm u_i \\
      +\infty & \text{otherwise}
      \end{cases}
      \end{align*}
      for every $i \in \set I$. Substituting these expressions into~\eqref{eq:db-ro} and eliminating the redundant decision variables $v_i$ and $\bm w_i$, $i \in \mathcal{I}_0$, yields the following optimization problem, which we will henceforth refer to as the ambiguous dual best problem.
      \begin{samepage}
      \begin{empheq}[box=\fbox]{equation*}  \tag{AD-B}
      \label{eq:adb}
      \begin{array}{l@{\quad}l@{\qquad}l}
      \sup  & \displaystyle \sum_{i \in \set{I}} \lambda_i t_i     \\
      \subj &  \displaystyle\sum_{ i \in \set{I}} \lambda_i = 1, \qquad
      \sum_{i \in \set{I}} \lambda_i  \bm u_i \le \bm \mu  \\
      & (\bm z_i, \bm u_i, t_i) \in \set U_i, \; i \in \set{I}, \;\; \bm \lambda \ge \bm 0  
      \end{array}
      \end{empheq}
      $\mspace{240mu}$  \textbf{Ambiguous Dual Best Problem}
      \end{samepage}





      In the following, we show that~\eqref{eq:adb} is equivalent to~\eqref{eq:fr}.
      \begin{prop}[Equivalence of~\eqref{eq:fr} and~\eqref{eq:adb}] \label{prop:fr=ad-b}
      The suprema of~\eqref{eq:fr} and~\eqref{eq:adb} coincide, and~\eqref{eq:fr} is solvable if and only if~\eqref{eq:adb} is solvable.
      \end{prop}

      In Section~\ref{sec:rco} we have seen that the seemingly intractable primal worst and dual best optimization problems~\eqref{eq:pw-ro} and~\eqref{eq:db-ro} admit exact reformulations as the finite convex programs~\eqref{eq:pw-cvx-ro} and~\eqref{eq:db-cvx-ro}, respectively. If we interpret~\eqref{eq:apw} as an instance of~\eqref{eq:pw-ro} as explained above, then the corresponding instance of~\eqref{eq:pw-cvx-ro} can be constructed as follows. First, we evaluate the partial conjugates of the objective and constraint functions~\eqref{eq:obj-con} with respect to the uncertain parameters and evaluate the conjugates of the constraint functions defining the uncertain sets. Substituting these (partial and global) conjugates into~\eqref{eq:pw-cvx-ro} and eliminating all the redundant variables with fixed values  then yields the following convex program.
      \begin{samepage}
      \begin{empheq}[box=\fbox]{equation*} \tag{AP-W$'$}
      \label{eq:apw-cvx}
      \begin{array}{l@{\quad}l}
      \inf  & \displaystyle \alpha  + \bm \mu^\top\bm \beta \\
      \subj &  \displaystyle (-g_i)^{*}\left(\bm y^{(0)}_{i}  \right) + \sum_{j\in \set{J}}  \beta_j  h_j^* \left( \frac{\bm y^{(1)}_{ij}} { \beta_j} \right) + \sum_{\ell \in \set{L}}  \nu_{i\ell} c_\ell^*    \left( \frac{\bm y^{(2)}_{i\ell} }{ \nu_{i\ell}} \right) \le \alpha   \quad \forall i \in \set{I}  \\
      & \displaystyle \bm y^{(0)}_{i} + \sum_{j\in \set{J} }  \bm y^{(1)}_{ij}   + \sum_{\ell \in \set{L}} \bm   y^{(2)}_{i\ell} = \bm 0  \qquad \qquad \qquad \qquad \qquad \quad \quad \ \ \ \forall i \in \set{I}   \\
      & \alpha \textnormal{ free}, \;\; \bm \beta \ge \bm 0, \;\; \bm y^{(0)}_{i}, \bm y^{(1)}_{ij}, \bm y^{(2)}_{i\ell} \textnormal{ free}, \;\; \nu_{i\ell} \ge 0, \; i \in \set{I}, \, j\in \set{J}, \, \ell \in \set{L}
      \end{array}
      \end{empheq}
      $\mspace{125mu}$  \textbf{Reformulated Ambiguous Primal Worst Problem}
      \end{samepage}

      The derivation of~\eqref{eq:apw-cvx} is tedious but completely mechanical and requires no new ideas. Details are omitted for brevity. Similarly, substituting the objective and constraint functions~\eqref{eq:obj-con} as well as the constraint functions defining $\set U_i$ into \eqref{eq:db-cvx-ro} yields the following convex program.
      \begin{empheq}[box=\fbox]{equation*} \tag{AD-B$'$} \label{eq:adb-cvx}
      \begin{array}{l@{\quad}l@{\qquad}l}
      \sup  & \displaystyle \sum_{ i \in \set{I}}  \tau_i  \\
      \subj & \displaystyle\sum_{i \in \set{I}} \lambda_i = 1, \qquad
      \sum_{i \in \set{I}}  \bm \omega_i \le \bm \mu\\
      & \lambda_i c_\ell (\bm \upsilon_i / \lambda_i) \le 0 &  \forall i\in\set{I}, \ \forall \ell \in \set{L}\\
      & \lambda_i \bm{h}  (\bm \upsilon_i/ \lambda_i) \le \bm \omega_i,  \quad \lambda_i g_i (\bm \upsilon_i/\lambda_i ) \ge \tau_i & \forall i \in \set{I}\\
      &  \bm \tau \text{ free}, \;\; \bm \lambda \ge \bm 0, \;\; \bm \omega_i, \bm \upsilon_i \text{ free}, \; i \in \set{I}\\
      \end{array}
      \end{empheq}
      $\mspace{175mu}$  \textbf{Reformulated Ambiguous Dual Best Problem}

      \noindent As  \eqref{eq:apw}, \eqref{eq:adb}, \eqref{eq:apw-cvx} and~\eqref{eq:adb-cvx} represent instances of ~\eqref{eq:pw-ro},~\eqref{eq:db-ro},~\eqref{eq:pw-cvx-ro} and~\eqref{eq:db-cvx-ro}, respectively, all results of Section~\ref{sec:rco} are applicable and offer conditions under which these problems share the same optimal values or solvability characteristics.

      We now establish minimal conditions for strong duality between the finite convex programs \eqref{eq:apw-cvx} and~\eqref{eq:adb-cvx}. These conditions will also be sufficient for strong duality between the semi-infinite programs~\eqref{eq:puq} and~\eqref{eq:duq}. We emphasize that our duality results follow from first principles of convex analysis and do not rely on the elaborate machinery of abstract semi-infinite duality theory such as \cite{Isii62}, \cite{an87} and \cite{s01}.

      \begin{thm}[Strong Duality for \eqref{eq:puq} and \eqref{eq:duq}]
      \label{thm:strong-duality-all-problems}
      The following statements hold.
      \begin{enumerate}[label=(\roman*)]
      \item If~\eqref{eq:adb-cvx} admits a Slater point with $\bm \lambda > \bm 0$,
      then~\eqref{eq:apw-cvx},~\eqref{eq:apw}, \eqref{eq:duq}, \eqref{eq:puq}, \eqref{eq:fr}, \eqref{eq:adb} and~\eqref{eq:adb-cvx} all have the same optimal value and \eqref{eq:apw-cvx} is solvable. Also, if $(\alpha^\star, \bm \beta^\star, \{\bm y_i^{(0)\star}\}_i, \{\bm y_{ij}^{(1)\star}\}_{ij},  \{ \bm y_{i\ell}^{(2)\star}, \nu_{i\ell}^\star \}_{i\ell})$ solves~\eqref{eq:apw-cvx}, then $(\alpha^\star, \bm \beta^\star)$ solves~\eqref{eq:duq}.
      \item If~\eqref{eq:apw-cvx}
      admits a Slater point and $\set S$ is bounded, then \eqref{eq:apw-cvx}, \eqref{eq:apw}, \eqref{eq:duq}, \eqref{eq:puq}, \eqref{eq:fr}, \eqref{eq:adb} and \eqref{eq:adb-cvx} all have the same optimal value and \eqref{eq:adb-cvx} is solvable. Also, if $(\bm \tau^\star, \bm \lambda^\star, \{ \bm \omega^\star_i, \bm v^\star_i\}_i)$ solves~\eqref{eq:adb-cvx}, then the discrete distribution that assigns probability $\lambda^\star_i$ to the point $\bm v_i^\star/\lambda_i^\star$ for every $i\in\set I$ with $\lambda_i^\star > 0$ solves~\eqref{eq:puq}.
      \end{enumerate}
      \end{thm}

      The equivalence of~\eqref{eq:apw-cvx},~\eqref{eq:apw}, \eqref{eq:duq}, \eqref{eq:puq}, \eqref{eq:fr}, \eqref{eq:adb} and~\eqref{eq:adb-cvx} can also be shown if~\eqref{eq:adb-cvx} admits a Slater point and $\set S$ is bounded or if~\eqref{eq:apw-cvx} admits a Slater point and~\eqref{eq:adb-cvx} has a feasible solution with $\bm \lambda > \bm 0$. However, these cases are less relevant in practice. 

      \begin{rem}[Relation to Semi-Infinite Duality Theory]
      Strong duality between the primal and dual uncertainty quantification problems~\eqref{eq:puq} and~\eqref{eq:duq} can also be established by appealing to the classical duality theory for generalized moment problems. In order to describe the sufficient condition that is most frequently used, we denote by $\mathcal M_+(\bar{\set S})$ the cone of all non-negative Borel measures supported on~$\bar{\set S}$ under which the functions $g$ and $\bm h$ are integrable, and we define
      \begin{equation*}
      \set C=\left\{ \int_{\bar{\set S}} (1, \bm h(\bm z)) \,\lambda({\rm d}\bm z) \; \Big| \; \lambda \in \mathcal M_+(\bar{\set S}) \right\} + \left( \{0\}\times \R^J_+ \right),
      \end{equation*}
      which constitutes a Minkowski sum of two convex cones and is thus itself a convex cone. By virtue of Proposition~3.4 by~\citet{s01}, the supremum of~\eqref{eq:puq} coincides with the infimum of~\eqref{eq:duq} if the vector $(1,\bm\mu)$ resides in the interior of~$\set C$. This condition is more general because it extends to arbitrary measurable functions $g$ and $\bm h$, but it is not always easy to check. Theorem~\ref{thm:strong-duality-all-problems} holds under more restrictive conditions as it relies on the convexity properties of the functions~$g_i$, $\bm h$ and $c_\ell$, but the existence of a Slater point is usually straightforward to verify by inspection. In addition, Slater points for the finite convex programs~\eqref{eq:apw-cvx} and~\eqref{eq:adb-cvx} can also be found numerically by solving suitable auxiliary convex optimization problems. Finally and most importantly, the conditions of Theorem~\ref{thm:strong-duality-all-problems} not only ensure strong duality between the semi-infinite optimization problems~\eqref{eq:puq} and~\eqref{eq:duq} but also guarantee that these semi-infinite optimization problems are equivalent to finite convex programs. In contrast, the standard approach to distributionally robust optimization imposes separate regularity conditions to ensure strong duality between the semi-infinite optimization problems and to ensure that these problems admit finite convex reformulations.
      \end{rem}

      The following two propositions provide sufficient conditions for the assumptions of Theorem~\ref{thm:strong-duality-all-problems} that may be easier to interpret. The first such condition relies on the notion of a Slater distribution.

      \begin{defi}[Slater Distribution]
      \label{def:Slater-distribution}
      The distribution~$\bP^{\rm S}\in \amb{P}_0(\bar{\set S})$ is a Slater distribution for~the uncertainty quantification problem~\eqref{eq:puq} if (i)~$\bP^{\rm S}$ is absolutely continuous on~$\R^{d_{\bm z}}$, (ii) $\bP^{\rm S}\left[ \bmt z\in\bar{\set S}_i\right]>0$ for all~$i\in\set I$, (iii)~$\mathbb{E}_{\bP^{\rm S}}[h_j(\bmt z)]\leq\mu_j$ for all~$j\in\set J$, with the inequality being strict if $h_j$ is nonlinear, and (iv)~$\mathbb{E}_{\bP^{\rm S}}[c_\ell(\bmt z)] < 0$ for all~$\ell\in\set L$ where $c_\ell$ is nonlinear.
      \end{defi}

      \begin{prop}[Slater Points for~\eqref{eq:adb-cvx}]
      \label{prop:slater-distribution}
      If the uncertainty quantification problem~\eqref{eq:puq} admits a Slater distribution, then the convex program~\eqref{eq:adb-cvx} admits a Slater point with $\bm \lambda > \bm 0$.
      \end{prop}

      \begin{prop}[Slater Points for~\eqref{eq:apw-cvx}]
      \label{prop:slater-distribution2}
      If the semi-infinite program~\eqref{eq:duq} admits a strict Slater point and $\set S$ is bounded, then the convex program~\eqref{eq:apw-cvx} admits a strict Slater point.
      \end{prop}

      Among all sufficient conditions for strong duality between the uncertainty quantification problems~\eqref{eq:puq} and~\eqref{eq:duq}, the assumptions of Theorem~\ref{thm:strong-duality-all-problems}~\emph{(i)} are---in our experience---most frequently satisfied, but they do not guarantee the solvability of~\eqref{eq:puq}. The following corollary shows, however, that~\eqref{eq:puq} can still be solved asymptotically under these assumptions.

      \begin{coro}[Approximate Numerical Solution of~\eqref{eq:puq}] \label{coro:puq-asy}
      If~\eqref{eq:adb-cvx} admits a Slater point with $\bm \lambda >\bm 0$, then the suprema of~\eqref{eq:puq} and~\eqref{eq:adb-cvx} coincide, and any $\epsilon$-optimal solution of~\eqref{eq:adb-cvx} can be used to construct a discrete distribution with $I$ atoms that is $2\epsilon$-optimal in~\eqref{eq:puq}.
      \end{coro}

      The main results of this section can be directly applied to distributionally robust optimization problems, in which one seeks a decision~$\bm x$ from within a closed feasible region $\set X\subseteq \R^{d_{\bm x}}$ that minimizes the worst-case expectation of a decision-dependent disutility function $g(\bm x,\bm z)$ with respect to all distributions $\bP\in\amb P$. In this case, the results of this section readily imply that~\eqref{eq:apw-cvx} remains a finite convex program when~$\bm x$ is appended to the list of optimization variables, provided that~$\set X$ is convex and that the disutility function satisfies~$g(\bm x,\bm z) = \max_{i \in \set I} g_i (\bm x, \bm z)$, where $g_i (\bm x, \bm z)$ is proper, convex and closed in~$\bm x$ and~$-g_i (\bm x, \bm z)$ is proper, convex and closed in~$\bm z$ for every fixed~$i\in\set I$. 

      \section{Extensions}
      \label{sec:extensions}
      The results of Section~\ref{sec:drco} remain valid if the support set $\set S$ is representable as a finite union of convex component sets and if each component set as well as the moment constraints in the ambiguity set~$\amb P$ are defined in terms of conic inequalities. In order to formally describe these generalizations, we first recall some further concepts and terminology from convex optimization.

      Any proper convex cone $\set C\subseteq \R^{d_{\set C}}$ \rev{ (\emph{cf.}~Defnition~\ref{def:proper-convex-cone} in~Appendix~\ref{app_before_A} in the Electric Companion)} induces weak as well as strict generalized inequalities on~$\R^{d_{\set C}}$. Specifically, for any $\bm y,\bm y'\in\R^{d_{\set C}}$, the relation~$\bm y\preceq_\set C \bm y'$ means that $\bm y'-\bm y\in\set C$, while the relation~$\bm y\prec_\set C \bm y'$ means that $\bm y'-\bm y \in{\rm int}(\set C)$. The reverse inequalities $\succeq_\set C$ and $\succ_\set C$ are defined analogously. In the following we attach to~$\R^{d_{\set C}}$ a largest element~$+\bm \infty_{\set C}$ and a smallest element~$-\bm \infty_{\set C}$ with respect to the partial ordering~$\preceq_{\set C}$, that is, we assume that~$-\bm \infty_{\set C}\preceq_{\set C} \bm y\preceq_{\set C} +\bm \infty_{\set C}$ for all~$\bm y\in \R^{d_{\set C}}$. All $d_{\set C}$-dimensional functions considered in the remainder are valued in~$\overline \R{}^{d_{\set C}}=\R^{d_{\set C}}\cup\{-\bm\infty_{\set C}, +\bm\infty_{\set C}\}$.
      The {\em domain} of a function~$\bm f: \R^{d_{\bm x}}\rightarrow \overline\R{}^{d_{\set C}}$
      is defined as~$\dom(\bm f)=\{\bm x\in \R^{d_{\bm x}} \, | \, \bm f(\bm x) \prec_{\set C} +\bm\infty_{\set C}\}$, and~$\bm f$ is {\em proper} if~$\bm f(\bm x)\succ_{\set C}-\bm\infty_{\set C}$ for all~$\bm x\in\R^{d_{\bm x}}$ and~$\bm f(\bm x)\prec+\bm\infty_{\set C}$ for at least one~$\bm x\in\R^{d_{\bm x}}$.

  \rev{ The definitions of Slater points for optimization problems involving only classical constraints can now be generalized to optimization problems involving conic constraints by replacing the weak and strict inequalities of Definition~\ref{def:Slater_sets} with $\preceq_{\set C_i}$ and $\prec_{\set C_i}$, respectively, where $\set C_i$ is a proper convex cone for all $i\in \set I$ (\emph{cf.}~Definition~\ref{def:slater-condition-for-sets} in Appendix~\ref{app_before_A} in the Electric Companion). In analogy to Definition~\ref{def:Slater_problems}, a vector $\bm x^{\rm S}$ is a (strict) Slater point of a minimization problem if it is a (strict) Slater point of the problem's feasible region and resides in the relative interior of the domain of the problem's objective function.} 


      \begin{defi}[$\set C$-Convex Function]
      If $\set C\subseteq \R^{d_{\set C}}$ is a proper convex cone, then $\bm f:\mathbb R^{d_{\bm x}} \rightarrow \overline\R{}^{d_{\set C}}$ is called $\set C$-convex if ${\rm dom}(\bm f)$ is a convex set and
      $\bm f(\theta \bm x + (1 - \theta) \bm x') \preceq_{\set C} \theta \bm f( \bm x ) +  (1 - \theta)  \bm f( \bm x')$ for all~$\bm x,\bm x'\in {\rm dom(}\bm f)$ and $\theta \in [0,1]$.
      \end{defi}

      Note that $\bm f$ is $\set C$-convex if and only if its $\set C$-epigraph $\textnormal{epi}_{\set C} (\bm f) = \{ (\bm x,\bm y) \in \R^{d_{\bm x}} \times \R{}^{d_{\set C}} \ | \ \bm f(\bm x ) \preceq_{\set C} \bm y \}$ is convex; see \citet[Exercise~3.20]{bn13}. The cone dual to a proper convex cone $\set C\subseteq \R^{d_{\set C}}$ is defined as $\set C^*=\{\bm \lambda \in\R^{d_{\set C}} \ | \ \bm \lambda^\top \bm y\ge 0\;\; \forall \bm y \in \set C\}$. As it constitutes an intersection of closed half-spaces whose boundaries contain the origin, $\set C^*$ is a closed convex cone. It is further known that $\set C^*$ is proper if and only if $\set C$ is proper \citep[Corollary~1.4.1]{bn13}. We also adopt the convention that~$\bm \lambda^\top (+\bm\infty_{\set C})=+\infty$ and~$\bm \lambda^\top (-\bm\infty_{\set C})=-\infty$ for all~$\bm \lambda\in \set C^*\backslash\{\bm 0\}$.

      \begin{lem}[Scalarization of $\set C$-Convex Functions]
      \label{lem:scalarization}
      If $\set C\subseteq \R^{d_{\set C}}$
      is a proper convex cone, then~$\bm f:\mathbb R^{d_{\bm x}} \rightarrow \overline\R{}^{d_{\set C}}$ is proper and $\set C$-convex if and only if~$\bm \lambda^\top \bm f$ is proper and convex for every $\bm \lambda \in \set C^*\backslash\{\bm 0\}$. 
      \end{lem}

      \rev{Example~\ref{ex:C-convex-functions} in Appendix~\ref{app_before_A} in the Electronic Companion describes vector- and matrix-valued functions that are convex with respect to some proper convex cones but have components that fail to be convex in the usual sense.} 

      Next, we introduce a generalized notion of lower semicontinuity due to~\citet{jeyakumar05}.

      \begin{defi}[Star $\set C$-Lower Semicontinuity]
      If $\set C\subseteq \R^{d_{\set C}}$ is a proper convex cone, then $\bm f:\mathbb R^{d_{\bm x}} \rightarrow \overline\R{}^{d_{\set C}}$ is called star $\set C$-lower semicontinuous if~$\bm \lambda^\top\bm f$ is lower semicontinuous for every~$\bm \lambda \in \set C^*\backslash\{\bm 0\}$.
      \end{defi}

      One can prove that if~$\bm f$ is star $\set C$-lower semicontinuous, then its $\set C$-epigraph~$\textnormal{epi}_{\set C} (\bm f)$ is closed \citep[Proposition~2.2.19]{bot2009duality}. Contrary to standard intuition, however, the converse implication is false in general. Indeed, there exist proper $\set C$-convex functions that have a closed $\set C$-epigraph but fail to be star $\set C$-lower semicontinuous; see, {\em e.g.}, \citet[Example~2.2.6]{bot2009duality}.

      The following proposition shows that the convex perspectives of proper, closed and convex functions naturally extend to proper, star $\set C$-lower semicontinuous and $\set C$-convex functions.
      \begin{prop}[$\set C$-Convex Perspective]
      \label{prop:C-convex-perspective}
      If $\set C\subseteq \R^{d_{\set C}}$ is a proper convex cone and $\bm f:\mathbb R^{d_{\bm x}} \rightarrow \overline\R{}^{d_{\set C}}$ is a proper, star $\set C$-lower semicontinuous and $\set C$-convex function, then there exists a unique function $\underline{\bm f}: \R^{d_{\bm x}}\times\R_+ \rightarrow \overline \R{}^{d_{\set C}}$, which we will call the $\set C$-convex perspective of~$\bm f$, with the following properties.
      \begin{itemize}
      \item[(i)] $\underline{\bm f}$ is proper, star $\set C$-lower semicontinuous and $\set C$-convex.
      \item[(ii)] $ \underline{\bm f}(\bm x, t) = t \bm f(\bm x/ t)$ for all~$t>0$ and $\bm x\in \mathbb R^{d_{\bm x}}$.
      \item[(iii)] $\bm \lambda^\top \underline{\bm f}(\bm x, 0)= \delta^*_{\dom((\bm \lambda^\top\bm f)^*)}(\bm x)$ for all $\bm x\in \mathbb R^{d_{\bm x}}$ and~$\bm \lambda \in \set C^*\backslash\{\bm 0\}$.
      \end{itemize}
      \end{prop}

      In the following we use $t\bm f(\bm x / t )$ to denote the $\set C$-convex perspective $\underline{\bm f}(\bm x, t)$ of any proper, star $\set C$-lower semicontinuous and $\set C$-convex function~$\bm f$ for all~$t\ge 0$.

      We now study the following generalization of the uncertainty quantification problem~\eqref{eq:puq}.
      \begin{samepage}
      \begin{empheq}[box=\fbox]{align*}
      \tag{P-UQ$_{\rm g}$}
      \label{eq:puq-conic}
      \begin{array}{l}
      \displaystyle \sup_{\bP \in \amb{P}_{\rm g}} \EP{g(\bmt{z})} \left[ g(\bmt{z}) \right]
      \end{array}
      \end{empheq}
      $\mspace{200mu}$   \textbf{Primal Uncertainty Quantification Problem}
      \end{samepage}

      \noindent In contrast to Section~\ref{sec:drco}, however, we now consider a generalized ambiguity set representable as
      \begin{equation*}
      \amb{P}_{\rm g}=\left\{ \bP \in \amb{P}_0(\{\set S_k,p_k\}_k) \; \left \vert \; \EP{\bm h(\bmt z)} \left[\bm h_j(\bmt z)\right] \preceq_{\set H_j} \bm{\mu}_j \ \forall j \in \set J \right. \right\},
      \end{equation*}
      where~$\amb{P}_0(\{\set S_k,p_k\}_k)$ denotes the set of all probability distributions $\mathbb{P}$ supported on~$\set S=\cup_{k\in\set K}\set S_k$ such that~$\bP[{\bmt z}\in\set S_k]=p_k$ for all~$k\in\set K$. We assume that the probabilities~$p_k$ are strictly positive for all~$k\in\set K$ and that they sum up to~$1$. Note that this assumption makes only sense if the different components $\set S_k = \{ \bm z \in \R^{d_{\bm z}} \mid \bm c_{\ell k} (\bm z) \preceq_{\set C_{\ell k}} \bm 0 \; \forall \ell \in \set L_k \}$, $k\in\set K$, of the support set are mutually disjoint. Here, the sets~$\set C_{\ell k}\in\mathbb R^{d_{\set C_{\ell k}}}$ represent proper convex cones, and the functions~$\bm c_{\ell k}:\R^{d_{\bm z}} \rightarrow \overline{\R}{}^{d_{\set C_{\ell k}}}$ obey the following regularity condition that generalizes condition~\textbf{\textbf{(C)}} from Section~\ref{sec:rco}.
      \begin{itemize}
      \item[\textbf{(C$_{\rm g}$)}] The function $\bm c_{\ell k}$ is proper, star $\set C_{\ell k}$-lower semicontinuous and $\set C_{\ell k}$-convex for every~$\ell \in \set L_k$ and~$k \in \set K$.
      \end{itemize}
      We further assume that the disutility function $g:\R^{d_{\bm z}}\rightarrow \overline \R$ satisfies $g(\bm z) = \max_{i \in \set I_k} g_{ik}( \bm z)$ whenever~$\bm z \in \set S_k$ for some~$k \in \set K$, where the component functions $g_{ik}:\R^{d_{\bm z}} \rightarrow \overline \R$ obey the following regularity condition that is the natural analogue of condition~\textbf{\textbf{(G)}} from Section~\ref{sec:drco}.
      \begin{itemize}
      \item[\textbf{(G$_{\rm g}$)}] The function $-g_{ik}$ is proper, closed and convex for every $i \in \set I_k$ and $k \in \set K$.
      \end{itemize}
      As in Section~\ref{sec:drco}, we set $\EP{g(\bmt{z})} \left[ g(\bmt{z}) \right] = - \infty$ if the expectation of the positive and negative parts of~$g(\bmt{z})$ are both infinite. Finally, we assume that the sets~$\set H_j\subseteq \R^{d_{\set H_j}}$ are proper convex cones, the vectors $\bm \mu_j\in\R^{d_{\set H_j}}$ represent moment bounds, and the moment functions $\bm h_j:\R^{d_{\bm z}} \rightarrow \overline{\R}{}^{d_{\set H_j}}$ satisfy $\bm  h_j(\bm z) = \bm h_{jk}( \bm z)$ whenever $\bm z \in \set S_k$ for some~$k \in \set K$, where the component functions $\bm h_{jk}:\R^{d_{\bm z}} \rightarrow \overline \R^{d_{\set H_j}}$ obey the following regularity condition that generalizes condition~\textbf{\textbf{(H)}} from Section~\ref{sec:rco}.
      \begin{itemize}
      \item[\textbf{(H$_{\rm g}$)}] The function $\bm h_{jk}$ is proper, star $\set H_j$-lower semicontinuous and $\set H_j$-convex for every $j \in \set J$ and~$k\in\set K$.
      \end{itemize}
      Some comments about the interpretation of the expectation~$\EP{\bm h(\bmt z)} [\bm h_j(\bmt z)]$ are in order. In analogy to Section~\ref{sec:drco}, for any fixed~$\bm \lambda\in\set H_j^*\backslash\{\bm 0\}$ we set $\EP{\bm h(\bmt z)} [\bm \lambda^\top\bm h_j(\bmt z)]=+\infty$ whenever the expectation of the positive and negative parts of $\bm\lambda^\top \bm h_j(\bmt{z})$ are both infinite. We then define~$\EP{\bm h(\bmt z)} [\bm h_j(\bmt z)]$ as~$+\bm \infty_{\set H_j}$ if there exists~$\bm \lambda\in\set H_j^*\backslash\{\bm 0\}$ with $\EP{\bm h(\bmt z)} [\bm \lambda^\top\bm h_j(\bmt z)]=+\infty$. Similarly, we define~$\EP{\bm h(\bmt z)} [\bm h_j(\bmt z)]$ as~$-\bm \infty_{\set H_j}$ if there exists~$\bm \lambda\in\set H_j^*\backslash\{\bm 0\}$ with $\EP{\bm h(\bmt z)} [\bm \lambda^\top\bm h_j(\bmt z)]=-\infty$ and if~$\EP{\bm h(\bmt z)} [\bm \lambda^\top\bm h_j(\bmt z)]<+\infty$ for every~$\bm \lambda\in\set H_j^*\backslash\{\bm 0\}$.
      Note that we have specified the disutility function~$g$ and the moment functions~$\bm h_j$, $j\in\set J$, only on the set~$\set S=\cup_{k\in\set K} \set S_k$. Specifying these functions beyond~$\set S$ is not necessary, however, because all distributions in the ambiguity set~$\amb{P}_{\rm g}$ are supported on~$\set S$. Next, we define $\bar{\set S}_{k}  = \cup_{i\in \set I_k}\bar{\set S}_{ik}$, where
      \begin{align*}
      \bar{\set S}_{ik}  = \{  \bm z \in \set S_k \ \mid \  \bm z \in \dom (\bm h_{jk}) \;\; \forall j\in \set J, \;\; \bm z \in \dom(-g_{ik}) \}
      \end{align*}
      is convex but not necessarily closed for every $i\in\set I_k$ and $k\in \set K$, and impose the following condition.
      \begin{itemize}
      \item[\textbf{(S$_{\rm g}$)}] The set $\bar{\set S}_{ik}$ is nonempty for every $i \in \set I_k$ and $k \in \set K$.
      \end{itemize}
      Assumption~\textbf{(S$_{\rm g}$)} may be imposed without much loss of generality because any distribution~$\bP\in\amb P_{\rm g}$ assigns a strictly positive probability~$p_k$ to the event~$\bmt z\in \bar{\set S_k}$ and because the conditional distribution~$\bP(\cdot|\bmt z\in \bar{\set S_k})$ must be supported on~$\bar{\set S}_k$. This observation implies that the optimal value of~\eqref{eq:puq-conic} does not change if we remove those components $g_{ik}$ from $g$ for which~$\bar{\set S}_{ik}$ is empty.


      As in Section~\ref{sec:drco}, the uncertainty quantification problem~\eqref{eq:puq-conic} admits a dual akin to~\eqref{eq:duq}.
      \begin{samepage}
      \begin{empheq}[box=\fbox]{equation*}
      \tag{D-UQ$_{\rm g}$}
      \label{eq:duq-conic}
      \begin{array}{l@{\quad}l@{\quad}l}
      \inf  & \displaystyle \sum_{k\in\set K} p_k \alpha_k  +  \sum_{j \in \set J } \bm \mu_j^\top \bm \beta_j \\
      \subj &  \displaystyle  \sup_{ \bm z_k \in \bar{\set S}_{k}} \left\{ g(\bm z_k)- \alpha_k -  \sum_{j \in \set J }\bm h_{jk} (\bm z_k)^\top \bm \beta_j \right\}  \le 0 & \forall k\in\set K\\
      & \bm \alpha \text{ free}, \;\; \bm \beta_j \in \set H^*_j, \; j\in\set J
      \end{array}
      \end{empheq}
      $\mspace{190mu}$  \textbf{Dual Uncertainty Quantification Problem}
      \end{samepage}

      \noindent By using a similar reasoning as in the proof of Theorem~\ref{coro:cvxdro-ro-wd}, it is easy to show that the infimum of~\eqref{eq:puq-conic} is always larger or equal to the supremum of~\eqref{eq:duq-conic}. In addition, the uncertainty quantification problem~\eqref{eq:puq-conic} admits the following finite reduction akin to~\eqref{eq:fr}.
      \begin{samepage}
      \begin{empheq}[box=\fbox]{equation*}  \tag{FR$_{\rm g}$}
      \label{eq:fr-conic}
      \begin{array}{l@{\quad}l@{\qquad}l}
      \sup  & \displaystyle  \sum_{k \in \set{K}} \sum_{i \in \set{I}_k} \lambda_{ik} g  (\bm z_{ik})  \\
      \subj & \displaystyle\sum_{i \in \set{I}_k} \lambda_{ik} =  p_k & \forall k \in \set K \\
      & \displaystyle \sum_{k \in \set{K}} \sum_{i \in \set{I}_k}  \lambda_{ik} \bm{h}_j (\bm z_{ik}) \preceq_{\set H_j} \bm \mu_j & \forall j \in \set J  \\
      & \multicolumn{2}{l}{\mspace{-10mu} \bm z_{ik} \in \bar{\set S}_{k}, \; i \in \set{I}_k, \, k \in \set K, \;\; \bm \lambda_k \ge\bm 0, \; k\in\set K}
      \end{array}
      \end{empheq}
      $\mspace{255mu}$  \textbf{Finite Reduction Problem}
      \end{samepage}

      \noindent Note that problem~\eqref{eq:fr-conic} provides a lower bound on~\eqref{eq:puq-conic} because it evaluates the worst-case expected disutility across all discrete distributions $\bP \in \amb{P}$ with discretization points~$\bm{z}_{ik}\in\bar{\set S}_k$ and corresponding probabilities $\lambda_{ik}$ for $i \in \set{I}_k$ and $k\in\set K$. Introducing the augmented support sets
      \[
      \set U_{ik} = \left\{\left. (\bm z, \{\bm u_j \}_j, t) \in \R^{d_{\bm z}}\times \left(\bigtimes_{j\in\set J} \R^{d_{\set H_j}}\right) \times \R \;\right|\!\! \begin{array}{l} \bm c_{\ell k} (\bm z) \preceq_{\set C_{\ell k}} \bm 0 \enskip \ \forall \ell \in \set L_k,\\  \bm h_{jk}(\bm z) \preceq_{\set H_j} \bm u_j \ \forall j\in \set J, \ g_{ik}(\bm z) \ge t \end{array}\right\}
      \]
      for $i \in \set I_k$ and $k \in \set K$, we can then construct two auxiliary optimization problems~\mbox{(AP-W${}_{\rm g}$)} and~\mbox{(AD-B${}_{\rm g}$)} equivalent to~\eqref{eq:duq-conic} and~\eqref{eq:fr-conic}, respectively, as well as two finite convex programs~\eqref{eq:apw-conic-cvx} and~\eqref{eq:adb-conic-cvx}. These problems are constructed in the same way as their natural counterparts from Section~\ref{sec:drco} with obvious minor modifications. { For the sake of brevity, we do not display the problems~\mbox{(AP-W${}_{\rm g}$)} and~\mbox{(AD-B${}_{\rm g}$)}.} \mbox{An explicit representation of~\eqref{eq:apw-conic-cvx} is shown below.}

      \begin{samepage}
      \begin{empheq}[box=\fbox]{equation*}  \tag{AP-W$'_{\rm g}$}
      \label{eq:apw-conic-cvx}
      \begin{array}{l@{\quad}l}
      \inf  &  \displaystyle \sum_{k\in \set{K}}   p_k \alpha_k  +  \bm \mu^\top  \bm \beta \\
      \subj &  \displaystyle   (-g_{ik})^{*}(\bm y^{(0)}_{ik} )  +  \sum_{j \in \set J } (\bm \beta_j^\top \bm h_{jk})^{*} \left( \bm y^{(1)}_{ijk} \right) \\
      &  \displaystyle \qquad \qquad \quad \quad  + \sum_{\ell \in \set L }  (\bm \nu_{i\ell k}^\top \bm c_{\ell k})^*    \left( \bm y^{(2)}_{i \ell k} \right)  \le \alpha_k  \qquad  \quad \qquad \enskip \forall i \in \set I_k,  \ \forall k\in \set{K}  \\
      & \displaystyle  \bm y^{(0)}_{ik} +   \sum_{j \in \set J } \bm y^{(1)}_{ijk} + \sum_{\ell \in \set L }  \bm y^{(2)}_{i\ell k} = \bm 0  \qquad \qquad \qquad \qquad  \qquad  \ \ \ \forall i \in \set I_k, \ \forall k\in \set{K}  \\
      &   \bm \alpha  \textnormal{ free} , \;\; \bm \beta_j \in \set H^*_j, 
      \; j \in \set J, \;\; \bm y^{(0)}_{ik}, \bm y^{(1)}_{ijk}, \bm y^{(2)}_{i\ell k} \textnormal{ free}, \\
      & \bm \nu_{i\ell k} \in \set C^*_{\ell k }, \; i\in \set{I}_k, \, j\in \set{J}, \, \ell \in \set{L}_k, \, k\in \set{K}
      \end{array}
      \end{empheq}
      $\mspace{125mu}$ \textbf{Reformulated Ambiguous Primal Worst Problem \quad \qquad}
      \end{samepage}

      \noindent Similarly, the finite convex program~\eqref{eq:adb-conic-cvx} can be represented as follows.

      \begin{samepage}
      \begin{empheq}[box=\fbox]{equation*}  \tag{AD-B$'_{\rm g}$}
      \label{eq:adb-conic-cvx}
      \begin{array}{l@{\quad}l@{\qquad}l}
      \sup  & \displaystyle  \sum_{k\in \set{K}} \sum_{i\in \set{I}_k} \tau_{ik}  \\
      \subj & \multicolumn{2}{l}{\displaystyle \mspace{-10mu} \sum_{i\in \set{I}_k} \lambda_{ik} = p_k \;\; \forall k \in \set{K}, \qquad
      \sum_{k\in \set{K}} \sum_{i\in \set{I}_k} \bm \omega_{ijk} \preceq_{\set H_j} \bm \mu_j \;\; \forall j \in \set J} \\
      & \lambda_{ik} \bm c_{\ell k} (\bm v_{ik}/\lambda_{ik} ) \preceq_{\set C_{\ell k}} \bm 0 & \forall i \in \set I_k, \ \forall \ell \in \set{L}_k,   \forall k\in \set{K}\\
      & \lambda_{ik} \bm h_{jk} (\bm v_{ik}/\lambda_{ik} ) \preceq_{\set H_j} \bm \omega_{ijk} & \forall i \in \set I_k, \ \forall j \in \set J, \ \forall k\in \set{K}\\
      & \lambda_{ik} g_{ik} ( \bm v_{ik}/\lambda_{ik} ) \ge \tau_{ik} & \forall i \in \set I_k, \ \forall k\in \set{K}\\
      & \multicolumn{2}{l}{\mspace{-10mu} \bm \tau_{k} \text{ free}, \;\; \bm \lambda_k \ge \bm 0, \;\; \bm \omega_{ijk}, \bm v_{ik} \text{ free}, \; i \in \set I_k, \, j \in \set J, \, k\in \set{K}}
      \end{array}
      \end{empheq}
      $\mspace{145mu}$\textbf{Reformulated Ambiguous Dual Best Problem}
      \end{samepage}

      We are now ready to state a strong duality result akin to Theorem~\ref{thm:strong-duality-all-problems}.


      \begin{thm}[Strong Duality for~\eqref{eq:puq-conic} and~\eqref{eq:duq-conic}]
      \label{thm:strong-duality-all-problems-conic}
      The following statements hold.
      \begin{enumerate}[label=(\roman*)]
      \item If~\eqref{eq:adb-conic-cvx} admits a Slater point with $\bm \lambda_k > \bm 0$, $k\in\set K$,
      then~\eqref{eq:apw-conic-cvx}, \eqref{eq:duq-conic}, \eqref{eq:puq-conic}, \eqref{eq:fr-conic} and~\eqref{eq:adb-conic-cvx} all have the same optimal value and~\eqref{eq:apw-conic-cvx} is solvable. Also, if $(\bm \alpha^\star,$ $\{ \bm \beta^\star_j\}_j, \{\bm y_{ik}^{(0)\star}\}_{ik}, \{\bm y_{ijk}^{(1)\star}\}_{ijk},  \{ \bm y_{i\ell k}^{(2)\star}, \nu_{i\ell k}^\star \}_{i\ell k})$ solves~\eqref{eq:apw-conic-cvx}, then $(\bm \alpha^\star, \{ \bm \beta^\star_j\}_j)$ solves~\eqref{eq:duq-conic}.
      \item If~\eqref{eq:apw-conic-cvx}
      admits a Slater point and $\set S_k$ is bounded for every~$k\in\set K$, then \eqref{eq:apw-conic-cvx}, \eqref{eq:duq-conic}, \eqref{eq:puq-conic}, \eqref{eq:fr-conic} and \eqref{eq:adb-conic-cvx} all have the same optimal value and \eqref{eq:adb-conic-cvx} is solvable. Also, if $(\{\bm \tau^\star_{k}\}_k, \{ \bm \lambda^\star_{k}\}_{k},  \{\bm \omega^\star_{ijk}\}_{ijk}, \{\bm v^\star_{ik}\}_{ik})$ solves~\eqref{eq:adb-conic-cvx}, then the discrete distribution that assigns probability $\lambda^\star_{ik}$ to the point $\bm v_{ik}^\star/\lambda_{ik}^\star$ for every $i\in\set I_k$ and $k \in \set K$ with $\lambda_{ik}^\star > 0$ solves~\eqref{eq:puq-conic}.
      \end{enumerate}
      \end{thm}



      The proof of Theorem~\ref{thm:strong-duality-all-problems-conic} parallels that of Theorem~\ref{thm:strong-duality-all-problems} and is omitted for the sake of brevity. The assumptions of Theorem~\ref{thm:strong-duality-all-problems-conic}~\emph{(i)} do not guarantee the solvability of~\eqref{eq:puq-conic}. The following corollary shows, however, that~\eqref{eq:puq-conic} can still be solved asymptotically under these assumptions.

      \begin{coro}[Approximate Numerical Solution of~\eqref{eq:puq-conic}] \label{coro:puq-asy-conic}
      If~\eqref{eq:adb-conic-cvx} admits a Slater point with $\bm \lambda_k >\bm 0$ for all~$k\in\set K$, then the suprema of~\eqref{eq:puq-conic} and~\eqref{eq:adb-conic-cvx} coincide, and any $\epsilon$-optimal solution of~\eqref{eq:adb-conic-cvx} can be used to construct a discrete distribution with $\sum_{k\in\set K}I_k$ atoms that is $2\epsilon$-optimal in~\eqref{eq:puq-conic}.
      \end{coro}

      The proof of Corollary~\ref{coro:puq-asy-conic} is similar to that of Corollary~\ref{coro:puq-asy} and therefore also omitted. Note also that Propositions~\ref{prop:slater-distribution} and~\ref{prop:slater-distribution2} generalize to the setting of this section in a natural way.

      { Example~\ref{example:random_matrix_theory} in Appendix~\ref{app:om} in the Electronic Companion employs the techniques developed in this section to analyze the spectral properties of random matrices governed by an ambiguous distribution.}
      
      \section{Application: Optimal Transport-Based Uncertainty Quantification and Distributionally Robust Optimization}  \label{sec:dot}

      We now apply the theory of Section~\ref{sec:extensions} to derive tractable reformulations for uncertainty quantification problems whose ambiguity sets are defined in terms of an optimal transport distance.

      \begin{defi}[Optimal Transport Distance] \label{def:opt}
      The optimal transport distance between two probability distributions $\bP, \bP' \in\amb{P}_0(\R^{d_{\bm z}})$ induced by the transportation cost~$d:\R^{d_{\bm z}} \times \R^{d_{\bm z}} \rightarrow [0, +\infty]$ is given by $D (\bP, \bP') =  \inf_{\mathbb{Q} \in \mathcal{Q}(\bP, \bP')} \mathbb E_{\mathbb Q} \left[ d(\bmt z, \bmt z') \right]$, where $\mathcal{Q}(\bP, \bP')$ denotes the set of all joint probability distributions or `couplings' $\mathbb{Q}$ of $\bmt z \in \R^{d_{\bm z}}$  and $\bmt z' \in \R^{d_{\bm z}}$ with marginals $\bP$ and~$\bP'$, respectively.
      \end{defi}

      Below we assume that the transportation cost satisfies the following regularity condition.

      \begin{itemize}
      \item[\textbf{(D)}]  The transportation cost $d(\bm z, \hat{\bm z})$ is proper, closed and convex in~$\bm z$ for every fixed $\hat{\bm z} \in \R^{d_{\bm z}}$.
      \end{itemize}

      The optimal transport distance~$D (\bP, \bP')$ can be interpreted as the minimum cost of turning one pile of dirt represented by $\bP$ into another pile of dirt represented by $\bP'$, where the cost of transporting a unit mass from $\bm z$ to $\bm z'$ amounts to $d( \bm z ,\bm z')$. Any coupling~$\mathbb Q$ of the distributions~$\mathbb P$ and~$\bP'$ can therefore be interpreted as a transportation plan. In the remainder of this section we study an optimal transport-based uncertainty quantification problem of the form
      \begin{equation} \label{eq:uq-ot}
      \sup_{\bP \in \mathbb B_\epsilon(\hat \bP) } \mathbb E_{\mathbb P}[ g(\bmt z)]
      \tag{OT}
      \end{equation}
      with ambiguity set
      \[
      \mathbb B_\epsilon(\hat \bP) =\left\{ \bP \in \amb P_0(\set S) \mid  D (\bP, \hat \bP ) \le  \epsilon   \right\},
      \]
      which can be viewed as a ball of radius~$\epsilon \ge 0$ around a nominal probability distribution~$\hat{\bP} \in \amb P_0(\set S)$ with respect to the optimal transport distance. We assume that the disutility function is representable as~$g(\bm z) = \max_{i \in \set I } g_i( \bm z)$ for some component functions that satisfy condition~\textbf{(G)} from Section~\ref{sec:drco} and that the support set is representable as~$\set S = \{ \bm z \in \R^{d_{\bm z}} \mid c_\ell (\bm z) \le 0 \;\; \forall \ell \in \set L \}$ for some constraint functions that satisfy condition~\textbf{(C)} from Section~\ref{sec:rco}.  We further assume that the nominal distribution is discrete and thus representable as~$\hat \bP = \sum_{ k \in \set{K} }  \hat{p}_k \, \delta_{ \hat{\bm z}_k}$, where~$\delta_{\hat{\bm z}_k}$ denotes the Dirac point mass at $\hat{\bm z}_k\in\R^{d_{\bm z}}$. Note that the Dirac measure~$\delta_{\hat{\bm z}_k}$ should not be confused with the indicator function~$\delta_{\{\hat{\bm z}_k\}}$ of the singleton set~$\{\hat{\bm z}_k\}$. Without loss of generality, we may finally assume that the probabilities~$\hat p_k$, $k\in\set K$, are strictly positive and that the atoms~$\hat{\bm z}_k\in\set S$,~$k\in\set K$, are mutually different for otherwise some atoms could be omitted or combined. The nominal distribution is often given by the empirical distribution on a set of training samples~$\hat{\bm z}_k$, $k\in\set K$, drawn independently from the unknown true distribution of~$\bmt z$. In this case, we simply set~$\hat p_k=1/K$ for every~$k\in\set K$.

      We now demonstrate that the optimal transport-based uncertainty quantification problem~\eqref{eq:uq-ot} can be addressed with the tools developed in Section~\ref{sec:extensions}. To see this, note that
      \begin{align*}
      \sup_{\bP \in \mathbb B_\epsilon(\hat \bP)} \mathbb E_{\mathbb P} \left[ g(\bmt z) \right]
      \; &= \; \sup_{\mathbb P \in \amb{P}_0(\R^{d_{\bm z}})} \Big\{ \mathbb E_{\mathbb P} \left[ g(\bmt z) \right] \,\Big|\, \inf_{\mathbb{Q} \in \mathcal{Q}(\bP, \hat \bP)} \mathbb E_{\mathbb Q} \left[ d(\bmt z, \bmt z') \right]\leq \epsilon
      \Big\}\\
      &= \; \sup_{\mathbb P \in \amb{P}_0(\R^{d_{\bm z}}),\,\mathbb Q \in \mathcal{Q}(\bP, \hat \bP)} \Big\{\mathbb E_{\mathbb P} \left[ g(\bmt z) \right] \,\Big| \, \mathbb E_{\mathbb Q} \left[ d(\bmt z, \bmt z') \right]\leq \epsilon
      \Big\},
      \end{align*}
      where the first equality exploits Definition~\ref{def:opt}, and the second equality follows from Theorem~1.7 by~\citet{s15}, which applies thanks to condition~\textbf{(D)}. Indeed, this theorem ensures that the infimum over~$\mathbb Q \in \mathcal{Q}(\bP, \hat \bP)$ is attained, which allows us to remove the infimum operator on the left hand side of the inequality constraint and to treat the transportation plan~$\mathbb Q$ as a decision variable of the overall maximization problem. Next, we define conditional support sets~$\set S_k=\set S\times\{\hat{\bm z}_k\}$, $k\in\set K$, corresponding to the atoms of the discrete nominal distribution, and in the remainder we use the following representation of these sets in terms of inequality constraints.
      \begin{equation}
      \label{eq:S_k-representation}
      \set S_k = \left\{\left. (\bm z,\bm z') \in \R^{d_{\bm z}}\times\R^{d_{\bm z}} \, \right| \, c_\ell (\bm z) \le 0 \;\; \forall \ell \in \set L, \;\; \bm z'\leq \hat{\bm z}_k, \;\; -\bm z'\leq -\hat{\bm z}_k \right\}\quad \forall k\in\set K
      \end{equation}
      By the construction of~$\set S_k$ we have~$\mathbb Q[(\bmt z,\bmt z')\in\set S_k] = \hat{\mathbb P}[\bmt z'=\hat{\bm z}_k]=\hat p_k$ for all~$k\in\set K$ and~$\mathbb Q \in \mathcal{Q}(\bP, \hat \bP)$. As~$\mathbb P$ is the marginal distribution of~$\bmt z$ under any transportation plan~$\mathbb Q \in \mathcal{Q}(\bP, \hat \bP)$, we can thus reformulate the problem~\eqref{eq:uq-ot} without using~$\mathbb P$~as
      \begin{align*}
      \sup_{\mathbb Q \in \amb{P}_0(\{\set S_k,\hat p_k\}_k)} \Big\{\mathbb E_{\mathbb Q} \left[ g(\bmt z) \right] \,\Big| \, \mathbb E_{\mathbb Q} \left[ d(\bmt z, \bmt z') \right]\leq \epsilon
      \Big\},
      \end{align*}
      where~$\amb{P}_0(\{\set S_k,\hat p_k\}_k)$ is defined as in Section~\ref{sec:extensions}. To show that this reformulation can be solved with the tools of Section~\ref{sec:extensions}, we then set~$\set I_k=\{i\in\set I \, | \, \dom (-g_i)\cap \dom(d(\cdot,\hat{\bm z}_k))\neq\emptyset \}$ for any~$k\in\set K$ and, by slight abuse of notation, we reinterpret the disutility function as a function~$g:\R^{d_{\bm z}} \times \R^{d_{\bm z}} \rightarrow \overline\R$ that depends on two copies~$\bm z$ and~$\bm z'$ of the uncertain problem parameters. Specifically, we assume that this augmented disutility function satisfies~$g(\bm z,\bm z')=\max_{i\in\set I_k} g_i(\bm z)$ whenever~$\bm z'=\hat{\bm z}_k$ for some~$k\in\set K$. Note that it is not necessary to specify~$g(\bm z,\bm z')$ for other values of~$\bm z'$. We also introduce auxiliary functions~$h_k:\R^{d_{\bm z}} \times \R^{d_{\bm z}} \rightarrow \overline{\mathbb R}$ defined through~$h_k(\bm z,\bm z')= d(\bm z,\bm z')$ if~$\bm z'=\hat{\bm z}_k$ and~$h_k(\bm z,\bm z')= +\infty$ otherwise, $k\in\set K$. In addition we introduce~$h:\R^{d_{\bm z}} \times \R^{d_{\bm z}} \rightarrow \overline{\mathbb R}$ and assume that~$h(\bm z,\bm z')=h_k(\bm z,\bm z')$ whenever~$\bm z'=\hat{\bm z}_k$ for some~$k\in\set K$. It is again not necessary to specify~$h(\bm z,\bm z')$ for other values of~$\bm z'$. Using these conventions, the problem~\eqref{eq:uq-ot} can thus be reformulated equivalently as
      \begin{align}
      \label{eq:pot}
      \sup_{\mathbb Q \in \amb{P}_0(\{\set S_k,\hat p_k\}_k)} \Big\{\mathbb E_{\mathbb Q} \left[ g(\bmt z, \bmt z') \right] \,\Big| \, \mathbb E_{\mathbb Q} \left[ h(\bmt z, \bmt z') \right]\leq \epsilon
      \Big\}.
      \tag{P-UG$_\text{OT}$}
      \end{align}
      One readily verifies that~\eqref{eq:pot} represents an instance of~\eqref{eq:puq-conic} that satisfies all pertinent regularity conditions. Indeed, condition~\textbf{(C$_{\rm g}$)} from Section~\ref{sec:extensions} holds because the functions $c_\ell$, $\ell \in \set L$, appearing in~\eqref{eq:S_k-representation} are assumed to satisfy condition~\textbf{(C)} from Section~\ref{sec:rco}. Similarly, condition~\textbf{(G$_{\rm g}$)} from Section~\ref{sec:extensions} holds because the components~$g_i$, $i \in \set I$, of the disutility function are assumed to satisfy condition~\textbf{(G)} from Section~\ref{sec:drco}. In addition, condition~\textbf{(H$_{\rm g}$)} from Section~\ref{sec:extensions} holds because the transportation cost~$d$ is assumed to satisfy condition~\textbf{(D)}, which ensures that~$h_k$ is proper, convex and closed for every~$k\in\set K$. In order to validate condition~\textbf{(S$_{\rm g}$)}, we define $\bar{\set S}_{k}  = \cup_{i\in \set I_k}\bar{\set S}_{ik}$, where
      \begin{align*}
      \bar{\set S}_{ik}  = \left\{ \left. ( \bm z, \bm z') \in \set S_k \,\right|\,  ( \bm z, \bm z') \in \dom (h_{k}), \;\; \bm z  \in \dom(-g_i) \right\}
      \end{align*}
      for every~$i\in\set I_k$ and~$k\in \set K$ as in Section~\ref{sec:extensions}. By the construction of~$\set I_k$, the set~$\bar{\set S}_{ik}$ is nonempty  for every~$i\in\set I_k$ and~$k\in \set K$, and thus problem~\eqref{eq:pot} indeed satisfies condition~\textbf{(S$_{\rm g}$)}.

      If we interpret~\eqref{eq:pot} as an instance of~\eqref{eq:puq-conic}, then one can show that the corresponding instance of the dual uncertainty quantification problem~\eqref{eq:duq-conic} is equivalent to
      \begin{equation} \label{eq:dot}
      \begin{array}{l@{\quad}l}
      \inf  & \displaystyle \sum_{k\in\set K} \hat p_k \alpha_k  +  \epsilon  \beta \\
      \subj &  \displaystyle  \sup_{ (\bm z_k, \bm z_k') \in \bar{\set S}_k} \left\{ g(\bm z_k) -  d(\bm z_k, \hat{\bm z}_k) \beta \right\}  \le \alpha_k  \quad \forall k\in\set K\\
      & \bm \alpha \text{ free}, \quad \beta \ge  0.
      \end{array}
      \tag{D-UQ$_\text{OT}$}
      \end{equation}
      This is an immediate consequence of the observation that~$(\bm z_k, \bm z_k') \in \bar{\set S}_{k}$ implies~$\bm z_k' = \hat{\bm z}_k$ and that~$h(\bm z, \hat{\bm z}_k) = d(\bm z_k, \hat{\bm z}_k)$ for every $k \in \set K$.
      An elementary calculation further shows that the corresponding instance of the finite convex program~\eqref{eq:apw-conic-cvx} is equivalent to
      \begin{equation}
      \label{eq:pot-cvx}
      \begin{array}{l@{\quad}l}
      \inf \hspace{0.5mm} & \displaystyle \sum_{k\in \set{K}}   \hat p_k \alpha_k  +  \epsilon  \beta \\
      \subj &  \displaystyle   (-g_i)^{*}(\bm y^{(0)}_{ik} ) +  \beta d^{*1}\left( \frac{\bm y^{(1)}_{ik}}{\beta}, \hat{\bm z}_k \right)   + \sum_{\ell \in \set{L}}  \nu_{i\ell k} c_\ell^*    \left( \frac{\bm y^{(2)}_{i\ell k} }{\nu_{i\ell k}} \right) \le \alpha_k \quad \forall i \in \set I_k, \ \forall k\in \set{K}  \\
      & \displaystyle  \bm y^{(0)}_{ik} +  \bm y^{(1)}_{ik}   + \sum_{\ell \in \set{L}} \bm   y^{(2)}_{i\ell k} = \bm 0  \qquad \qquad \qquad \qquad \qquad \qquad \qquad  \enskip \ \forall i \in \set I_k, \ \forall k\in \set{K}  \\
      &   \alpha_k, \bm y^{(0)}_{ik}, \bm y^{(1)}_{ik}, \bm y^{(2)}_{i\ell k} \textnormal{ free} , \quad \beta \ge  0  , \quad   \nu_{i\ell k} \ge 0 \qquad \quad \ \ \ \forall i\in \set{I}_k, \ \forall \ell \in \set{L}, \ \forall k\in \set{K},
      \end{array}
      \tag{AP-W$^\prime_\text{OT}$}
      \end{equation}
      \noindent
      while the corresponding instance of~\eqref{eq:adb-conic-cvx} is equivalent to
      \begin{equation}
      \label{eq:dot-cvx}
      \begin{array}{l@{\quad}l@{\quad}l}
      \sup  & \displaystyle  \sum_{k\in \set{K}} \sum_{i\in \set{I}_k} \tau_{ik}  \\
      \subj & \displaystyle\sum_{i\in \set{I}_k} \lambda_{ik} = \hat p_k & \forall k \in \set{K}\\
      & \displaystyle \sum_{k\in \set{K}}  \sum_{i\in \set{I}_k} \omega_{ik} \le \epsilon \\
      & \lambda_{ik} c_\ell (\bm v_{ik}/\lambda_{ik} ) \le 0 & \forall i \in \set I_k, \ \forall \ell \in \set L, \ \forall k\in \set{K}\\
      & \lambda_{ik} d(\bm v_{ik}/\lambda_{ik}, \hat{\bm z}_k) \le \omega_{ik} & \forall i \in \set I_k, \ \forall k\in \set{K}\\
      & \lambda_{ik} g_i ( \bm v_{ik}/\lambda_{ik} ) \ge \tau_{ik} & \forall i \in \set I_k, \ \forall k\in \set{K}\\
      & \tau_{ik} \text{ free}, \quad \lambda_{ik} \ge 0, \quad \omega_{ik}, \bm v_{ik} \text{ free} & \forall i \in \set I_k, \ \forall k\in \set{K}.
      \end{array}
      \tag{AD-B$^\prime_\text{OT}$}
      \end{equation}
      Note that in~\eqref{eq:dot-cvx} we have eliminated the inequality constraints~$\bm z'\leq \hat{\bm z}_k$ and $-\bm z'\leq -\hat{\bm z}_k$, which emerge in the representation~\eqref{eq:S_k-representation} of the conditional support set~$\set S_k$, $k \in \set K$, and in~\eqref{eq:pot-cvx} we have eliminated the corresponding dual variables. Theorem~\ref{thm:strong-duality-all-problems-conic}~\emph{(i)} guarantees that if~\eqref{eq:dot-cvx} admits a Slater point with $\lambda_{ik} > 0$ for all~$i\in\set I_k$ and~$k\in\set K$, then~\eqref{eq:pot}, \eqref{eq:dot}, \eqref{eq:pot-cvx} and~\eqref{eq:dot-cvx} all have the same optimal value as~\eqref{eq:uq-ot} and~\eqref{eq:pot-cvx} is solvable. Also, if $(\bm \alpha^\star, \{\bm y_{ik}^{(0)\star}, \bm y_{ik}^{(1)\star}\}_{ik}, \{ \bm y_{i\ell k}^{(2)\star}\}_{i\ell k}, \beta^\star, \{\nu_{i\ell k}^\star \}_{i\ell k})$ solves~\eqref{eq:pot-cvx}, then $(\bm \alpha^\star, \beta^\star)$ solves~\eqref{eq:dot}. Similarly, Theorem~\ref{thm:strong-duality-all-problems-conic}~\emph{(ii)} guarantees that if~\eqref{eq:pot-cvx} admits a Slater point and~$\set S$ is bounded, then~\eqref{eq:pot}, \eqref{eq:dot}, \eqref{eq:pot-cvx} and~\eqref{eq:dot-cvx} all have the same optimal value as~\eqref{eq:uq-ot} and~\eqref{eq:dot-cvx} is solvable. Also, if $(\{\bm \tau^\star_{k}\}_k, \{ \bm \lambda^\star_{k}\}_{k},  \{ \omega^\star_{ik}\}_{ik}, \{\bm v^\star_{ik}\}_{ik})$ solves~\eqref{eq:dot-cvx}, then the discrete distribution that assigns probability $\lambda^\star_{ik}$ to the point $\bm v_{ik}^\star/\lambda_{ik}^\star$ for every $i\in\set I_k$ and $k \in \set K$ with $\lambda_{ik}^\star > 0$ solves~\eqref{eq:uq-ot}. Theorem~\ref{thm:strong-duality-all-problems-conic} thus establishes, among other things, different conditions for strong duality between the semi-infinite programs~\eqref{eq:uq-ot} and~\eqref{eq:dot}. Such strong duality results are at the heart of modern Wasserstein distributionally robust optimization; see Theorem~4.2 of \citet{med17} and Theorem~1 of \citet{ref:zhao2018data} for finite dimensional and Theorem~1 of \citet{Gao_Kleywegt_2016} and Theorem~1 of \citet{ref:blanchet2019quantifying} for infinite dimensional uncertainty sets. Theorem~\ref{thm:strong-duality-all-problems-conic} provides a new and elementary proof for strong semi-infinite duality, and it relies only on explicit conditions that are easy to check. Note that the conditions of Theorem~\ref{thm:strong-duality-all-problems-conic}~\emph{(i)} are indeed very weak and are essentially always satisfied if $\epsilon>0$. While the uncertainty quantification problem~\eqref{eq:uq-ot} and its convex reformulation~\eqref{eq:dot-cvx} may fail to be solvable under these conditions, Corollary~\ref{coro:puq-asy-conic} shows that near-optimal solutions to~\eqref{eq:dot-cvx} can systematically be converted to near-optimal solutions to~\eqref{eq:uq-ot}.

      { We will argue next} that problem~\eqref{eq:dot-cvx} can be further simplified and that it is guaranteed to be solvable under mild additional conditions that are usually met in practice. To this end, note first that by eliminating the auxiliary decision variables~$\{ \tau_{ik}, \omega_{ik}\}_{ik}$ and applying the linear variable substitution $\bm v_{ik}\leftarrow \bm v_{ik} + \hat{\bm z}_k\lambda_{ik}$ for all~$i\in\set I_k$ and~$k\in\set K$, problem~\eqref{eq:dot-cvx} simplifies to
      \begin{equation}
      \label{eq:dot-cvx-explicit}
      \begin{array}{l@{\quad}l@{\qquad}l}
      \sup  &  \displaystyle \sum_{k\in \set{K}} \sum_{i\in \set{I}_k} \lambda_{ik} g_i   \left(  \hat{\bm z}_k + \frac{\bm v_{ik}}{\lambda_{ik}} \right)  \\
      \subj &  \displaystyle\sum_{i \in \set{I}_k} \lambda_{ik} =  \hat p_k & \forall k\in \set{K}\\
      & \displaystyle \lambda_{ik} c_\ell \left(\hat{\bm z}_k + \frac{\bm v_{ik}}{\lambda_{ik}} \right) \le 0 & \forall i \in \set I_k, \ \forall \ell \in \set L, \ \forall k\in \set{K}\\
      & \displaystyle \sum_{k\in \set{K}}  \sum_{i\in \set{I}_k} \lambda_{ik} d\left(  \hat{\bm z}_k + \frac{\bm v_{ik}}{\lambda_{ik}},  \hat{\bm z}_k \right) \le \epsilon \\
      & \lambda_{ik} \ge 0,  \quad \bm v_{ik} \text{ free} & \forall i \in \set I_k, \ \forall k\in \set{K}.
      \end{array}
      \end{equation}
      This reformulation is always solvable under mild assumptions on the transportation costs.

      \begin{prop}[Solvability of~\eqref{eq:dot-cvx-explicit}]
      \label{prop:solvability-of-POT-convex}
      Assume that the transportation cost~$d(\bm z,\bm z')$ satisfies the identity of indiscernibles, that is, $d(\bm z, \bm z')=0$ if and only if~$\bm z=\bm z'$. Then, problem~\eqref{eq:dot-cvx-explicit} is solvable.
      \end{prop}

      In the following we explain how any maximizer $( \{ \lambda_{ik}^\star, \bm v_{ik}^\star\}_{ik})$ of the finite convex program~\eqref{eq:dot-cvx-explicit} can be used to construct a distribution~$\mathbb P^\star$ that is optimal in~\eqref{eq:uq-ot} (if such a~$\mathbb P^\star$ exists) or a sequence of asymptotically optimal distributions $\{\mathbb P_n\}_{n\in\mathbb N}$ (if such a $\mathbb P^\star$ does not exist). 

      \rev{ 
      \begin{prop}\label{prop:show_me_the_distributions}
      Assume that the conditions of Proposition~\ref{prop:solvability-of-POT-convex} hold and that~\eqref{eq:dot-cvx} admits a Slater point with $\lambda_{ik} > 0$ for all~$i\in\set I_k$ and~$k\in\set K$. If $\set I^{\infty}_k = \{ i \in \set I_k \ | \ \lambda_{ik}^\star = 0, \  \bm v_{ik}^\star \ne \bm 0 \} = \emptyset$ for every $k \in \set K$, then the discrete distribution
      \begin{equation}\label{eq:OT:discrete_distr_optimal}
      \mathbb P^\star \; = \; \sum_{k \in \set K} \sum_{i \in \set I_k^{+}} \lambda_{ik}^\star \delta_{\hat{\bm z}_{k} + \bm v_{ik}^\star/\lambda_{ik}^\star}
      \end{equation}
      is optimal in~\eqref{eq:uq-ot}, where $\set I^{+}_k = \{ i \in \set I_k \ | \ \lambda_{ik}^\star > 0 \}$. Otherwise, the discrete distributions
      \begin{equation}\label{eq:OT:sequence_optimal}
      \mathbb P_n \; = \; \sum_{k \in \set K} \sum_{i \in\set I^{+}_k \cup \set I^{\infty}_k} \lambda_{ik} (n) \delta_{\bm z_{ik}(n)}
      \end{equation}
      for~$n\in\mathbb N$ are asymptotically optimal in~\eqref{eq:uq-ot}, where
      \[
      \lambda_{ik} (n) =  \begin{cases}\lambda_{ik}^\star \left(1 - \frac{ |\set I^{\infty}_k| }{n} \right) & \text{if } i \in \set I^{+}_k    \\
      \frac{\hat p_k}{n} & \text{if } i \in \set I^{\infty}_k
      \end{cases}\qquad \text{and}\qquad
      \bm z_{ik} (n) =  \begin{cases}\hat{\bm z}_k + \frac{\bm v^\star_{ik}}{\lambda_{ik}^\star} & \text{if } i \in \set I^{+}_k    \\
      \hat{\bm z}_k + n \frac{\bm v^\star_{ik}}{\hat p_k} & \text{if } i \in \set I^{\infty}_k.
      \end{cases}
      \]
      \end{prop}

      The distributions $\mathbb P_n$ in~\eqref{eq:OT:sequence_optimal} send some atoms with decaying probabilities~$\hat p_k/n$ to infinity along the vectors $\bm v_{ik}^\star$, $i\in \set I^{\infty}_k$, $k\in \set K$, which are recession directions of the support set. It is perhaps surprising that all distributions~$\bP_n$ can be constructed from one single optimal solution of problem~\eqref{eq:dot-cvx-explicit}. Conversely, in order to construct asymptotically optimal probability distributions for generic non-degenerate uncertainty quantification problems of the form~\eqref{eq:puq} and~\eqref{eq:puq-conic}, one has to compute sequences of asymptotically optimal solutions for the finite convex programs~\eqref{eq:adb-cvx} and~\eqref{eq:adb-conic-cvx}, respectively, which may not be solvable in general; see Corollaries~\ref{coro:puq-asy} and~\ref{coro:puq-asy-conic}.} { Remark~\ref{rem:superlinear} in Appendix~\ref{app:om} in the Electronic Companion shows that~\eqref{eq:uq-ot} is guaranteed to be solvable whenever the transportation cost~$d(\bm z,\hat{\bm z}_k)$ grows superlinearly in~$\bm z$ for every~$k\in\set K$.}



      If the transportation cost is set to~$d(\bm z,\bm z')=\|\bm z-\bm z'\|^p$ for an arbitrary norm $\|\cdot\|$ on~$\R^{d_{\bm z}}$ and constant~$p\in[1,+\infty)$, then $D (\bP, \bP')^{1/p}$ reduces to the $p$-th Wasserstein distance between~$\bP$ and~$\bP'$ \citep[Definition~6.1]{villani}. In this case, the ambiguity set~$\mathbb B_{\epsilon^p}(\hat \bP)$ coincides with the $p$-th Wasserstein ball of radius~$\epsilon$ around the nominal distribution~$\hat \bP$. Note also that~$d(\bm z,\bm z')=\|\bm z-\bm z'\|^p$ obeys assumption~\textbf{(D)} and satisfies the identity of indiscernibles.  Theorem~\ref{thm:strong-duality-all-problems-conic}~\emph{(i)} thus ensures that if the finite convex program~\eqref{eq:dot-cvx} admits a Slater point with $\lambda_{ik} > 0$ for all~$i\in\set I_k$ and~$k\in\set K$, then the supremum of~\eqref{eq:uq-ot} coincides with the minimum of the finite convex program~\eqref{eq:pot-cvx}. The transportation cost~$d(\bm z,\bm z')=\|\bm z-\bm z'\|^p$ impacts~\eqref{eq:pot-cvx} only through the partial conjugate
      \[
      \beta d^{*1}\left( \frac{\bm y_{ik}^{(1)}}{\beta}, \hat{\bm z}_k \right) = \hat{\bm z}_k^\top \bm y_{ik}^{(1)} +   \varphi(q) \beta \left\| \frac{\bm y_{ik}^{(1)}}{\beta} \right\|^q_*,
      \]
      where~$\|\cdot\|_*$ denotes the norm dual to~$\|\cdot\|$ on~$\R^{d_{\bm z}}$, $q\in[1,+\infty]$ is the unique constant with~$\frac{1}{p}+ \frac{1}{q}=1$ and~$\varphi(q)=(q-1)^{(q-1)}/q^q$; see~Lemma~\ref{lem:p-norm}~\emph{(ii)}. If $p=1$ and~$q=+\infty$, then $\varphi(q) \left\| \cdot \right\|^q_*$ must be interpreted as the indicator function of the closed unit ball around~$\bm 0$ with respect to~\mbox{$\|\cdot\|_*$}. The minimization problem~\eqref{eq:pot-cvx} thus significantly generalizes known convex reformulations of uncertainty quantification problems over $1$-Wasserstein balls developed by~\cite{med17}. By letting~$p$ tend to~$\infty$ in the finite convex programs~\eqref{eq:pot-cvx} and~\eqref{eq:dot-cvx-explicit} with transportation cost~$d(\bm z,\bm z')=\|\bm z-\bm z'\|^p$, one further recovers convex reformulations of uncertainty quantification problems over $\infty$-Wasserstein balls akin to those studied by~\citet{bss19}.


    { Example~\ref{ex:shaping_the_costs} in Appendix~\ref{app:om} in the Electronic Companion showcases how our general class of transportation costs allows to incorporate prior structural information into the uncertainty quantification problem~\eqref{eq:uq-ot}.}

      Similar to Section~\ref{sec:drco}, the results of this section can be directly applied to distributionally robust optimization problems over transport-based ambiguity sets, in which one seeks a decision~$\bm x$ from within a closed feasible region $\set X\subseteq \R^{d_{\bm x}}$ that minimizes the worst-case expectation of a decision-dependent disutility function $g(\bm x,\bm z)$ with respect to all distributions $\bP\in \mathbb{B}_\epsilon (\hat{\mathbb{P}})$. Indeed, \eqref{eq:pot-cvx} remains a finite convex program when~$\bm x$ is appended to the list of decision variables, provided that~$\set X$ is convex and that the disutility function satisfies~$g(\bm x,\bm z) = \max_{i \in \set I} g_i (\bm x, \bm z)$, where $g_i (\bm x, \bm z)$ is proper, convex and closed in~$\bm x$ and~$-g_i (\bm x, \bm z)$ is proper, convex and closed in~$\bm z$ for every~$i\in\set I$. 

      \textbf{Acknowledgements.} We are grateful to Melvyn Sim and Anthony Man-Cho So for inspiring discussions that motivated this paper, and we acknowledge the constructive comments of the anonymous review team that helped us improve the exposition. This research was supported by the Swiss National Science Foundation under the NCCR Automation, grant agreement~51NF40\_180545, as well as the Engineering and Physical Sciences Research Council under the grant EP/R045518/1. For the purpose of open access, the authors have applied a ‘Creative Commons Attribution (CC BY) licence to any Author Accepted Manuscript (AAM) version arising.

      \linespread{1}
      \small

      \bibliographystyle{plainnat}

      \linespread{1.5}
      \normalsize

      \newpage

      \appendix
    
      \section{Basic Concepts of Convex Analysis}\label{app_before_A}
      
            Throughout the paper we use the following key concepts of convex analysis. The {\it domain} of a function $f:\R^{d_{\bm x}}\rightarrow  \overline{\R}$ is defined as $\dom (f) = \{ \bm x \in \R^{d_{\bm x}} \ | \ f(\bm x ) < +\infty \}$. The {\it epigraph} of  $f$ is defined as~$\textnormal{epi} (f) = \{ (\bm x,\tau) \in \R^{d_{\bm x}} \times \R \ | \ f(\bm x ) \le \tau \}$. The function $f$ is {\it proper} if $f(\bm x) > -\infty$ for all $\bm x \in \R^{d_{\bm x}} $ and~$f(\bm x) < +\infty$ for at least one $\bm x\in \R^{d_{\bm x}}$, implying that $\dom(f) \ne \emptyset$. In addition, $f$ is {\it closed} if  $f$ is lower semicontinuous and \mbox{}either $f(\bm x) > -\infty$ for all $\bm x \in \R^{d_{\bm x}} $  or $f(\bm x) = -\infty$ for all $\bm x \in \R^{d_{\bm x}} $.
            
      We now define the notions of conjugate functions and perspective~\mbox{functions}.

      \begin{app_defi}[Conjugate Function] \label{def:conjugate}
      The conjugate of a function $f:\R^{d_{\bm x}}\rightarrow  \overline{\R}$ is the function $f^*:\R^{d_{\bm x}}\rightarrow  \overline{\R}$ defined through	$f^*({{\bm w}}) = \sup_{\bm x}\left\{{{\bm x}}^\top{{\bm w}} - f({{\bm x}})\right\}$. The conjugate $(f^* )^*$ of $f^*$ is called the biconjugate of $f$ and is abbreviated as $f^{**}$.
      \end{app_defi}

      The {\it indicator function} $\delta_{\set{X}}: \R^{d_{\bm x}} \rightarrow  \overline{\R}$ of a set $\set{X}\subseteq \mathbb{R}^{d_{\bm x}}$ is defined through $\delta_{\set{X}} ({\bm x}  ) = 0$ if~$\bm x\in \set{X}$ and $\delta_{\set{X}} ({\bm x}) = +\infty$ if $\bm x\notin \set{X}$. The {\it support function} $\delta^*_{\set{X}}: \R^{d_{\bm x}} \rightarrow \overline{\R}$ of a set $\set{X}\subseteq \mathbb{R}^{d_{\bm x}}$ is defined through $\delta_{\set{X}}^*({\bm w} )  = \sup_{\bm x \in \set{X}} \ \{\bm x^\top {\bm w} \}$. Note that the support function of $\set{X}$ coincides with the conjugate of the indicator function of $\set{X}$, which justifies our notation.

      \begin{app_defi}[Perspective Functions] \label{def:convex-perspective}
      The convex perspective of a proper, closed and convex function $f:\R^{d_{\bm x}}\rightarrow  \overline{\R}$  is the function $\underline{f}:\R^{d_{\bm x}} \times \R_+ \rightarrow \overline{\R}$ defined through $\underline{f}(\bm x, t) = tf(\bm x/t)$ if $t > 0$ and~$\underline{f}(\bm x, 0) = \delta_{\dom(f^*)}^*({\bm x} )$. Similarly, the concave perspective of a function $f$ for which $-f$ is proper, closed and convex is the function $\overline{f}:\R^{d_{\bm x}} \times \R_+ \rightarrow \overline{\R}$ defined through $\overline{f}(\bm x, t) = tf(\bm x/t)$ if~$t > 0$ and~$\overline{f}(\bm x, 0) = -\delta_{\dom((-f)^*)}^*({\bm x} )$.
      \end{app_defi}

      One can show that for $t > 0$, the epigraph of $\underline{f}(\cdot,t)$ coincides with the epigraph of $f$ multiplied by $t$. Moreover, the epigraph of $\underline{f}$ coincides with the closure of the cone generated by $\epi(f) \times \{1 \} \subseteq \R^{d_{\bm x}} \times \R$. Finally, our definitions of the convex and concave perspectives satisfy
      \begin{equation} \label{eq:convex-perspective}
      \underline{f}(\bm x, 0) = \displaystyle \underset{(\bm x', t')\rightarrow (\bm x, 0)}{\lim \inf} t'f\left( \bm x'/ t' \right)
      \qquad \text{and} \qquad
      \overline{f}(\bm x, 0) = \displaystyle \underset{(\bm x', t')\rightarrow (\bm x, 0)}{\lim \sup} t'f\left( \bm x'/ t' \right)
      \end{equation}
      for convex and concave $f$, respectively, and $\bm x \in \R^{d_{\bm x}}$ \citep[p.~67 and Theorem 13.3]{Rockafellar1970}. For ease of notation, we henceforth use $tf(\bm x / t )$ to denote both $\underline{f}(\bm x, t)$ and $\overline{f}(\bm x, t)$. The correct interpretation of $0 f(\bm x / 0 )$ will be clear from the context. Specifically, $0 f (\bm x / 0)$ should be interpreted as $\underline{f}(\bm x, 0)$ if $f$ is convex and as $\overline{f}(\bm x, 0)$ if $f$ is concave. This convention is justified in view of~\eqref{eq:convex-perspective}.

      By construction, the convex perspective of a proper, closed and convex function is guaranteed to be proper, closed and convex; see Proposition~\ref{prop:perspective-convex}. \rev{The next example shows that alternative constructions of the convex perspective that are sometimes adopted in the literature fail to be~closed.}

      \begin{app_ex}[Perspective Functions]
      Define $f:\mathbb R\rightarrow\overline{\mathbb R}$ through $f(x)=\delta_{\{x_0\}}(x)$ for some~$x_0 \neq 0$, and note that~$f$ is proper, closed and convex. An elementary calculation shows that the convex perspective of~$f$ is given by~$\underline f(x,t)=\delta_{\{tx_0\}}(x)$ for every~$x\in\mathbb R$ and~$t\ge 0$, which is also proper, closed and convex. Note, however, that~$\underline f(x,0)=\delta_{\{0\}}(x) \neq +\infty =\lim_{t\downarrow 0} tf(x/t)$, where the last equality holds because~$x_0\neq 0$. This example shows that if one were to define~$\underline f(x,0)=\lim_{t\downarrow 0} tf(x/t)$, as is sometimes done in the literature, then the resulting perspective would fail to be closed. As a second example, define $f:\mathbb R\rightarrow\overline{\mathbb R}$ through $f(x)=|x|$, and note that $f$ is again proper, closed and convex. The convex perspective of $f$ is given by~$\underline f(x,t)=|x|$ for every~$x\in\mathbb R$ and~$t\ge 0$, which is also proper, closed and convex. This example shows that if one were to define~$\underline f(x,0)=\delta_{\{0\}}(x)$, as is sometimes done in the literature, then the resulting perspective would fail to be closed.
      \end{app_ex}

      Throughout the paper we use the following terminology for optimization problems, which is in line with \citet{Rockafellar1970}. Any assignment of real values to the decision variables of an optimization problem is a solution. A solution is feasible in an optimization problem if it satisfies all the constraints and attains an objective value other than $+\infty$ ($-\infty$) in a minimization (maximization) problem; otherwise, it is infeasible. An optimization problem is feasible if it has at least one feasible solution. We refer to the feasible region of an optimization problem as the set containing all of its feasible solutions. A feasible optimization problem is solvable if its optimal value is attained by a feasible solution, whereas an infeasible optimization problem is solved by any (necessarily infeasible) solution. Whenever the domain of a variable in an optimization problem is omitted, it is understood to be the entire space (whose definition will be clear from the context).

	In Section~\ref{sec:rco} we study functions $f(\bm x, \bm z)$ with two arguments, where the first argument $\bm x$ represents a decision variable, while the second argument~$\bm z$ represents an exogenous uncertain parameter. Below we show how the notions of conjugates and perspectives are extended to such functions.

      \begin{app_defi}[Partial Conjugates] \label{def:partial-conjugate}
      The partial conjugate of a function $f:\R^{d_{\bm x}} \times \R^{d_{\bm z}} \rightarrow  \overline{\R}$ with respect to its first argument is the function $f^{*1}:\R^{d_{\bm x}} \times \R^{d_{\bm z}} \rightarrow  \overline{\R}$ defined through $f^{*1} (\bm{w}, \bm z)=\sup_{\bm x\in  \R^{d_{\bm x}}} \left\{\bm{w}^\top \bm x - f(\bm x, \bm z) \right\}$. Likewise, the partial conjugate of $f$ with respect to its second argument is the function $f^{*2}:\R^{d_{\bm x}} \times \R^{d_{\bm z}} \rightarrow  \overline{\R}$ defined  through $f^{*2} (\bm x, \bm y)=\sup_{\bm z\in  \R^{d_{\bm z }}} \left\{ \bm y^\top \bm z - f(\bm x, \bm z) \right\}.$
      \end{app_defi}

      \begin{app_defi}[Partial Perspectives] \label{def:partial-perspective}
      If $f:\R^{d_{\bm x}} \times \R^{d_{\bm z}} \rightarrow  \overline{\R}$ is proper, closed and convex in its first argument, then we define its convex partial perspective $\underline{f}:\R^{d_{\bm x}} \times \R_+ \times \R^{d_{\bm z}} \rightarrow  \overline{\R}$ through $\underline{f}(\bm x, t, \bm z) = t f(\bm x/t, \bm z)$ if $t > 0$ and $\underline{f}(\bm x, 0, \bm z) = \delta^*_{\dom(f^{*1}(\cdot, \bm z))}(\bm x)$. If $-f$ is proper, closed and convex in its second argument, then we define its concave partial perspective $\overline{f}:\R^{d_{\bm x}} \times \R^{d_{\bm z}} \times \R_+  \rightarrow  \overline{\R}$ through~$\overline{f}(\bm x, \bm z, t) = t f(\bm x, \bm z/t)$ if $t > 0$ and $\overline{f}(\bm x,  \bm z, 0) = -\delta^*_{\dom( (-f)^{*2}(\bm x, \cdot))}(\bm z)$.
      \end{app_defi}

      For ease of notation, throughout the paper we use $tf(\bm x / t , \bm z)$ and $tf(\bm x  , \bm z / t)$ to denote $\underline{f}(\bm x, t, \bm z)$ and $\overline{f}(\bm x, \bm z, t)$, respectively. The correct interpretation will always be clear from the context.

      Section~\ref{sec:extensions} makes extensive use of proper convex cones, which we define next.
      \begin{app_defi}[Proper Convex Cone] \label{def:proper-convex-cone} 
      A convex cone $\set C\subseteq \R^{d_{\set C}}$ is called proper if it is closed, solid (i.e., it has nonempty interior) and pointed (i.e., it contains no line).
      \end{app_defi}

      Section~\ref{sec:extensions} also uses a generalized Slater condition for sets represented by conic inequalities.
      
      \begin{app_defi}[Slater Condition for Sets] \label{def:slater-condition-for-sets}
      The vector~$\bm x^{\rm S}$ is a Slater point of the set $\set X$ represented by $\set X = \{ \bm x \in \R^{d_{\bm x}} \mid  \bm f_i(\bm x) \preceq_{\set C_i} \bm 0\;\;\forall i\in\set I, \; h_j(\bm x)=0\;\;\forall j\in\set J \}$, where $\set C_i$ is a proper convex cone for all $i\in \set I$, if {\emph{(i)}} 
      $\bm{x}^{\textnormal{S}} \in {\rm ri} (\dom ( \bm f_{i}))$ and $\bm{x}^{\textnormal{S}} \in {\rm ri} (\dom (h_j))$ for all $i$ and $j$; \emph{(ii)}~$\bm{x}^{\textnormal{S}} \in \mathcal{X}$; and \emph{(iii)}~$\bm f_i (\bm{x}^{\textnormal{S}}) \prec_{\set C_i} \bm 0$ for all $i \in \mathcal{I}$ with the exception of those for which $\bm f_i$ is affine and $\set C_i$ is the non-negative orthant. The Slater point $\bm{x}^{\textnormal{S}}$ is strict if $\bm f_i (\bm{x}^{\textnormal{S}}) \prec_{\set C_i} \bm 0$ for all~$i \in \mathcal{I}$.
      \end{app_defi}

      We close with an example that describes vector- and matrix-valued functions that are convex with respect to proper convex cones but have components that fail to be convex in the usual sense.

      \begin{app_ex}[$\set C$-Convex Functions]
      \label{ex:C-convex-functions}
      The set~$\mathbb S^n_+$ of all positive semidefinite matrices represents a proper convex cone in the space~$\mathbb S^n$ of symmetric $n\times n$-matrices, and its interior is given by the set~$\mathbb S^n_{++}$ of positive definite matrices. The matrix inversion~$\bm F:\mathbb S^n\rightarrow \mathbb S^n\cup\{+\bm \infty_{\mathbb S^n_+}\}$ defined through
      \[
      \bm F(\bm X)=\left\{ \begin{array}{ll}
      \bm X^{-1} & \text{if } \bm X\in \mathbb S^n_{++}  \\
      +\bm \infty_{\mathbb S^n_+} & \text{otherwise}
      \end{array} \right.
      \]
      is an example of an $\mathbb S^n_+$-convex function. To see this, note that~$\dom(\bm F)=\mathbb S^n_{++}$ is convex and that the $\mathbb S^n_+$-epigraph of~$\bm F$ can be represented as{\em
      \begin{align*}
      \text{epi}_{\mathbb S^n_+}(\bm F)&=\left\{ (\bm X,\bm Y)\in\mathbb S_{++}^n\times\mathbb S^n\left|\bm X^{-1}\preceq_{\mathbb S^n_+} \bm Y\right. \right\}=\left\{ (\bm X,\bm Y)\in\mathbb S_{++}^n\times\mathbb S^n\left|\begin{pmatrix}
      \bm X &\bm I_n \\ \bm I_n & \bm Y\end{pmatrix} \succeq_{\mathbb S^{2n}_+} \bm 0\right. \right\},
      \end{align*}}where~$\bm I_n$ stands for the identity matrix in~$\mathbb S^n$. The second equality in the above expression follows from a standard Schur complement argument. Thus, {\em $\text{epi}_{\mathbb S^n_+}(\bm F)$} is manifestly convex. Other~$\set C$-convex functions can be constructed as follows. If~$\set C\subseteq \mathbb R^{d_{\set C}}$ is a proper convex cone, $g(\bm x,\bm y)$ is a Borel-measurable function that is convex in~$\bm x$ for every fixed~$\bm y\in\set C$ and $\mu$ is a Borel measure on~$\set C$, then $\bm f(\bm x)=\int_{\set C} \bm y\cdot g(\bm x,\bm y) \,\mu({\rm d}\bm y)$ is $\set C$-convex (provided the integral exists) because $\bm \lambda^\top \bm f$ is convex for every~$\bm \lambda\in\set C^*\backslash\{\bm 0\}$. To see this, recall that~$\bm \lambda^\top\bm y\geq 0$ for all~$\bm y\in\set C$ and that convexity is preserved by integration against a non-negative weighting function \citep[\S~3.2.1]{bv04}.
      \end{app_ex}
      
      \section{Proofs}\label{app_A}

      \noindent \textbf{Proof of Theorem~\ref{prop:weak-duality}.} $\;$ If~\eqref{eq:p-co} or~\eqref{eq:d-co} is infeasible, the statement trivially holds. In the remainder of the proof, we thus assume that both~\eqref{eq:p-co} and~\eqref{eq:d-co} are feasible. Set $\mathcal C= \cap_{i\in\mathcal I_0} \dom(f_i)$, which is nonempty and convex by the feasibility of~\eqref{eq:p-co} and assumption~\textbf{(F)}, respectively. However, $\mathcal C$ is not necessarily closed because not every proper, closed and convex function has a closed domain. Next, define the Lagrangian $\mathscr{L}:\R^{d_{\bm x}}\times\R^I \rightarrow \overline{\R}$ associated with problem~\eqref{eq:p-co} through
      \begin{equation*}
      \mathscr{L}(\bm x, \bm \lambda)= \left\{ \begin{array}{cl}
      f_0(\bm x)+\sum_{i\in\mathcal I} \lambda_i f_i(\bm x) & \text{if }\bm x\in\mathcal C,~\bm \lambda\ge\bm 0,\\
      -\infty & \text{if }\bm x\in\mathcal C,~\bm \lambda\not\ge\bm 0,\\
      +\infty & \text{otherwise.}\end{array}\right.
      \end{equation*}
      As the objective and constraint functions of problem~\eqref{eq:p-co} are proper and convex by assumption~\textbf{(F)}, the Lagrangian $\mathscr{L}(\bm x, \bm \lambda)$ is proper and convex in $\bm x$ for every fixed $\bm \lambda\ge \bm 0$. As $\mathcal C$ may fail to be closed, however, $\mathscr{L}(\bm x, \bm \lambda)$ is not necessarily closed in $\bm x$ even if $\bm \lambda\ge \bm 0$. One also easily verifies that $-\mathscr{L}(\bm x, \bm \lambda)$ is proper, closed and convex in $\bm \lambda$ for every fixed $\bm x\in\mathcal C$.

      Using the Lagrangian, the primal problem~\eqref{eq:p-co} can be expressed as the min-max problem
      \begin{equation} \label{eq:primal-l}
      \inf_{\bm x\in\mathcal C} f(\bm x), \quad \text{where}\quad f(\bm x)=\sup_{\bm \lambda \geq \bm 0}  \mathscr{L}(\bm x, \bm \lambda)=\left\{\begin{array}{cl}
      f_0(\bm x) & \text{if } f_i(\bm x)\le 0 \;\; \forall i\in\mathcal I, \\ +\infty & \text{otherwise.}
      \end{array}\right.
      \end{equation}
      Below we will show that the dual problem~\eqref{eq:d-co} can be bounded above by the max-min problem
      \begin{equation} \label{eq:dual}
      \sup_{\bm \lambda \geq \bm 0} g(\bm \lambda), \quad \text{where}\quad  g(\bm \lambda)=\inf_{\bm x\in\mathcal C} \mathscr{L}(\bm x, \bm \lambda).
      \end{equation}
      The statement of the theorem then follows because
      \begin{equation}
      \label{eq:min-max}
      \inf\eqref{eq:p-co} \; = \; \inf_{\bm x\in\mathcal C} f(\bm x)
      \; = \;
      \inf_{\bm x\in\mathcal C} \sup_{\bm \lambda \geq \bm 0}  \mathscr{L}(\bm x, \bm \lambda)
      \; \ge \;
      \sup_{\bm \lambda \geq \bm 0}  \inf_{\bm x\in\mathcal C} \mathscr{L}(\bm x, \bm \lambda)
      \; = \;
      \sup_{\bm \lambda \geq \bm 0}  g(\bm \lambda)
      \; \ge\;
      \sup\eqref{eq:d-co},
      \end{equation}
      where the first inequality is  a direct consequence of the classical min-max inequality.
      %
      To see that the second inequality holds, fix any $\bm \lambda \geq \bm 0$ and note that
      \begin{align}
      g( \bm \lambda)	\; = \; & \inf_{\bm x} \left\{ f_0  (\bm x) +  \sum_{i\in \set{I}: \atop \lambda_i > 0} \lambda_i f_i  (\bm x) + \sum_{i\in \set{I}: \atop \lambda_i = 0}   \delta_{\dom (f_i)}  (\bm x)   \right\} \nonumber\\
      \; = \; &  - \sup_{\bm x}  \left\{ \bm 0 ^\top \bm x -  f_0  (\bm x) - \sum_{i\in \set{I}: \atop \lambda_i > 0} \lambda_i f_i  (\bm x) - \sum_{i\in \set{I}: \atop \lambda_i = 0}   \delta_{\dom (f_i)}  (\bm x)  \right\} \nonumber\\
      \; \ge \; &- \inf_{\{\bm w_i\}_{i\in \set{I}_0}}    \left\{ f_0^*  (\bm w_0)  + \sum_{i\in \set{I}:\atop   \lambda_i > 0}  (\lambda_i f_i)^* \left(   \bm w_i \right) +  \sum_{i\in \set{I}:\atop  \lambda_i = 0} \delta^*_{\dom (f_i)}  (\bm w_i)  \ \Big|  \ \sum_{i\in \set{I}_0} \bm w_i = \bm 0\ \right\} \nonumber\\
      \; = \;  &- \inf_{\{\bm w_i\}_{i\in \set{I}_0}}    \left\{ f_0^*  (\bm w_0)  + \sum_{i\in \set{I}:\atop   \lambda_i > 0} \lambda_i f_i^* \left(   \bm w_i/ \lambda_i \right) +  \sum_{i\in \set{I}:\atop  \lambda_i = 0} \delta^*_{\dom (f_i)}  (\bm w_i)  \ \Big|  \ \sum_{i\in \set{I}_0} \bm w_i = \bm 0\ \right\} \nonumber\\
      \; = \; & \sup_{\{\bm w_i\}_{i\in \set{I}_0}}   \left\{ - f_0^*  (\bm w_0) - \sum_{i\in \set{I}} \lambda_i f_i^* \left(  \bm w_i / \lambda_i \right) \ \Big|  \ \sum_{i\in \set{I}_0} \bm w_i = \bm 0\ \right\}, \label{eq:dual-explicit-proof}
      \end{align}
      where the first equality expresses $g(\bm \lambda)$ as the optimal value of an unconstrained minimization problem, in which any solution $\bm x\notin\mathcal C$ adopts an infinite objective value. The objective function of this minimization problem is proper because $\mathcal C\neq \emptyset$, and it is convex because the functions $f_i$, $i\in \set I_0$, are all convex. The inequality in the above expression follows from Proposition~\ref{prop:inf-conv}, which asserts that the conjugate of a sum of proper convex functions provides a lower bound on the infimal convolution of the conjugates of these functions. The third equality follows from Theorem 16.1 by \citet{Rockafellar1970}, which asserts that the conjugate of a positive multiple of a proper convex function equals the perspective of the conjugate of this function. The fourth equality holds due to our definition of the convex perspective and because $\delta^*_{\dom(f_i)} = \delta^*_{\dom(f_i^{**})}$ by virtue of Theorem~12.2 by \citet{Rockafellar1970}. Substituting the lower bound~\eqref{eq:dual-explicit-proof} on $g(\bm \lambda)$ into~\eqref{eq:dual} then shows that the dual problem~\eqref{eq:d-co} is indeed bounded above by~\eqref{eq:dual}.
      This observation completes the proof. \hfill \qed

      ~\\[-8mm]

      The proof of Theorem~\ref{prop:strong-duality} relies on the following auxiliary result.

      \begin{app_lem}\label{lem:dual_of_d_is_p}
      The Lagrangian dual of problem~\eqref{eq:d-co} is equivalent to problem~\eqref{eq:p-co}.
      \end{app_lem}

      \noindent \textbf{Proof of Lemma~\ref{lem:dual_of_d_is_p}.} $\;$
      We express~\eqref{eq:d-co} as the max-min problem
      \begin{equation} \label{eq:dual-explicit-l}
      \sup_{  \bm \lambda \geq \bm 0} \ \sup_{\{\bm w_i\}_{i \in \mathcal{I}_0} } \ \inf_{\bm x} \ -f_0^* \left( \bm w_0 \right)- \sum_{i\in \set{I}}  \lambda_i f_i^* \left( \bm w_i /  \lambda_i \right) + \sum_{i \in \set I_0} \bm w_i^\top \bm x,
      \end{equation}
      where the equality constraint involving $\{\bm w_i\}_{i \in \mathcal{I}_0}$ is enforced implicitly through the embedded minimization over the Lagrange multiplier $\bm{x}$. Interchanging the order of the maximization and minimization operators in~\eqref{eq:dual-explicit-l}, we obtain the standard Lagrangian dual of the dual problem~\eqref{eq:d-co}. 
      \begin{equation} \label{eq:dual-dual}
      \inf_{\bm x} \ \sup_{  \bm \lambda \geq \bm 0} \ \sup_{\{\bm w_i\}_{i \in \mathcal{I}_0} } \ -f_0^* \left( \bm w_0 \right)- \sum_{i\in \set{I}}  \lambda_i f_i^* \left( \bm w_i /  \lambda_i \right) + \sum_{i \in \set I_0} \bm w_i^\top \bm x
      \end{equation}
      For any fixed $\bm x$, the embedded maximization problems in~\eqref{eq:dual-dual} evaluate to~$f(\bm x)$ as defined in~\eqref{eq:primal-l}. Indeed, a direct calculation reveals that
      \begin{align*}
      & \sup_{ \bm \lambda \geq \bm 0} \ \sup_{ \{\bm w_i\}_{i \in \mathcal{I}_0}} \left\{ -f_0^* \left( \bm w_0 \right)- \sum_{i\in \set{I}}  \lambda_i f_i^* \left( \bm w_i /  \lambda_i \right) +  \sum_{i \in \set I_0} \bm w_i^\top \bm x  \right\} \nonumber\\
      = \ &  \sup_{ \bm \lambda \geq \bm 0} \  \Bigg\{\sup_{\bm w_0} \Bigg\{ \bm w_0^\top \bm x   -  f^{*}_0(\bm w_0)\Bigg\} + \sum_{i\in \set I: \atop \lambda_i > 0} \sup_{\bm w_i} \Bigg\{ \bm w_i^\top \bm x - \lambda_i f^{*}_i\left(\bm w_i / \lambda_i\right)\Bigg\} \nonumber \\
      & \mspace{243mu} + \sum_{i\in \set{I}: \atop \lambda_i = 0} \sup_{\bm w_i} \Bigg\{\bm w_i^\top \bm x - \delta^*_{\dom (f_i)}  (\{\bm w_i\}_i)  \Bigg\} \Bigg\} \nonumber\\
      = \ &   \sup_{\bm \lambda \geq \bm 0}  \left\{  f_0( \bm x) + \sum_{i\in \set I: \atop \lambda_i > 0}  \lambda_i f_i(\bm x ) + \sum_{i\in \set{I}: \atop \lambda_i = 0}   \delta_{{\rm cl} (\dom (f_i))}  (\bm x)  \right\}
      = \ f(\bm x). 
      \end{align*}
      Here, the first equality follows from a regrouping of terms, the definition of the convex perspective and Theorem~12.2 by \citet{Rockafellar1970}, which applies because $f_i$, $i\in \set I_0$, is proper, closed and convex by assumption \textbf{(F)}. The second equality exploits the fact that each inner maximization evaluates a conjugate. The explicit expressions for these conjugates follow from Theorems~12.2 and~13.2 and the remarks before Theorem~16.1 in the monograph by \citet{Rockafellar1970}, where we again use the fact that $f_i$, $i\in \set I_0$, is proper, closed and convex. The third equality, finally, follows from a case distinction. Thus, the Lagrangian dual of~\eqref{eq:d-co} is equivalent to~\eqref{eq:p-co}. 
      \hfill \qed

      ~\\[-8mm]

      \noindent \textbf{Proof of Theorem~\ref{prop:strong-duality}.} $\;$
      In view of assertion \emph{(i)}, assume first that~\eqref{eq:p-co} admits a Slater point. In this case~\eqref{eq:p-co} is feasible, and its infimum is strictly smaller than $+ \infty$. If the infimum of~\eqref{eq:p-co} evaluates to $-\infty$, then the supremum of~\eqref{eq:d-co} also amounts to $-\infty$ by weak duality (see Theorem~\ref{prop:weak-duality}), and~\eqref{eq:d-co} is solvable because any solution of an infeasible problem is optimal according to our convention. In the remainder we may thus assume that the infimum of~\eqref{eq:p-co} is finite. In this case, we will first show that both inequalities in~\eqref{eq:min-max} collapse to equalities, which implies that the duality gap between~\eqref{eq:p-co} and~\eqref{eq:d-co} vanishes. Indeed, the first inequality in~\eqref{eq:min-max} becomes tight due to Proposition~5.3.6 by~\citet{b09}, which applies because~\eqref{eq:p-co} has a finite infimum and admits a Slater point 
      \[
      \bm x^{\rm S} \in \bigcap_{i\in \set{I}_0} {\rm ri} (\dom (f_i)) =  {\rm ri} \Big( \bigcap_{i\in \set{I}_0} \dom (f_i) \Big),
      \]
      where the above equality holds due to Proposition~2.42 by \citet{rw09}. The second inequality in~\eqref{eq:min-max} becomes tight due to Proposition~\ref{prop:inf-conv} and because the existence of a Slater point guarantees that $\cap_{i\in \set{I}_0} {\rm ri} (\dom (f_i)) \ne \emptyset$. To establish the solvability of~\eqref{eq:d-co}, note that~\eqref{eq:dual} is solved by some $\bm \lambda^\star$ due to Proposition~5.3.6 by~\citet{b09} and that the parametric problem~\eqref{eq:dual-explicit-proof} for $\bm \lambda =\bm \lambda^*$ is solved by some $\{\bm w^\star_i\}_{i\in \set{I}_0}$ due to Proposition~\ref{prop:inf-conv} and because $\cap_{i\in \set{I}_0} {\rm ri} (\dom (f_i)) \ne \emptyset$. By construction, $(\bm \lambda^\star, \{\bm w^\star_i\}_{i\in \set{I}_0})$ thus constitutes an optimal solution for~\eqref{eq:d-co}.

      Assume now that~\eqref{eq:d-co} admits a Slater point. Similar arguments as in the previous paragraph show that strong duality and solvability of~\eqref{eq:p-co} trivially hold if the supremum of~\eqref{eq:d-co} evaluates to $+\infty$, and we may thus assume that the supremum of~\eqref{eq:d-co} is finite. Strong duality between~\eqref{eq:p-co} and~\eqref{eq:d-co} as well as the solvability of~\eqref{eq:p-co} then follow from Lemma~\ref{lem:dual_of_d_is_p} and Proposition~5.3.6 by~\citet{b09}, which applies because~\eqref{eq:d-co} has a finite supremum and admits a Slater point 
      that satisfies all explicit (linear) constraints and resides in the relative interior of the objective function.

      As for assertion \emph{(ii)}, assume first that the feasible region of~\eqref{eq:p-co} is nonempty and bounded. This ensures via assumption~\textbf{(F)} that the function~$f$ in~\eqref{eq:primal-l} is proper and has compact sublevel sets. Strong duality between~\eqref{eq:p-co} and~\eqref{eq:d-co} as well as solvability of~\eqref{eq:p-co} then follow from Lemma~\ref{lem:dual_of_d_is_p} as well as Proposition~5.5.4 by \citet{b09}.

      Finally, assume that the feasible region of~\eqref{eq:d-co} is nonempty and bounded, which ensures via our definition of the convex perspective and assumption~\textbf{(F)} that the (negative) optimal value function of the inner minimization problem of~\eqref{eq:dual-explicit-l} in the proof of Lemma~\ref{lem:dual_of_d_is_p}, which is given by
      \[
      -\inf_{\bm x} \ -f_0^* \left( \bm w_0 \right)- \sum_{i\in \set{I}}  \lambda_i f_i^* \left( \bm w_i /  \lambda_i \right) + \sum_{i \in \set I_0} \bm w_i^\top \bm x=
      \left\{ \begin{array}{cl} f_0^* \left( \bm w_0 \right)+ \sum_{i\in \set{I}}  \lambda_i f_i^* \left( \bm w_i /  \lambda_i \right) & \text{if } \sum_{i \in \set I_0} \bm w_i=\bm 0,\\
      +\infty & \text{otherwise,}
      \end{array}\right.
      \]
      is proper and has compact sublevel sets. Strong duality and solvability of~\eqref{eq:d-co} thus follow from Lemma~\ref{lem:dual_of_d_is_p} and Proposition~5.5.4 by \citet{b09}. 
      \hfill \qed

      ~\\[-8mm]

      \noindent \textbf{Proof of Proposition~\ref{prop:bounded-slater}} $\;$
      We construct a strict Slater point for~\eqref{eq:d-co} from \emph{(i)} a (possibly infeasible) solution $(\{\bm w^+_i\}_i, \bm \lambda^+)$ to~\eqref{eq:d-co} that resides in the relative interior of the domain of the objective function of~\eqref{eq:d-co} and \emph{(ii)} a point $(\{\bm w^-_i\}_i, \bm \lambda^-)$ that resides in (but not necessarily in the relative interior of) the domain of the objective function of~\eqref{eq:d-co} and that offsets any infeasibility of $(\{\bm w^+_i\}_i, \bm \lambda^+)$.

      By assumption~\textbf{(F)}, the function~$f_i$ is proper, and by Theorem~12.2 of \cite{Rockafellar1970}, its conjugate~$f_i^*$ inherits properness from $f_i$ for each $i \in \set I_0$. Thus, there exists $\bm w^+_i \in {\rm ri} ( \dom (f_i^*) )$ for every $i \in \set I_0$. Setting $\bm \lambda^+ = \bm 1>\bm 0$, it is then easy to verify that $( \{\bm w^+_i\}_i, \bm \lambda^+)$ resides within the relative interior of the domain of the objective function of~\eqref{eq:d-co}. However, because $\bm w=\sum_{i \in \set I} \bm w^+_i$ may differ from~$\bm 0$, the solution $( \{\bm w^+_i\}_i, \bm \lambda^+)$ may nevertheless be infeasible in~\eqref{eq:d-co}.

      To construct the point $(\{\bm w^-_i\}_i, \bm \lambda^-)$, we consider the following variant of~\eqref{eq:p-co}, where we add the linear term $\bm w^\top \bm x$ to the objective function with the fixed gradient $\bm w\in \R^{d_{\bm x}}$.
      \begin{equation} \label{eq:variant-primal}
      \begin{array}{c@{\quad}l@{\qquad}l}
      \displaystyle \inf & \multicolumn{2}{l}{\displaystyle \mspace{-8mu} f_0(\bm x) + \bm w^\top \bm x} \\
      \displaystyle \subj & \displaystyle f_i(\bm x) \le 0 & \forall i \in \set{I} \\
      & \displaystyle \bm x \text{ free}
      \end{array}
      \tag{P${}_{\bm w}$}
      \end{equation}
      As the conjugate of the new objective function $f_0(\bm x)+\bm w^\top \bm x$ evaluated at $\bm w_0$ amounts to $f_0^*(\bm w_0-\bm w)$, the variable substitution $\bm w_0\leftarrow \bm w_0- \bm w$ allows us to express the problem dual to~\eqref{eq:variant-primal} as
      \begin{equation} \label{eq:variant-dual}
      \begin{array}{c@{\quad}l@{\qquad}l}
      \displaystyle \sup & \multicolumn{2}{l}{\displaystyle \mspace{-8mu} -f_0^* \left( \bm w_0 \right)- \sum_{i\in \set{I}}  \lambda_i f_i^* \left( \bm w_i /  \lambda_i \right)} \\
      \displaystyle \subj & \displaystyle \sum_{i\in \set{I}_0} \bm w_i = - \bm w\\
      & \displaystyle \bm w_i \textnormal{ free} & \forall i \in \set{I}_0 \\
      & \displaystyle \bm \lambda \ge \bm 0. \\
      \end{array}
      \tag{D${}_{\bm w}$}
      \end{equation}
      By construction, \eqref{eq:variant-primal} and~\eqref{eq:p-co} share the same feasible region, which is nonempty and bounded by assumption, whereas~\eqref{eq:variant-dual} and~\eqref{eq:d-co} share the same objective function. Similar arguments as in the proof of Theorem~\ref{prop:strong-duality}~\emph{(ii)} thus imply that \eqref{eq:variant-primal} and~\eqref{eq:variant-dual} share the same (finite) optimal value, which in turn ensures that problem~\eqref{eq:variant-dual} admits a feasible solution $(\{\bm w^-_i\}_i, \bm \lambda^-)$. By construction, this solution resides within the domain of the common objective function of~\eqref{eq:variant-dual} and~\eqref{eq:d-co} but not necessarily within its relative interior.

      Next, define $(\bm w^{\rm S}_i,  \lambda^{\rm S}_i) = \frac{1}{2} (\bm w^+_i, \lambda^+_i) + \frac{1}{2} (\bm w^-_i, \lambda^-_i)$ for every $i\in \set I_0$. By the line segment principle of \citet[Proposition~1.3.1]{b09}, the constructed solution $(\{\bm w^{\rm S}_i\}_i,  \bm \lambda^{\rm S})$ belongs to the relative interior of the domain of the objective function of~\eqref{eq:d-co}. In addition, we have
      \[
      \sum_{i\in \set I} \bm w^{\rm S}_i
      \; = \;
      \frac{1}{2} \sum_{i\in \set I} \bm w^+_i + \frac{1}{2} \sum_{i\in \set I} \bm w^-_i
      \; = \;
      \frac{1}{2}\bm w-\frac{1}{2}\bm w
      \; = \;
      \bm 0
      \]
      and $\bm \lambda^{\rm S} > \bm 0$. Therefore, the solution $(\{\bm w^{\rm S}_i\}_i, \bm \lambda^{\rm S})$ constitutes a strict Slater point for~\eqref{eq:d-co}.
      \hfill \qed

      ~\\[-8mm]

      \noindent \textbf{Proof of Theorem~\ref{thm:primal-worst>dual-best-ro}.} $\;$
      For any fixed ${\bm z}_i\in \set Z$, $i\in\mathcal{I}_0$, the problems~\eqref{eq:pw-ro} and~\eqref{eq:db-ro} collapse to instances of~\eqref{eq:ps-ro} and~\eqref{eq:ds-ro}, respectively, and the following inequalities are due to Theorem~\ref{prop:weak-duality}.
      \begin{align*}
      \inf_{\bm x \in \set X({\bm z}_1, \ldots, {\bm z}_I)} f_0 (\bm x, {\bm z}_0)  & \ge \sup_{\sum_{i\in \set{I}_0} \bm w_i = \bm 0 \atop \bm \lambda \ge \bm 0} -   f_0^{*1}   \left(\bm w_0, {\bm z}_0 \right) - \sum_{i\in \set{I}}   \lambda_i f_i^{*1}   (\bm w_i/\lambda_i,{\bm z}_i) \\
      \implies \quad \sup_{ \{\bm z_i\}_{i \in \mathcal{I}_0} \subseteq \set Z } \inf_{\bm x \in \set X( \bm z_1, \ldots, \bm z_I)} f_0 (\bm x, \bm z_0)  & \ge  \sup_{ \{\bm z_i\}_{i \in \mathcal{I}_0} \subseteq \set Z }  \sup_{\sum_{i\in \set{I}_0} \bm w_i = \bm 0 \atop \bm \lambda \ge \bm 0} -   f_0^{*1}   \left(\bm w_0, \bm z_0 \right) - \sum_{i\in \set{I}}   \lambda_i f_i^{*1}   (\bm w_i/\lambda_i, \bm z_i),
      \end{align*}
      where $\set X(\bm z_1, \ldots, \bm z_I) = \left\{ \bm x \in \R^{d_{\bm x}} \ | \ f_i(\bm x, \bm z_i)\le 0 \;\; \forall i \in \set I  \right\}$. Note that the right-hand side of the second inequality is equivalent to~\eqref{eq:db-ro}, while left-hand side is upper bounded by~\eqref{eq:pw-ro} because
      \[
      \inf_{\bm x \in \set X( \bm z_1, \ldots, \bm z_I)} \sup_{ \{\bm z_i\}_{i \in \mathcal{I}_0} \subseteq \set Z } f_0 (\bm x, \bm z_0)
      \;\; \ge \;\;
      \sup_{ \{\bm z_i\}_{i \in \mathcal{I}_0} \subseteq \set Z } \inf_{\bm x \in \set X( \bm z_1, \ldots, \bm z_I)} f_0 (\bm x, \bm z_0)
      \]
      due to the min-max inequality.
      \hfill \qed

      ~\\[-8mm]

      \noindent \textbf{Proof of Proposition~\ref{prop:p-w=p-w-cvx-ro}.} $\;$
      To show that~\eqref{eq:pw-cvx-ro} upper bounds~\eqref{eq:pw-ro}, we dualize the embedded maximization problems in~\eqref{eq:pw-ro} that evaluate the worst-case uncertainty realizations in the objective and the constraint functions. Specifically, for any fixed $i\in \set{I}_0$ and $\bm x \in \R^{d_{\bm x}}$, we have
      \begin{subequations}
      \begin{align}  \label{eq:subproblem-inf}
      \sup_{\bm z_i \in \supp}  f_i(\bm x, \bm z_i) & = - \begin{cases}  \inf \hspace{0.5mm}&  - f_i(\bm x, \bm z_i)\\
      \subj & c_\ell(\bm z_i) \le 0  \qquad \forall \ell \in \set{L}\\
      & \bm z_{i} \textnormal{ free}
      \end{cases}   \\
      & \le - \begin{cases}
      \label{eq:subproblem-sup}
      \sup \  & \displaystyle	- (-   f_i)^{*2}\left(\bm x,  \bm y_{i0} \right)  - \sum_{\ell \in \set{L}} \nu_{i\ell} c_\ell^*  \left( \bm y_{i\ell}/\nu_{i\ell} \right)  \\
      \subj  & \displaystyle \sum_{\ell \in \set{L}_0} \bm   y_{i\ell} = \bm 0 \\
      &  \bm y_{i\ell} \textnormal{ free} \qquad \qquad \forall \ell \in\mathcal L_0\\
      &  \nu_{i \ell} \ge 0 \mspace{38mu} \qquad \forall \ell \in\mathcal L,
      \end{cases}
      \end{align}
      \end{subequations}
      where the inequality follows from Theorem~\ref{prop:weak-duality}, which applies because the assumptions \textbf{(RF)} and~\textbf{(C)} imply that~\eqref{eq:subproblem-inf} satisfies assumption \textbf{(F)} from Section \ref{sec:co}. Interchanging the minus sign and the supremum operator in~\eqref{eq:subproblem-sup} results in a minimization problem. Substituting the resulting minimization problem into~\eqref{eq:pw-ro} for every $i\in \set{I}_0$ and then merging the infimum operators in the objective and removing the infimum operators in the constraints yields~\eqref{eq:pw-cvx-ro}, and thus the infimum of~\eqref{eq:pw-ro} is indeed smaller or equal to that of~\eqref{eq:pw-cvx-ro}. Note that if the optimal solution of~\eqref{eq:subproblem-sup} is not attained for some $i\in \set{I}$, then removing the infimum operators in the constraints may lead to a further restriction of the problem and therefore result in a higher optimal value.

      As for assertion~\emph{(ii)}, assume that $\mathcal Z$ admits a Slater point $\bm z^{\rm S}$. As $\dom(-f_i(\bm x, \cdot))=\R^{d_{\bm x}}$ due to assumption~\textbf{(RF)}, $\bm z^{\rm S}$ is also a Slater point for the minimization problem in~\eqref{eq:subproblem-inf}. By Theorem~\ref{prop:strong-duality}~\emph{(i)}, the duality gap between~\eqref{eq:subproblem-inf} and~\eqref{eq:subproblem-sup} thus vanishes, and~\eqref{eq:subproblem-sup} is solvable. This implies that the infima of~\eqref{eq:pw-ro} and~\eqref{eq:pw-cvx-ro} coincide and that any optimizer of~\eqref{eq:pw-ro} can be combined with optimizers of the dual subproblems~\eqref{eq:subproblem-sup} for $i\in\mathcal{I}_0$ to construct an optimizer for~\eqref{eq:pw-cvx-ro}.

      As for assertion~\emph{(iii)}, assume finally that $\set Z$ is compact and that problem~\eqref{eq:pw-ro} admits a strict Slater point $\bm x^{\rm S}$. In this case, the functions $F_i(\bm x)=\sup_{\bm z_i \in \supp}  f_i(\bm x, \bm z_i)$, $i\in\mathcal I_0$, are convex and continuous in $\bm x$ by virtue of assumption~\textbf{(RF)}. Indeed, $F_i(\bm x)$ is convex and closed because~$f_i(\bm x, \bm z_i)$ is convex and closed in $\bm x$ for every fixed $\bm z_i$. Moreover, $F_i(\bm x)$ is finite for every fixed $\bm x$ due to Weierstrass' extreme value theorem, which applies because $\set Z$ is compact and $-f_i(\bm x, \bm z_i)$ is closed (and thus lower semicontinuous) in $\bm z_i$. As any convex function is continuous on the relative interior of its domain, we may thus conclude that each $F_i$ is continuous on $\mathbb{R}^{d_{\bm x}}$. By forming convex combinations with the strict Slater point~$\bm x^{\rm S}$, one can now use the continuity and convexity of the functions $F_i$, $i\in\mathcal I_0$, to prove that any $\bm x$ feasible in~\eqref{eq:pw-ro} can be represented as a limit of strict Slater points for~\eqref{eq:pw-ro}. Therefore,~\eqref{eq:pw-ro} is equivalent to
      \begin{equation}
      \label{eq:cvxo-general-strict}
      \begin{array}{l@{\quad}l@{\qquad}l}
      \inf  &  \displaystyle \sup_{\bm z_0 \in \supp}f_0(\bm x, \bm z_0) \\
      \subj &  \displaystyle  \sup_{\bm z_i \in \supp}  f_i  (\bm x, \bm z_i) < 0  & \forall i\in \set{I}\\
      &  \bm{x} \text{ free.}
      \end{array}
      \end{equation}
      We can now dualize the embedded maximization problems in~\eqref{eq:cvxo-general-strict} as in the proof of assertion~\emph{(i)}. The compactness of $\set Z$ implies via Theorem~\ref{prop:strong-duality}~\emph{(ii)} that the duality gap between~\eqref{eq:subproblem-inf} and~\eqref{eq:subproblem-sup} vanishes. The minimization problems resulting from interchanging the minus sign and the supremum operator in~\eqref{eq:subproblem-sup} can then be substituted back into~\eqref{eq:cvxo-general-strict}, the infimum operators in the objective can be merged, and the infimum operators in the constraints can be removed to obtain a variant of~\eqref{eq:pw-cvx-ro} with strict inequalities. Note that because the constraints in~\eqref{eq:cvxo-general-strict} are strict, the infimum operators in the constraints may indeed be removed without restricting the problem even if the corresponding subproblems are not solvable. Next, we argue that the strict inequalities in the resulting problem can again be relaxed to weak inequalities without changing the problem’s optimal value. By Remark~\ref{rem:strict-weak-inequalities}, this is the case if problem~\eqref{eq:pw-cvx-ro} admits a strict Slater point. Such a strict Slater point can be constructed by combining the strict Slater point $\bm x^{\rm S}$ of~\eqref{eq:pw-ro} with strict Slater points $(\{ \bm{y}_{i_\ell}^{\rm S}, \nu_{i_\ell}^{\rm S}\}_\ell )$ for the dual subproblems, $i\in I_0$, which exist thanks to Proposition~\ref{prop:bounded-slater}.
      Thus, the infima of~\eqref{eq:pw-ro} and~\eqref{eq:pw-cvx-ro} are indeed equal.

      Finally, to see that the solvability of~\eqref{eq:pw-cvx-ro} implies the solvability of~\eqref{eq:pw-ro}, assume that $(\bm x^\star, \{\bm y^\star_{i\ell}, \nu^\star_{i\ell} \}_{i, \ell})$ solves~\eqref{eq:pw-cvx-ro}. The above reasoning then implies that the optimal value of~\eqref{eq:pw-cvx-ro} amounts to~$F_0(\bm x^\star)$, which in turn shows that~$\bm x^\star$ solves~\eqref{eq:pw-ro}.
      \hfill \qed

      ~\\[-8mm]

      \noindent \textbf{Proof of Proposition~\ref{prop:d-b=d-b-cvx-ro}.} $\;$
      As for~\emph{(i)}, we prove that any feasible solution to~\eqref{eq:db-ro} corresponds to a feasible solution to~\eqref{eq:db-cvx-ro} with the same objective value. To this end, select any $(\{\bm w_i, \bm z_i\}_i, \bm \lambda)$ feasible in~\eqref{eq:db-ro} and define $\bm \upsilon_i=\lambda_i \bm z_i$ for $i\in\mathcal I$. We show that $( \{ \bm w_i\}_i,  {\bm z}_0, \bm \lambda, \left\{ {\bm \upsilon}_i \right\}_{i} )$ is feasible in~\eqref{eq:db-cvx-ro} and attains the same objective value. Indeed, it is clear that $\lambda_i c_\ell(\bm \upsilon_i/\lambda_i)\le 0$ for all $\ell \in\mathcal L$ and $i\in\mathcal I$ with $\lambda_i>0$. If $\lambda_i=0$ for some $i\in\mathcal I$, on the other hand, we have $\bm \upsilon_i=\bm 0$ and
      \[
      0 c_\ell(\bm 0 / 0)
      \; = \;
      \delta^*_{\dom(c^*_\ell)}(\bm 0)
      \; = \;
      0 \qquad \forall \ell \in \mathcal L,
      \]
      where the first equality follows from the definition of the convex perspective, while the second equality holds because $c^*_\ell$ inherits properness from $c_\ell$ \citep[Theorem~12.2]{Rockafellar1970} and because the support function of $\dom(c^*_\ell) \ne \emptyset$ vanishes at the origin. All other constraints of~\eqref{eq:db-cvx-ro} are trivially satisfied. Next, we show that the objective value of  $(\{\bm w_i, \bm z_i\}_i, \bm \lambda )$ in~\eqref{eq:db-ro} equals that of  $( \{ \bm w_i\}_i,  {\bm z}_0, \bm \lambda, \left\{ {\bm \upsilon}_i \right\}_{i} )$ in~\eqref{eq:db-cvx-ro}. Indeed, it is clear  that $ \lambda_i f_i^{*1}   (\bm w_i/\lambda_i,\bm \upsilon_i/ \lambda_i)= \lambda_i  f_i^{*1} (\bm w_i/ \lambda_i,\bm z_i)$ for all $i\in\mathcal I$ with $\lambda_i>0$. If $\lambda_i=0$ for some $i\in\mathcal I$, on the other hand, then $\bm \upsilon_i = \bm 0$, and the convex perspective function $ 0 f_i^{*1}  (\bm w_i/ 0,\bm 0/ 0)$ is defined as the support function of the domain of $(f_i^{*1})^*$. As~$(f_i^{*1})^*=(-f_i)^{*2}$ due to Proposition~\ref{prop:partial-conjugates}, we may thus conclude that
      \begin{align*}
      0 f_i^{*1}   (\bm w_i/0,\bm 0 /0)
      \; & = \;
      \delta^*_{\dom ( (-f_i)^{*2} )}(\bm w_i, \bm 0) \\
      \; & = \;
      \sup_{\bm x} \left\{ \bm w_i^\top \bm x \ | \ \exists \bm y \in \R^{d_{\bm z}} : (-f_i)^{*2} (\bm x, \bm y) < +\infty \right\}\\
      \; & = \;
      \delta_{\{\bm 0 \} }(\bm w_i) \\
      \; & = \;
      \delta^*_{\dom(f_i(\cdot, \bm z_i))}(\bm w_i)
      \; = \; 0 f_i^{*1} (\bm w_i/0,\bm z_i),
      \end{align*}
      where the third equality follows from \citet[Theorem~12.2]{Rockafellar1970}, which ensures that for any $\bm x \in \R^{d_{\bm x}}$ the partial conjugate $(-f_i)^{*2}(\bm x, \cdot )$ of the proper function $-f_i(\bm x, \cdot)$ is also proper. Thus, for any $\bm x\in \R^{d_{\bm x}}$, there exists $\bm y \in \R^{d_{\bm z}}$ with $(-f_i)^{*2}(\bm x, \bm y) < +\infty$, which implies that $\bm x$ is actually free, and the supremum evaluates to $+\infty$ unless $\bm w_i=\bm 0$. The fourth equality holds because $\dom(f_i(\cdot, \bm z_i)) = \R^{d_{\bm x}}$ for every $\bm z_i \in \R^{d_{\bm z}}$, and the last equality follows from the definition of the partial convex perspective and from Theorem~12.2 of \citet{Rockafellar1970}, which implies that $(f_i^{*1})^{*1}(\cdot, \bm z_i)=f_i(\cdot, \bm z_i)$. In summary, we have shown that the optimal value of~\eqref{eq:db-ro} does not exceed that of~\eqref{eq:db-cvx-ro}, and thus assertion \emph{(i)} follows.

      Assume now that $( \{ \bm w_i^{\textnormal{S}},   {\bm z}_i^{\textnormal{S}} \}_i,  \bm \lambda^{\textnormal{S}} )$ is a strict Slater point for problem~\eqref{eq:db-ro}, which implies that $\bm \lambda^{\textnormal{S}} > \bm 0$. In that case, $(  \{ \bm w_i^{\textnormal{S}}\}_i,  {\bm z}_0^{\textnormal{S}},  \bm \lambda^{\textnormal{S}},  \left\{ {\bm \upsilon}_i^{\textnormal{S}} \right\}_{i} )$ with  ${\bm \upsilon}_i^{\textnormal{S}} = \lambda_i^{\textnormal{S}} \cdot {\bm z}_i^{\textnormal{S}}$, $i \in \set I$, is a strict Slater point for problem~\eqref{eq:db-cvx-ro}. To prove assertion \emph{(ii)}, we show that any feasible solution to~\eqref{eq:db-cvx-ro} corresponds to a sequence of feasible solutions to~\eqref{eq:db-ro} that asymptotically attain a non-inferior objective value. This implies that the optimal value of~\eqref{eq:db-cvx-ro} is smaller or equal to that of~\eqref{eq:db-ro}, and together with assertion \emph{(i)} we can then conclude that the optimal values of~\eqref{eq:db-ro} and~\eqref{eq:db-cvx-ro} coincide.
      To this end, select any solution $( \{ \bm w_i\}_i,  {\bm z}_0, \bm \lambda,  \left\{ {\bm \upsilon}_i \right\}_{i} )$ feasible in~\eqref{eq:db-cvx-ro} and any $\epsilon>0$. As the feasible region of~\eqref{eq:db-cvx-ro} is convex and the objective function of~\eqref{eq:db-cvx-ro} is concave, there exists $\theta\in (0,1)$ such that the solution $(\{\bm w_i^\epsilon\}_i,  {\bm z}_0^\epsilon, \bm \lambda^\epsilon , \left\{ {\bm \upsilon}_i^\epsilon \right\}_{i})$ defined through
      \begin{equation}\label{eq:sequence_of_solutions}
      (\{\bm w_i^\epsilon\}_i,  {\bm z}_0^\epsilon, \bm \lambda^\epsilon , \left\{ {\bm \upsilon}_i^\epsilon \right\}_{i}) = \theta \cdot (\{\bm w_i\}_i,  {\bm z}_0, \bm \lambda, \left\{ {\bm \upsilon}_i \right\}_{i}) \; + \; (1 - \theta) \cdot (\{\bm w_i^{\textnormal{S}}\}_i,  {\bm z}_0^{\textnormal{S}}, \bm \lambda^{\textnormal{S}}, \left\{ {\bm \upsilon}_i^{\textnormal{S}} \right\}_{i})
      \end{equation}
      is feasible in~\eqref{eq:db-cvx-ro} and attains an objective function value that is at least as large as that of $(  \{ \bm w_i\}_i,  {\bm z}_0, \bm \lambda, \left\{ {\bm \upsilon}_i\right\}_{i} )$ minus $\epsilon$. Setting $\bm z^\epsilon_i = \bm \upsilon^\epsilon_i / \lambda^\epsilon_i$ for all $i\in \set{I}$, which is possible because $\bm \lambda^\epsilon>\bm 0$, it is clear that $( \{\bm w_i^\epsilon, \bm z_i^\epsilon\}_i, \bm \lambda^\epsilon)$ is feasible in~\eqref{eq:db-ro} and attains the same objective value as $( \{ \bm w_i^\epsilon\}_i,  {\bm z}_0^\epsilon,   \bm \lambda^\epsilon,  \left\{ {\bm \upsilon}_i^\epsilon \right\}_{i} )$ in~\eqref{eq:db-cvx-ro}.  As $( \{ \bm w_i\}_i,  {\bm z}_0, \bm \lambda,  \left\{ {\bm \upsilon}_i \right\}_{i} )$ and $\epsilon > 0 $ were chosen arbitrarily, the supremum of~\eqref{eq:db-ro} is thus at least as large as that of~\eqref{eq:db-cvx-ro}. Together with assertion \emph{(i)}, we thus conclude that the suprema of~\eqref{eq:db-ro} and~\eqref{eq:db-cvx-ro} coincide. Moreover, since our proof of assertion \emph{(i)} has shown that any feasible solution to~\eqref{eq:db-ro} corresponds to a feasible solution to~\eqref{eq:db-cvx-ro} with the same objective value, \eqref{eq:db-cvx-ro} is solvable whenever~\eqref{eq:db-ro} is solvable.

      Assume now that $\supp$ is bounded. To prove assertion \emph{(iii)}, we show that any feasible solution to~\eqref{eq:db-cvx-ro} corresponds to a feasible solution to~\eqref{eq:db-ro} with the same objective value. Together with assertion \emph{(i)}, this implies that the optimal values of~\eqref{eq:db-ro} and~\eqref{eq:db-cvx-ro} coincide. To this end, select any solution $( \{ \bm w_i\}_i,  {\bm z}_0, \bm \lambda,  \left\{ {\bm \upsilon}_i \right\}_{i} )$ feasible in~\eqref{eq:db-cvx-ro}, and define $\bm z_i = \bm \upsilon_i / \lambda_i$ if $\lambda_i > 0$ and  $\bm z_i = \bm z_0$ if $\lambda_i=0$, $i\in \set{I}$. Lemma~\ref{lemma:recession_directions}~\emph{(i)} implies that if $\lambda_i = 0$, then $\bm v_i$ must be a recession direction for the uncertainty set $\supp$. As $\supp$ is nonempty and bounded, this in turn implies that $\bm v_i = \bm 0$. Using the same reasoning as in the proof of assertion \emph{(i)}, one can thus show that $\lambda_i c_\ell(\bm \upsilon_i/\lambda_i) = \lambda_i c_\ell(\bm z_i)$ for all $\ell \in\mathcal L$ and $i\in\mathcal I$ and $ \lambda_i f_i^{*1}   (\bm w_i/\lambda_i,\bm \upsilon_i/ \lambda_i)= \lambda_i  f_i^{*1} (\bm w_i/ \lambda_i,\bm z_i)$ for all $i\in\mathcal I$. This implies that $(  \{\bm w_i, \bm z_i\}_i ,  \bm \lambda )$ is feasible in~\eqref{eq:db-ro} and attains the same objective value as $( \{ \bm w_i\}_i,  {\bm z}_0, \bm \lambda,  \left\{ {\bm \upsilon}_i \right\}_{i} )$ in~\eqref{eq:db-cvx-ro}. Note that our proof of assertion \emph{(i)} has shown that any feasible solution to~\eqref{eq:db-ro} corresponds to a feasible solution to~\eqref{eq:db-cvx-ro} with the same objective value, and our proof of assertion \emph{(iii)} has shown that any feasible solution to~\eqref{eq:db-cvx-ro} corresponds to a feasible solution to~\eqref{eq:db-ro} with the same objective value. Since the optimal values of both problems coincide, we can conclude that~\eqref{eq:db-ro} is solvable if and only if~\eqref{eq:db-cvx-ro} is solvable.
      \hfill \qed

      ~\\[-8mm]

      \noindent \textbf{Proof of Theorem~\ref{thm:p-w-cvx=d-b-cvx-ro}.} $\;$
      We show that~\eqref{eq:pw-cvx-ro} and~\eqref{eq:db-cvx-ro} can be viewed as instances of~\eqref{eq:p-co} and~\eqref{eq:d-co}, respectively. Assertions~\emph{(i)}, \emph{(ii)} and \emph{(iii)} can then be derived from Theorems~\ref{prop:weak-duality} and~\ref{prop:strong-duality}.
      For ease of exposition, we first rewrite the convex optimization problem~\eqref{eq:pw-cvx-ro} more concisely as
      \begin{equation} \label{eq:pw-cvx-compact}
      \begin{array}{l@{\quad}l@{\qquad}l}
      \inf  &   \displaystyle  \varphi_0 ( \bm x, \{ \bm y_{0 \ell}, \nu_{0 \ell} \}_\ell )  \\
      \subj &   \displaystyle  \varphi_i ( \bm x, \{ \bm y_{i \ell}, \nu_{i \ell} \}_\ell ) \le 0   & \forall i \in \set{I}\\
      &  \psi^+_{ik} (\{  y_{i \ell k} \}_\ell ) \le 0 & \forall i \in \set{I}_0,  \ \forall k\in \set K \\
      &  \psi^-_{ik} (\{ y_{i \ell k} \}_\ell ) \le  0 & \forall i \in \set{I}_0,  \ \forall k\in \set K \\
      &  \bm x, \, \bm y_{i\ell} \textnormal{ free}   & \forall i \in \set{I}_0, \; \forall \ell \in \mathcal L_0 \\
      &   \nu_{i\ell} \textnormal{ free}    & \forall i \in \set{I}_0, \; \forall \ell \in \mathcal L,
      \end{array}
      \end{equation}
      where the extended real-valued functions $\varphi_i$ for $i \in \set I_0$ are defined through
      \begin{align*}
      \varphi_i (\bm x, \{\bm y_{i\ell}, \nu_{i\ell}\}_\ell) & =  \left\{\begin{array}{c@{\quad}l}
      \displaystyle (- f_i)^{*2}\left(\bm x, \bm y_{i0}  \right)  +   \sum_{\ell \in \set{L}}   \nu_{i\ell}c_\ell^*   \left(\bm y_{i\ell} /  \nu_{i\ell}\right) & \text{if $\nu_{i\ell} \ge 0$,} \ \ell \in\mathcal L, \\
      +\infty & \text{otherwise,}
      \end{array} \right.
      \end{align*}
      the linear equalities in~\eqref{eq:pw-cvx-ro} are split into two sets of linear inequalities defined through
      \begin{align*}
      \psi^+_{ik} ( \{ y_{i\ell k}\}_\ell )  =  \sum_{\ell \in \set L_0}   y_{i\ell k} \quad \text{ and } \quad
      \psi^-_{ik} ( \{ y_{i\ell k}\}_\ell )  = - \sum_{\ell \in \set L_0}  y_{i\ell k}, \qquad \forall i\in \set I_0, \ \forall k\in \set K,
      \end{align*}
      and $y_{i\ell k}$ is the $k$-th element of the vector  $\bm y_{i\ell}$ for every $k\in \set K = \{1, \ldots, K\}$, where $K= d_{\bm z}$.

      Note that~\eqref{eq:pw-cvx-compact} can be viewed as an instance of~\eqref{eq:p-co}. Moreover, one can show that its objective and constraint functions satisfy assumption~\textbf{(F)}, that is, one can show that $\varphi_i$, $\bm \psi^+_i$ and $\bm \psi^-_i$ are proper, closed and convex for every $i\in\set I_0$. To see this, note first that the partial conjugate $(- f_i)^{*2}$ is proper, closed and convex by Proposition~\ref{prop:partial-conjugate-dual} and by Theorem~12.2 of \citet{Rockafellar1970}, which apply because $f_i$ obeys assumption~\textbf{(RF)}. Similarly, the convex perspective $\nu_{i\ell} c_\ell^*(\bm y_{i\ell} /  \nu_{i\ell})$ defined for~$\nu_{i\ell} \ge 0$ is proper, closed and convex by Proposition~\ref{prop:perspective-convex} and by Theorem~12.2 of \citet{Rockafellar1970}, which apply because $c_\ell$ obeys assumption \textbf{(C)}. Thus, $\varphi_i$ constitutes a sum of proper, closed and convex functions with different arguments and is therefore also proper, closed and convex.\footnote{The fact that the summands do not share common arguments is crucial here. The sum of the two proper, closed and convex functions $\delta_{[0, 1]} (x)$ and $\delta_{[2, 3]} (x)$ in the common argument $x \in \mathbb{R}$, for example, is not proper.} Finally, $\psi^+_{ik}$ and $\psi^-_{ik}$ are linear functions and therefore proper, closed and convex.

      If we interpret~\eqref{eq:pw-cvx-compact} as an instance of~\eqref{eq:p-co}, denote the variables conjugate to $\bm x$ and $\nu_{i\ell}$ by $\bm w_i$ and $u_{i\ell}$, respectively, and denote the variables conjugate to $\bm y_{i\ell}$ by $\bm r_{i\ell}, \bm r^+_{i\ell}$ and $\bm r^-_{i\ell}$, then the corresponding instance of~\eqref{eq:d-co} can be represented as
      \begin{equation} \label{eq:db-cvx-compact}
      \begin{array}{l@{\quad}l@{\qquad}l}
      \sup & \multicolumn{2}{l}{\displaystyle  \mspace{-8mu} - \varphi_0^* ( \bm w_0, \{ \bm r_{ 0 \ell}, u_{0 \ell} \}_\ell )  - \sum_{i\in \set I} \Bigg[ \lambda_i \varphi_i^* \left( \frac{ \bm w_i}{\lambda_i}, \frac{ \{ \bm r_{i \ell} \}_\ell}{\lambda_i}  ,  \frac{\{ u_{i \ell}\}_\ell}{\lambda_i}  \right) +} \\
      & \multicolumn{2}{l}{\displaystyle \mspace{245mu} \sum_{k\in \set K}  v_{ik}^+ (\psi^+_{ik})^* \left( \frac{\{  r^+_{i\ell k} \}_\ell}{v_{ik}^+} \right) + \sum_{k\in \set K}  v_{ik}^- (\psi^-_{ik})^* \left( \frac{\{ r^-_{i\ell k} \}_\ell}{v_{ik}^-} \right) \Bigg]} \\
      \subj &   \displaystyle  \sum_{i\in \set I_0} \bm w_i = \bm 0 & \forall \ell \in \set L_0 \\
      & \bm r_{i\ell} + \bm r^+_{i\ell} + \bm r^-_{i\ell}= \bm 0 & \forall i \in \set{I}_0, \; \forall \ell \in \set L_0 \\
      & u_{i\ell} = 0 & \forall i \in \set{I}_0, \; \forall \ell \in \set L \\
      & \bm w_i, \, r_{i\ell}, \, \bm r^+_{i\ell}, \, \bm r^-_{i\ell} \textnormal{ free} & \forall i \in \set{I}_0, \; \forall \ell \in \set L_0 \\
      & \bm \lambda, \, \bm v_{i}^+, \, \bm v_{i}^- \ge \bm 0 & \forall i \in \set{I}_0,
      \end{array}
      \end{equation}
      where $ r^+_{i\ell k}$ and $ r^-_{i\ell k}$ are the $k$-th elements of the respective  vectors $\bm  r^+_{i\ell}$ and $\bm  r^-_{i\ell}$ for every $k\in \set K$, and~$ (\bm \lambda, \{\bm v_{i}^+, \bm v_{i}^- \}_i)$ are the dual variables associated with the three sets of inequalities in~\eqref{eq:pw-cvx-compact}. Note that $u_{i\ell}$ is forced to $0$ in~\eqref{eq:db-cvx-compact} because its conjugate variable $\nu_{i\ell}$ only appears in the objective (if $i=0$) or in the $i$-th constraint (if $i\in\set I$) of the primal problem~\eqref{eq:pw-cvx-compact}.
      %
      %
      The conjugate of $\varphi_i$, $i \in \mathcal{I}_0$, can be calculated explicitly as
      \begin{align*}
      \varphi_i^* ( \bm w_i , \{ \bm r_{ i \ell}, u_{i \ell}  \}_\ell )
      \; & = \;
      \sup_{\bm x, \bm y_{i0}} \left\{  \bm w_i^\top \bm x + \bm r_{i0}^\top \bm y_{i0}   - (-f_i)^{*2}(\bm x, \bm y_{i0}) \right\}  \\
      & \qquad \qquad +  \sum_{\ell \in \set L } \sup_{\bm y_{i\ell} \atop \nu_{i\ell} > 0} \left\{  u_{i\ell}  \nu_{i\ell}  +   \bm r_{i\ell}^\top \bm y_{i\ell}  - \nu_{i\ell}c_\ell^*   \left( \frac{\bm y_{i\ell}}{\nu_{i\ell}}\right) \right\} \\
      \; & = \;
      f_i^{*1}(\bm w_i , \bm r_{i0} )  + \sum_{\ell \in \set L } \sup_{ \nu_{i\ell} > 0} \left\{  u_{i\ell}  \nu_{i\ell}  +   \nu_{i\ell} c_\ell \left(\bm r_{i\ell}  \right) \right\} \\
      \; & = \;
      \left\{\begin{array}{c@{\quad}l}
      f_i^{*1}(\bm w_i , \bm r_{i0} )  & \text{if }   u_{i\ell}  + c_\ell \left(\bm r_{i\ell}  \right) \le 0 \;\; \forall \ell \in \set L \\
      +\infty & \text{otherwise.}
      \end{array} \right.
      \end{align*}
      Note that we may restrict $\nu_{i\ell}$ to be strictly positive because the convex perspective of $c^*_\ell$ at $\nu_{i\ell}=0$ is defined as the lower semicontinuous extension of the perspective for $\nu_{i\ell}>0$; see~\eqref{eq:convex-perspective}. The second equality then follows from Proposition~\ref{prop:partial-conjugates}, which applies because $f_i$ satisfies assumption~\textbf{(RF)}, and from Theorem~16.1 of \citet{Rockafellar1970}, which applies because $c_\ell$ satisfies assumption~\textbf{(C)}.

      Similarly, the conjugates of $\psi^+_{ik}$ and $\psi^-_{ik}$ can be expressed as follows.
      \begin{align*}
      (\psi^+_{ik})^* ( \{ r^+_{i\ell k} \}_\ell )  & =  \left\{\begin{array}{c@{\quad}l}
      0 & \text{if }   r^+_{i\ell k} =  1 \;\; \forall \ell \in \set L_0 \\
      +\infty & \text{otherwise}
      \end{array} \right.\\
      (\psi^-_{ik})^* ( \{  r^-_{i\ell k } \}_\ell )  & =  \left\{\begin{array}{c@{\quad}l}
      0 & \text{if }   r^-_{i\ell k} =  -1 \;\; \forall \ell \in \set L_0 \\
      +\infty & \text{otherwise.}
      \end{array} \right.
      \end{align*}
      Substituting the formulas for $\varphi_i^*$, $(\psi^+_{ik})^*$ and $(\psi^-_{ik})^*$ into~\eqref{eq:db-cvx-compact} with $\bm z_0 = \bm v_{0}^+ - \bm v_{0}^-$ and $\bm v_{i} = \bm v_{i}^+ - \bm v_{i}^-$, $i\in \set I$, and eliminating the variables $\bm r_{i\ell}$, $\bm r^+_{i\ell}$, $\bm r^-_{i\ell}$ and $u_{i\ell}$,  $i\in \set I_0$ and $\ell \in \set L_0$, finally yields~\eqref{eq:db-cvx-ro}.
      \hfill \qed

      ~\\[-8mm]

      \textbf{Proof of Theorem~\ref{thm:p-w=d-b-ro}.} $\;$
      As for assertion \emph{(i)}, assume that~\eqref{eq:pw-ro} admits a strict Slater point $\bm x^{\rm S}$ and that the uncertainty set~$\supp$ is nonempty (by assumption) and compact (by assumption \textbf{(C)} and the assertion). Then the infima of~\eqref{eq:pw-ro} and~\eqref{eq:pw-cvx-ro} coincide due to Proposition~\ref{prop:p-w=p-w-cvx-ro}~\emph{(iii)}. Moreover, since $f_i$ is real-valued, the problem~\eqref{eq:subproblem-sup} admits a strict Slater point $( \{\bm y^{\rm S}_{i\ell}, \nu^{\rm S}_{i\ell}\}_\ell )$ for fixed $\bm x^{\rm S}$ and for every~$i\in \set I_0$ due to Proposition~\ref{prop:bounded-slater}. We can combine these strict Slater points to a Slater point for problem~\eqref{eq:pw-cvx-ro}, and  Theorem~\ref{thm:p-w-cvx=d-b-cvx-ro}~\emph{(ii)} implies that~\eqref{eq:pw-cvx-ro} and~\eqref{eq:db-cvx-ro} satisfy strong duality, and~\eqref{eq:db-cvx-ro} is solvable. The claim then follows from Proposition~\ref{prop:d-b=d-b-cvx-ro}~\emph{(iii)}, which ensures that the suprema of~\eqref{eq:db-ro} and~\eqref{eq:db-cvx-ro} coincide, and that~\eqref{eq:db-ro} is solvable because~\eqref{eq:db-cvx-ro} is solvable.

      As for assertion \emph{(ii)}, assume that the feasible region of~\eqref{eq:pw-ro} is nonempty and bounded and that~$\supp$ is bounded. Note that problem~\eqref{eq:pw-ro} can be represented more concisely as \begin{equation} \label{eq:primal-worst-compact}
      \inf_{\bm x} \left\{ F_0(\bm x) \ |\ F_i(\bm x) \le 0 \;\; \forall i \in \set{I}\right\},
      \end{equation}
      where $F_i(\bm x)=\sup_{\bm z_i \in \supp}  f_i  (\bm x, \bm z_i)$ constitutes a pointwise maximum of convex functions and is therefore convex. Moreover, as $\set{Z}$ is compact and $f_i(\bm x, \bm z_i)$ is continuous in~$\bm z_i$ for every~$\bm x$, $F_i(\bm x)$ is indeed finite for every~$\bm x$, \ie,  $\dom(F_i) = \R^{d_{\bm x}}$. As any finite-valued convex function is continuous,  we may thus conclude that $F_i$ is proper, closed and convex. The problem dual to~\eqref{eq:primal-worst-compact} can be expressed as
      \begin{equation} \label{eq:primal-worst-compact-dual}
      \max_{\bm \lambda\ge \bm 0, \bm w} \left\{  -F^*_0 \left( \bm w_0 \right) - \sum_{i \in \set{I}} \lambda_i F^*_i(\bm w_i/ \lambda_i) \ \bigg| \ \sum_{ i\in \set{I}_0}  \bm w_i = \bm 0 \right\}.
      \end{equation}
      As the feasible region of~\eqref{eq:pw-ro} is nonempty and bounded, Theorem~\ref{prop:strong-duality}~\emph{(ii)} ensures that strong duality holds and~\eqref{eq:pw-ro} is solvable, while Proposition~\ref{prop:bounded-slater} implies that~\eqref{eq:primal-worst-compact-dual} admits a strict Slater point $(\bm \lambda^{\rm S}, \{\bm w^{\rm S}_i\}_i)$. It remains to be shown that~\eqref{eq:primal-worst-compact-dual} is equivalent to the dual best problem~\eqref{eq:db-ro}. To this end, we will show that
      \begin{equation} \label{eq:primal-worst-perspective}
      \lambda_i F_i^*(\bm w_i/ \lambda_i) = \inf_{\bm z_i \in \set{Z}}   \lambda_i f_i^{*1}(\bm w_i/\lambda_i, \bm z_i)
      \end{equation}
      for any fixed $ \lambda_i \ge  0$ and $\bm w_i \in \R^{d_{\bm x}}$,  $i\in \set{I}$. Problem~\eqref{eq:db-ro} is then obtained by substituting~\eqref{eq:primal-worst-perspective} into~\eqref{eq:primal-worst-compact-dual}, and the assertion follows. To show~\eqref{eq:primal-worst-perspective}, assume first that $\lambda_i > 0$. We then have that
      \begin{align*}
      \lambda_i F_i^*(\bm w_i/ \lambda_i) &= \lambda_i \sup_{\bm x} \left\{ \bm x^\top   \bm w_i/\lambda_i - \sup_{\bm z_i \in \set{Z}} f_i(\bm x, \bm z_i) \right\}\\ 
      &=  \inf_{\bm z_i \in \set{Z}}   \lambda_i \sup_{\bm x} \left\{ \bm x^\top   \bm w_i/\lambda_i   - f_i(\bm x, \bm z_i) \right\} \\ 
      &= \inf_{\bm z_i \in \set{Z}}   \lambda_i f_i^{*1}(\bm w_i/\lambda_i, \bm z_i), \nonumber
      \end{align*}
      where the second equality follows from Sion's min-max theorem \citep{s58}, which applies because~$\supp$ is compact and $f_i$ is a convex-concave saddle function that is continuous in each of its arguments, while the last equality follows from the definition of the partial conjugate. If $\lambda_i = 0$, on the other hand, then we have
      \begin{equation*}
      0 F_i^*(\bm w_i/ 0)
      \;\; = \;\;
      \delta^*_{\dom (F_i)} (\bm w_i)
      \;\; = \;\;
      \delta_{\{0\}} (\bm w_i)
      \;\; = \;\;
      \inf_{\bm z_i \in \set{Z}}    \delta^*_{\dom (f_i(\cdot, \bm z_i))} (\bm w_i)
      \;\; = \;\;
      \inf_{\bm z_i \in \set{Z}}   0 f_i^{*1}(\bm w_i/0, \bm z_i),
      \end{equation*}
      where the first equality follows from the definition of the convex perspective, while the second equality holds because $\dom(F_i) = \R^{d_{\bm x}}$. Similarly, the third equality holds because $\dom(f_i(\cdot, \bm z_i)) = \R^{d_{\bm x}}$ for every $\bm z_i \in \R^{d_{\bm z}}$, while the last equality follows from the definition of the partial convex perspective. Thus,~\eqref{eq:primal-worst-perspective} holds for all $\bm \lambda \ge \bm 0$ and $\bm w_i \in \R^{d_{\bm x}}$, $i\in \set{I}$.

      As for assertion \emph{(iii)}, assume that~\eqref{eq:db-ro} admits a strict Slater point $(\{\bm w^{\rm S}_i, \bm z_i^{\rm S}\}_i,  \bm \lambda^{\rm S})$. Then the suprema of~\eqref{eq:db-ro} and~\eqref{eq:db-cvx-ro} coincide due to Proposition~\ref{prop:d-b=d-b-cvx-ro}~\emph{(ii)}. Moreover, it is easy to verify that $(  \{ \bm w_i^{\rm S}\}_{i},  {\bm z}_0^{\rm S}, \bm \lambda^{\rm S}, \left\{ {\bm \upsilon}_i^{\rm S} \right\}_{i})$ is a strict Slater point for~\eqref{eq:db-cvx-ro} where ${\bm \upsilon}_i^{\rm S} = {\bm z}_i^{\rm S} \cdot { \lambda}_i^{\rm S}$ for $i\in \set{I}$. Therefore, the problems~\eqref{eq:pw-cvx-ro} and~\eqref{eq:db-cvx-ro}  satisfy strong duality, and~\eqref{eq:pw-cvx-ro} is solvable due to  Theorem~\ref{thm:p-w-cvx=d-b-cvx-ro}~\emph{(ii)}. Finally, the infima of~\eqref{eq:pw-ro} and~\eqref{eq:pw-cvx-ro} coincide, and~\eqref{eq:pw-ro} is solvable as~\eqref{eq:pw-cvx-ro} is solvable due to Proposition~\ref{prop:p-w=p-w-cvx-ro}~\emph{(ii)}, which applies since any $\bm{z}_i^{\rm S}$ is a Slater point of $\supp$. The claim then follows.
      \hfill \qed

      ~\\[-8mm]

      \noindent \textbf{Proof of Theorem~\ref{coro:cvxdro-ro-wd}.} $\;$
      The statement trivially holds if either of the problems is infeasible. In the remainder of the proof we may thus assume that both~\eqref{eq:puq} and~\eqref{eq:duq} are feasible. Choose now an arbitrary~$\bP$ feasible in~\eqref{eq:puq} and an arbitrary $(\alpha, \bm{\beta})$ feasible in~\eqref{eq:duq}. As $\bP$ is feasible in~\eqref{eq:puq}, we have $\EP{} \left[ h_j(\bmt z) \right] < +\infty$ for every $j \in \set J$ and $\mathbb{E}_{\bP} [g(\bmt z)] > -\infty$. Thanks to our conventions for infinite integrals, this ensures that $\bP[\bmt z \in \dom(h_j)]=1$ for every $j\in \set J$ and $\bP[\bmt z \in \dom(-g)] = 1$, respectively. This implies that $\bP[\bmt z \in \bar{\set S}] = 1$. We then have
      $$
      \EP{} \left[ g(\bmt z) \right] \; \le \; \mathbb E_{\mathbb P} \left[ \alpha +  \bm h(\bmt z)^\top \bm \beta \right] \; \le \; \alpha  +   \bm \mu^\top \bm \beta,
      $$
      where the first inequality follows from the constraints in~\eqref{eq:duq} and our insight that $\bP[\bmt z \in \bar{\set S}] = 1$, and the second inequality follows from the constraints in~\eqref{eq:puq} and the nonnegativity of~$\bm \beta$. Thus, the objective value of $( \alpha, \bm{\beta})$ in~\eqref{eq:duq} is non-inferior to the objective value of $\bP$ in~\eqref{eq:puq}. As the primal and dual feasible solutions  $\bP$ and $(\alpha, \bm{\beta})$ were chosen arbitrarily, we may conclude that problem~\eqref{eq:duq} indeed provides an upper bound on~\eqref{eq:puq}.
      \hfill \qed

      ~\\[-8mm]

      \noindent \textbf{Proof of Proposition~\ref{prop:d-uq=ap-w}.} $\;$
      We show that the robust constraint in~\eqref{eq:duq} has the same feasible region as the $I$ robust constraints in~\eqref{eq:apw}. As $g(\bm z) = \max_{i \in \set I} g_i(\bm z)$, the robust constraint in~\eqref{eq:duq} is equivalent to
      \begin{align*}
      \quad  \max_{i\in \set I} \sup_{\bm z_i \in \bar{\set S}}  \left\{  g_i(\bm z_i) - \alpha - \bm h(\bm z_i)^\top \bm \beta  \right\} \le 0
      \quad & \iff  \quad  \max_{i\in \set I} \sup_{\bm z_i \in \bar{\set S}_i }  \left\{  g_i(\bm z_i) - \alpha - \bm h(\bm z_i)^\top \bm \beta  \right\} \le 0  \\
      & \iff  \quad  \max_{i\in \set I} \sup_{(\bm z_i, \bm u_i, t_i) \in  \set U_i}   \left\{  t_i - \alpha - \bm u_i^\top \bm \beta  \right\}  \le 0,
      \end{align*}
      where the first equivalence holds because $\bar{\set S}_i\subseteq \bar{\set S}$ and because $g_i(\bm z_i) = -\infty$ for all $\bm z\notin \bar{\set S}_i$. The second equivalence follows from~\eqref{eq:Z and U_i} and the fact that $\bm{u}_i = \bm{h} (\bm{z}_i)$ and $t_i = g_i (\bm{z}_i)$ maximize the inner supremum for any fixed admissible $\bm{z}_i$, $i \in \mathcal{I}$. The last inequality in the above expression is manifestly equivalent to the $I$ robust constraints in~\eqref{eq:apw}.
      \hfill \qed

      ~\\[-8mm]

      \noindent \textbf{Proof of Proposition~\ref{prop:fr=ad-b}.} $\;$
      As $\bar{\set S}_i$ represents the projection of~$\set U_i$ onto $\R^{d_{\bm z}}$, problem~\eqref{eq:adb} is equivalent to
      \begin{equation}
      \begin{aligned}  \label{eq:fr_fix_z}
      \sup  & \quad \displaystyle \sum_{i \in \set{I}} \lambda_i  g_i (\bm z_i)    \\
      \subj & \quad \displaystyle\sum_{i \in \set{I}} \lambda_i = 1 \\
      & \quad  \sum_{i \in \set{I}} \lambda_i \bm h (\bm z_i)  \le \bm \mu \\
      & \quad \bm z_i \in \bar{\set S}_i && \forall i \in \set I \\
      & \quad  \bm \lambda \ge \bm 0.
      \end{aligned}
      \end{equation}
      Indeed, the equivalence between \eqref{eq:adb} and \eqref{eq:fr_fix_z} holds due to~\eqref{eq:Z and U_i}, which ensures $\bm z_i \in \bar{\set S}_i$ if and only if there exist $\bm u_i \in \R^J$ and $t_i \in \R$ with $(\bm z_i, \bm u_i, t_i) \in \set U_i$, and because for any fixed $(\{\bm z_i\}_i, \bm \lambda)$ feasible in~\eqref{eq:fr_fix_z} it is optimal to set $\bm u_i = \bm h(\bm z_i)$ and $t_i= g_i(\bm z_i)$ for all $i\in \set I$.

      In the remainder of the proof we will show that~\eqref{eq:fr_fix_z} is equivalent to~\eqref{eq:fr}. As $g(\bm z_i) \ge g_i(\bm z_i)$ for all $\bm z_i \in \bar{\set S}_i$ and as $\bar{\set S_i} \subseteq \bar{\set S}$, it is clear that the optimal value of~\eqref{eq:fr} provides an upper bound on~\eqref{eq:fr_fix_z}. To prove the converse inequality, select any $(\{\bm z_i\}_i, \bm \lambda)$ feasible in~\eqref{eq:fr}, and define
      \[
      \hat{\set S}_i = \{ \bm z \in \bar{\set S}_i \ | \ g_i(\bm z) \ge g_{i'}(\bm z) \;\; \forall i' \in \set I \, : \, i' < i, \; g_i(\bm z) > g_{i'} (\bm z) \;\; \forall i' \in \set I \, : \, i' > i \}
      \]
      for all $i \in \set I$. Note that these sets form a partition of $\bar{\set S}$. By construction, we have $g_i(\bm z) = g(\bm z) > -\infty$ for all $\bm z \in \hat{\set S}_i$. Next, define~$\set I_i = \{ i' \in \set I \mid \bm z_{i'} \in \hat{\set S_i}\}$, and construct $(\{\bm z'_i \}_i, \bm \lambda')$ by setting~$\lambda'_i= \sum_{i' \in \set I_i} \lambda_{i'}$ and $\bm z'_i = \frac{1}{\lambda'_i} \sum_{i' \in \set I_i}  \lambda_{i'} \bm z_{i'}$ if $\lambda'_i > 0$. Otherwise, if $\lambda'_i = 0$, set~$\bm z_i'$ to an arbitrary point in $\bar{\set S}_i$, which is always possible because $\bar{\set S}_i$ is nonempty due to assumption~\textbf{(S)}. As  both~$-g_i$ and~$\bm h$ are proper, closed and convex thanks to assumptions~\textbf{(G)} and \textbf{(H)}, Jensen's inequality implies that
      \[
      \sum_{i' \in \set{I}_i} \lambda_{i'} g (\bm z_{i'})
      \; \le \;
      \lambda'_{i} g_{i} (\bm z_{i}')
      \quad \text{and} \quad
      \lambda'_{i} \bm h (\bm z_{i}')
      \; \le \;
      \sum_{i' \in \set{I}_i} \lambda_{i'} \bm h (\bm z_{i'})
      \; \le \;
      \lambda'_{i} \bm \mu \qquad \forall i \in \set I,
      \]
      that is,~$(\bm \lambda', \{\bm z'_i \}_i)$ is feasible in~\eqref{eq:fr_fix_z} and its objective value in~\eqref{eq:fr_fix_z} is larger or equal to that of~$(\bm \lambda,\{\bm z_i\}_i)$ in~\eqref{eq:fr}. Therefore, the optimal value of~\eqref{eq:fr} further provides a lower bound on the optimal value of~\eqref{eq:fr_fix_z}. The above arguments also reveal that one can construct a feasible solution for~\eqref{eq:fr} from a feasible solution of~\eqref{eq:adb} and vice versa. Hence, the claim follows.
      \hfill \qed

      ~\\[-8mm]

      \noindent \textbf{Proof of Theorem~\ref{thm:strong-duality-all-problems}.} $\;$
      In the absence of any regularity conditions, we have
      \begin{equation} \label{eq:overview}
      \inf\eqref{eq:apw-cvx}
      \; \ge \;
      \inf\eqref{eq:apw}
      \; = \;
      \inf\eqref{eq:duq}
      \; \ge \;
      \sup \eqref{eq:puq}
      \; \ge \;
      \sup\eqref{eq:fr}
      \; = \;
      \sup\eqref{eq:adb}
      \end{equation}
      where the two equalities follow from Propositions~\ref{prop:d-uq=ap-w} and~\ref{prop:fr=ad-b}, respectively, while the first inequality exploits Proposition~\ref{prop:p-w=p-w-cvx-ro}~\emph{(i)}, the second equality follows from the weak duality result established in Theorem~\ref{coro:cvxdro-ro-wd}, and the second inequality holds trivially because~\eqref{eq:fr} constitutes a restriction of~\eqref{eq:puq}. Proposition~\ref{prop:d-b=d-b-cvx-ro}~\emph{(i)} further implies that $\sup\eqref{eq:adb}\leq \sup\eqref{eq:adb-cvx}$. The relationships among the different problems are also summarized in Figure~\ref{fig:weak-relations-dro}. It remains to be shown that either of the conditions in assertions~\emph{(i)} or~\emph{(ii)} imply the equivalence of~\eqref{eq:adb} and~\eqref{eq:adb-cvx} as well as strong duality between~\eqref{eq:apw-cvx} and~\eqref{eq:adb-cvx}.

      As for assertion~\emph{(i)}, note first that the Slater point for~\eqref{eq:adb} can be used to construct a Slater point for~\eqref{eq:adb-cvx} with $\bm{\lambda} > \bm{0}$. The suprema of~\eqref{eq:adb} and~\eqref{eq:adb-cvx} then coincide thanks to Proposition~\ref{prop:d-b=d-b-cvx-ro}~\emph{(ii)} and Remark~\ref{rem:positive-lambda}.  Theorem~\ref{thm:p-w-cvx=d-b-cvx-ro}~\emph{(ii)} further guarantees that the infimum of~\eqref{eq:apw-cvx} coincides with the supremum of~\eqref{eq:adb-cvx} and that~\eqref{eq:apw-cvx} is solvable. This allows us to conclude that all problems in~\eqref{eq:overview} have the same optimal value as~\eqref{eq:adb-cvx}. As~\eqref{eq:adb} admits a Slater point, finally, it is clear that the augmented support set $\set U_i$ admits a Slater point for every~$i\in\set I$, and therefore Propositions~\ref{prop:p-w=p-w-cvx-ro}~\emph{(ii)} and~\ref{prop:d-uq=ap-w} ensure that if~$(\alpha^\star, \bm \beta^\star, \{\bm y_i^{(0)\star}\}_i, \{\bm y_{ij}^{(1)\star}\}_{ij},  \{ \bm y_{i\ell}^{(2)\star}, \nu_{i\ell}^\star \}_{i\ell})$ solves~\eqref{eq:apw-cvx}, then $(\alpha^\star, \bm \beta^\star)$ solves~\eqref{eq:duq}.

      As for assertion~\emph{(ii)}, note first that the infimum of~\eqref{eq:apw-cvx} coincides with the supremum of~\eqref{eq:adb-cvx} and that~\eqref{eq:adb-cvx} is solvable. This is an immediate consequence of Theorem~\ref{thm:p-w-cvx=d-b-cvx-ro}~\emph{(ii)}, which applies because~\eqref{eq:apw-cvx} admits a Slater point. To show that all problems in~\eqref{eq:overview} have the same optimal value, it thus remains to prove that the suprema of~\eqref{eq:adb} and~\eqref{eq:adb-cvx} coincide. To this end, fix any optimal solution~$(\bm \tau^\star, \bm \lambda^\star,\{\bm \omega^\star_i, \bm v^\star_i\}_i)$ of~\eqref{eq:adb-cvx}, and assume without loss of generality that $\bm \omega^\star_i = \lambda^\star_i \bm h ( \bm v^\star_i/ \lambda^\star_i)$ and $\tau^\star_i = \lambda^\star_i g_i (\bm v^\star_i / \lambda^\star_i)$ for all $i \in \set I$. We now show that this solution gives rise to an optimal solution $( \{ \bm z^\star_i, \bm u^\star_i, t^\star_i \}_i, \bm \lambda^\star)$ of~\eqref{eq:adb} that attains the same optimal value. To this end, set $\bm z^\star_i = \bm v^\star_i / \lambda^\star_i$ if $\lambda^\star_i > 0$, and let $\bm z^\star_i$ be an arbitrary point in $\bar{\set S}_i$ otherwise, $i\in \set I$. If there is~$i \in \set I$ with~$\lambda^\star_i = 0$, then Lemma~\ref{lemma:recession_directions}~\emph{(i)} implies that $\bm v^\star_i$ is a recession direction for~$\set S$. As~$\set S$ is nonempty and bounded, this in turn implies that $\bm v^\star_i = \bm 0$. Using the same reasoning as in the proof of Proposition~\ref{prop:d-b=d-b-cvx-ro}~\emph{(i)}, one can show that $\lambda^\star_i c_\ell (\bm v^\star_i/ \lambda^\star_i) = \lambda^\star_i c_\ell (\bm z^\star_i)$, $\lambda^\star_i \bm h ( \bm v^\star_i/ \lambda^\star_i) = \lambda^\star_i \bm h ( \bm z^\star_i) $ and $\lambda^\star_i g_i (\bm v^\star_i / \lambda^\star_i) = \lambda^\star_i g_i (\bm z^\star_i)$ for all $i\in \set I$ and $\ell \in \set L$. Setting $\bm u^\star_i = \bm h ( \bm z^\star_i)$ and $t^\star_i = g_i (\bm z^\star_i)$ for all $i\in \set I$, one readily verifies that $( \{ \bm z^\star_i, \bm u^\star_i, t^\star_i \}_i, \bm \lambda^\star)$ is feasible in~\eqref{eq:adb}. Moreover, since $\sum_{i \in \set I} \tau^\star_i = \sum_{i \in \set I} \lambda^\star_i t^\star_i$, this solution attains the same objective value as $(\bm \tau^\star, \bm \lambda^\star, \{\bm \omega^\star_i, \bm v^\star_i\}_i)$ in~\eqref{eq:adb-cvx}. Since~\eqref{eq:adb} bounds~\eqref{eq:adb-cvx} from below by Proposition~\ref{prop:d-b=d-b-cvx-ro}~\emph{(i)}, $( \{ \bm z^\star_i, \bm u^\star_i, t^\star_i \}_i, \bm \lambda^\star)$ must be optimal in~\eqref{eq:adb}. The proof of Proposition~\ref{prop:fr=ad-b} further implies that $(\{ \bm z^\star_i\}_i, \bm \lambda^\star)$ solves~\eqref{eq:fr}, which is a restriction of~\eqref{eq:puq}. As all problems in~\eqref{eq:overview} share the same optimal value, the discrete distribution that assigns probability $\lambda^\star_i$ to the point $\bm z^\star_i= \bm v_i^\star/\lambda_i^\star$ for all $i\in\set I$ with $\lambda_i^\star > 0$ indeed solves~\eqref{eq:puq}.
      \hfill \qed

      ~\\[-8mm]

      \noindent \textbf{Proof of Proposition~\ref{prop:slater-distribution}.} $\;$
      Denote by~$\bP^{\rm S}$ the Slater distribution of~\eqref{eq:puq} that exists by assumption. We will first argue that for each $i\in\set I$ there exists a probability $\lambda_i>0$ and a probability measure~$\bP^{\rm S}_i$ supported on $\bar{\set S}_i$ such that $\bP^{\rm S} =\sum_{i\in \set I} \lambda_i \bP^{\rm S}_i$. To see this, we define for every index set $\set I'\subseteq \set I$ the non-negative Borel measure $\rho_{\set I'}$ obtained by restricting~$\bP^{\rm S}$ to $\bar{\set S}_{\set I'}=\{\bm z\in\bar{\set S} \mid \bm z\in\bar{\set S}_{i'} ~\forall i'\in\set I', ~ \bm z\notin\bar{\set S}_{i'} ~\forall i'\in\set I\backslash \set I'\}$, that is, we set $\rho_{\set I'}[\set B]= \bP^{\rm S}[\bmt z\in \set B\cap \bar{\set S}_{\set I'}]$ for every Borel set $\set B\subseteq \mathbb R^{d_{\bm z}}$. By construction, we thus have $\bP^{\rm S}=\sum_{\set I'\subseteq \set I} \rho_{\set I'}$. Similarly, one readily verifies that $\sum_{\set I'\subseteq \set I:i\in\set I'} \rho_{\set I'} [\bar{\set S}_i]=\bP^{\rm S}[\bmt z\in\bar{\set S}_i] >0$, which implies that for all $i\in\set I$ there exists an index set $\set I'\subseteq \set I$ with $i\in\set I'$ and $\rho_{\set I'}[\bar{\set S}_i]>0$. Next, we define another family of non-negative Borel measures $\hat\rho_i=\sum_{\set I'\subseteq\set I: i\in\set I'} \frac{1}{|\set I'|} \rho_{\set I'}$, $i\in\set I$. Note that $\hat \rho_i$ is supported on $\bar {\set S}_i$ and satisfies $\hat \rho_i[\bar{\set S}_i]>0$ for all $i\in\set I$. In addition, we have $\bP^{\rm S} =\sum_{i\in\set I} \hat\rho_i$. Therefore, we can finally define $\lambda_i=\hat \rho_i[\bar{\set S}_i]$ and $\bP^{\rm S}_i=\frac{1}{\lambda_i}\hat \rho_i$ for all $i\in\set I$. As desired, this construction ensures that $\lambda_i>0$ and that the probability measure~$\bP^{\rm S}_i$ is supported on $\bar{\set S}_i$ such that $\bP^{\rm S} =\sum_{i\in \set I} \lambda_i \bP^{\rm S}_i$. Since~$\bP^{\rm S}$ is a probability measure, the last relation implies that $\sum_{i\in\set I}\lambda_i=1$. It is also clear that~$\bP^{\rm S}_i$ is absolutely continuous on~$\R^{d_{\bm z}}$ for every~$i\in\set I$. Next, define $\bm z_i= \mathbb{E}_{\bP^{\rm S}_i}[\bmt z]$ for all $i\in\set I$. As $h_j$ is proper, closed and convex, we may then use Jensen's inequality to verify that
      \[
      \sum_{i \in \set I} \lambda_{i} h_j (\bm z_{i})
      \; \leq \;
      \sum_{i \in \set I} \lambda_{i} \mathbb{E}_{\bP^{\rm S}_i} [h_j (\bmt z)]
      \; = \;
      \mathbb{E}_{\bP^{\rm S}} [h_j (\bmt z)]
      \; \leq \;
      \mu_j
      \]
      for every $j\in\set J$, where the last inequality is strict whenever~$h_j$ is nonlinear because $\bP^{\rm S}$ is a Slater distribution. Similarly, as $c_\ell$ is proper, closed and convex, Jensen's inequality implies that
      \[
      c_\ell \left(\sum_{i \in \set I} \lambda_{i}\bm z_{i}\right)
      \; \leq \;
      \sum_{i \in \set I} \lambda_{i} c_\ell (\bm z_{i})
      \; \leq \;
      \sum_{i \in \set I} \lambda_{i} \mathbb{E}_{\bP^{\rm S}_i} [c_\ell (\bmt z)]
      \; = \;
      \mathbb{E}_{\bP^{\rm S}} [c_\ell (\bmt z)]
      \; \leq \;
      0
      \]
      for every $\ell\in\set L$, where the last inequality is strict whenever~$c_\ell$ is nonlinear. We may thus conclude that $\sum_{i \in \set I} \lambda_{i}\bm z_{i}$ is a Slater point for the support set~$\set S$, which in turn implies via Lemma~\ref{lem:slater-points} that each $\bm z_i$, $i \in \mathcal{I}$, is a Slater point for~$\set S$ provided that $\bm z_i\in {\rm ri}(\set S)$. As $\bP^{\rm S}_i$ is absolutely continuous on~$\R^{d_{\bm z}}$ and supported on the convex set~$\bar{\set S}_i$, one can indeed prove that its mean~$\bm z_i$ belongs even to the interior of~$\bar{\set S}_i\subseteq \set S$. Otherwise, by the separating hyperplane theorem \cite[Section~2.5.1]{bv04}, there exist~$\bm a\in\R^{d_{\bm z}}$, $\bm a\neq \bm 0$, and~$b\in\R$ such that~$\bm a^\top\bm z_i\ge b$ and $\bm a^\top\bm z\le b$ for all~$\bm z\in \bar{\set S}_i$. These two inequalities imply via Theorem~1.6.6~(b) by \citet{ref:Ash-00} that~$\bP^{\rm S}_i[\bm a^\top \bmt z= b]=1$, which, however, contradicts the absolute continuity of~$\bP^{\rm S}_i$ on~$\R^{d_{\bm z}}$. We have thus shown that~$\bm z_i$ belongs to the interior of~$\bar{\set S}_i$ and, as a consequence, in particular to the interior of~${\rm dom}(-g_i)$, the interior of~${\rm dom}(h_j)$ for every~$j\in\set J$ and the interior of~$\set S$. By Lemma~\ref{lem:slater-points}, $\bm z_i$ is thus a Slater point for~$\set S$. As $\bm z_i\in\bar{\set S}_i$, we may finally select any $\bm u_i> \bm h (\bm z_i)\in\R^J$ and $t_i< g_i(\bm z_i)\in\R$ for every $i\in\set I$. By construction, $(\{\bm z_i, \bm u_i, t_i\}_i, \bm \lambda)$ constitutes a Slater point for~\eqref{eq:adb} that satisfies~$\bm \lambda > \bm{0}$. This Slater point for~\eqref{eq:adb} can easily be converted to a Slater point for~\eqref{eq:adb-cvx} that satisfies~$\bm \lambda > \bm{0}$.
      \hfill \qed

      ~\\[-8mm]

      \noindent \textbf{Proof of Proposition~\ref{prop:slater-distribution2}.} $\;$
      Denote by $(\alpha^{\rm S}, \bm \beta^{\rm S})$ a strict Slater point of problem~\eqref{eq:duq}, which exists by assumption. By using similar arguments as in Proposition~\ref{prop:d-uq=ap-w}, one can show that~$(\alpha^{\rm S}, \bm \beta^{\rm S})$ also constitutes a strict Slater point for~\eqref{eq:apw}. If we fix $\alpha = \alpha^{\rm S}$ and $\bm \beta = \bm \beta^{\rm S}$, then the embedded maximization problem in the~$i$-th constraint of~\eqref{eq:apw} is equivalent to
      \begin{equation}
      \begin{aligned} \label{eq:AP-W-Z_i}
      \sup  & \quad \displaystyle  g_i(\bm z_i) - \alpha^{\rm S} - \sum_{j\in \set J} h_j (\bm z_i)  \beta^{\rm S}_j  \\
      \subj & \quad \displaystyle c_\ell(\bm z_i) \le 0 \qquad \forall \ell \in \set L\\
      & \quad \bm z_i \text{ free}
      \end{aligned}
      \end{equation}
      because the strict inequality~$\bm \beta^{\rm S} > \bm 0$ implies that for every fixed $\bm z_i$ it is optimal to set $\bm u_i = \bm h(\bm z_i)$ and $t_i = g_i(\bm z_i)$. As $(\alpha^{\rm S},\bm \beta^{\rm S})$ is a strict Slater point for~\eqref{eq:apw}, the optimal value of~\eqref{eq:AP-W-Z_i} is strictly smaller than $0$. Note also that~\eqref{eq:AP-W-Z_i} can be viewed as an instance of the minimization problem~\eqref{eq:p-co} that satisfies assumption~\textbf{(F)} because the components~$g_i$, $i \in \set I$, of the disutility function satisfy~\textbf{(G)}, the moment functions $h_j$, $j\in \set J$, satisfy \textbf{(H)} and the constraint functions~$c_\ell$,~$\ell \in \set L$, of the support set satisfy \textbf{(C)}. The corresponding dual minimization problem is given by
      \begin{equation} \label{eq:AP-W-i}
      \begin{aligned}
      \inf \hspace{0.5mm}  & \quad \displaystyle(-g_i)^{*}\left(\bm y^{(0)}_{i}  \right) + \sum_{j\in \set{J}}  \beta_j^{\rm S}  h_j^* \left( \frac{\bm y^{(1)}_{ij}} { \beta_j^{\rm S}} \right) + \sum_{\ell \in \set{L}}  \nu_{i\ell} c_\ell^*    \left( \frac{\bm y^{(2)}_{i\ell} }{ \nu_{i\ell}} \right) - \alpha^{\rm S} \\
      \subj & \quad \displaystyle  \bm y^{(0)}_{i} + \sum_{j\in \set{J} }  \bm y^{(1)}_{ij}   + \sum_{\ell \in \set{L}} \bm   y^{(2)}_{i\ell} = \bm 0   \\
      & \quad   \bm y^{(0)}_{i}, \bm y^{(1)}_{ij}, \bm y^{(2)}_{i\ell} \textnormal{ free} , \quad   \nu_{i\ell} \ge 0 \qquad  \forall j\in \set{J}, \; \forall \ell \in \set{L}.
      \end{aligned}
      \end{equation}
      Note that the feasible region of the primal problem~\eqref{eq:AP-W-Z_i} coincides with $\bar{\set S}_i$ and is thus nonempty for every $i \in \set I$ thanks to assumption~\textbf{(S)}. In addition, it constitutes
      a subset of~$\set S$ and is thus bounded by assumption. Theorem~\ref{prop:strong-duality}~\emph{(ii)} then implies that problems~\eqref{eq:AP-W-Z_i} and~\eqref{eq:AP-W-i} share the same optimal value, which is strictly smaller than $0$. In addition, Proposition~\ref{prop:bounded-slater} implies that problem~\eqref{eq:AP-W-i} admits a strict Slater point $(\bm y^{(0){\rm S}}_{i}, \{\bm y^{(1){\rm S}}_{ij}\}_j, \{ \bm y^{(2){\rm S}}_{i\ell},\nu^{\rm S}_{i\ell} \}_\ell)$ for every $i \in \set I$. As the infimum of~\eqref{eq:AP-W-i} is strictly smaller than $0$, we may assume without loss of generality that the objective function value of this strict Slater point in~\eqref{eq:AP-W-i} is strictly negative, too. This is a direct consequence of Remark~\ref{rem:strict-weak-inequalities}. By construction, $(\alpha^{\rm S}, \bm \beta^{\rm S}, \{ \bm y^{(0){\rm S}}_{i}\}_{i}, \{\bm y^{(1){\rm S}}_{ij}\}_{ij}, \{\bm y^{(2){\rm S}}_{i\ell}, \nu^{\rm S}_{i\ell} \}_{i\ell})$ is thus a strict Slater point for~\eqref{eq:apw-cvx}.
      \hfill \qed

      ~\\[-8mm]

      \noindent \textbf{Proof of Corollary~\ref{coro:puq-asy}.} $\;$
      Note first that the suprema of~\eqref{eq:puq} and~\eqref{eq:adb-cvx} coincide by virtue of Theorem~\ref{thm:strong-duality-all-problems}~\emph{(i)}, which applies because~\eqref{eq:adb-cvx} admits a Slater point $(\bm \tau^\mathrm{S}, \bm \lambda^\mathrm{S}, \{ \bm \omega_i^\mathrm{S}, \bm v_i^\mathrm{S} \}_i)$ with $\bm \lambda^\mathrm{S} > \bm 0$.
      Next, select any tolerance~$\epsilon>0$ and any $\epsilon$-optimal solution $(\bm \tau^\epsilon, \bm \lambda^\epsilon, \{ \bm \omega^\epsilon_i, \bm v^\epsilon_i \}_i)$ of problem~\eqref{eq:adb-cvx}. If~\eqref{eq:adb-cvx} is unbounded, then we adopt the standard convention that $(\bm \tau^\epsilon, \bm \lambda^\epsilon, \{ \bm \omega^\epsilon_i, \bm v^\epsilon_i \}_i)$ is feasible in~\eqref{eq:adb-cvx} and that its objective function value is larger than or equal to~$1/\epsilon$. Next, define
      \[
      (\bm \tau^\theta, \bm \lambda^\theta, \{ \bm \omega^\theta_i, \bm v^\theta_i \}_i) = \theta \cdot (\bm \tau^\mathrm{S}, \bm \lambda^\mathrm{S}, \{ \bm \omega_i^\mathrm{S}, \bm v_i^\mathrm{S} \}_i) + (1-\theta) \cdot (\bm \tau^\epsilon, \bm \lambda^\epsilon, \{ \bm \omega^\epsilon_i, \bm v^\epsilon_i \}_i)
      \]
      for any $\theta \in [0,1]$, and note that this solution is feasible in~\eqref{eq:adb-cvx} as it constitutes a convex combination of two feasible solutions. Note also that $\bm \lambda^\theta>\bm 0$ whenever $\theta>0$. As the objective function of~\eqref{eq:adb-cvx} is linear and thus continuous, there exists $\theta\in(0,1]$ such that~$(\bm \tau^\theta, \bm \lambda^\theta, \{ \bm \omega^\theta_i, \bm v^\theta_i \}_i)$ represents a $2\epsilon$-optimal solution of~\eqref{eq:adb-cvx}. Next, fix such a~$\theta$, and define $\bP$ as the discrete distribution that assigns probability $\lambda_i^\theta>0$ to $\bm \upsilon_i^{\theta}/ \lambda_i^{\theta}$ for all $i\in\set I$. As~$(\bm \tau^\theta, \bm \lambda^\theta, \{ \bm \omega^\theta_i, \bm v^\theta_i \}_i)$ is feasible in~\eqref{eq:adb-cvx}, we can readily verify that $\bP$ is supported on~$\set S$ and satisfies
      \[
      \mathbb{E}_{\bP} [\bm h (\bmt z)]
      \; = \;
      \sum_{i \in \set I} \lambda^\theta_{i}\, \bm h(\bm \upsilon_i^{\theta} / \lambda_i^{\theta})
      \; \leq \;
      \sum_{i \in \set I} \bm \omega^\theta_i
      \; \leq \;
      \bm \mu,
      \]
      which implies that $\bP\in\amb{P}$. Similarly, the objective function value of~$\bP$ in~\eqref{eq:adb-cvx} satisfies
      \[
      \mathbb{E}_{\bP} [g (\bmt z)]
      \; = \;
      \sum_{i \in \set I} \lambda^\theta_{i}\, g (\bm \upsilon_i^{\theta} / \lambda_i^{\theta})
      \; \geq \;
      \sum_{i \in \set I} \lambda^\theta_{i}\, g_i (\bm \upsilon_i^{\theta}/ \lambda_i^{\theta})
      \; \geq \;
      \sum_{i \in \set I} \tau^\theta_i. \]
      The last expression non-inferior to~$\sup\eqref{eq:adb-cvx} -2\epsilon$ if the supremum of~\eqref{eq:adb-cvx} is finite and non-inferior to~$1/2\epsilon$ otherwise. As the suprema of~\eqref{eq:puq} and~\eqref{eq:adb-cvx} match, the above reasoning implies that~$\bP$ constitutes a $2\epsilon$-optimal solution of the original uncertainty quantification problem~\eqref{eq:puq}. As~$\epsilon > 0$ was chosen arbitrarily, we can thus construct feasible discrete distributions with $I$ atoms whose objective function values are arbitrarily close to the supremum of~\eqref{eq:puq}. 
      \hfill \qed

      ~\\[-8mm]

      \noindent \textbf{Proof of Lemma~\ref{lem:scalarization}.} $\;$
      For every~$\bm \lambda \in \set C^*\backslash\{\bm 0\}$ we have that~$\dom(\bm \lambda^\top\bm f)=\dom(\bm f)$ and $\bm f(\bm x)\succ_{\set C}-\bm \infty_{\set C}$ if and only if $\bm\lambda^\top\bm f(\bm x)>-\infty$. This implies that $\bm f$ is proper if and only if~$\bm \lambda^\top \bm f$ is proper for every~$\bm \lambda \in \set C^*\backslash\{\bm 0\}$. Also, it implies that $\dom(\bm f)$ is convex if and only if $\dom(\bm \lambda^\top \bm f)$ is convex for every~$\bm \lambda \in \set C^*\backslash\{\bm 0\}$. Next, select any $\bm x,\bm x'\in {\rm dom(}\bm f)$ and $\theta \in [0,1]$. By the definition of the dual cone $\set C^*$, we then have
      \begin{align*}
      & \theta \bm f( \bm x ) +  (1 - \theta)  \bm f( \bm x') - \bm f(\theta \bm x + (1 - \theta) \bm x') \in \set C \\
      \iff \quad & \theta  \bm \lambda^\top  \bm f( \bm x ) +  (1 - \theta)  \bm \lambda^\top  \bm f( \bm x')- \bm \lambda^\top \bm f(\theta \bm x + (1 - \theta) \bm x') \ge 0 \quad \forall \bm \lambda \in \set C^*\backslash\{\bm 0\},
      \end{align*}
      where the reverse implication holds because $\mathcal{C}$ is proper and convex, which implies that $\mathcal{C}^{**} = \mathcal{C}$. Thus,~$\bm f$ is $\set C$-convex if and only if the scalarized function~$\bm \lambda^\top \bm f$ is convex for every $\bm \lambda \in \set C^*\backslash\{\bm 0\}$.
      \hfill \qed

      ~\\[-8mm]

      \noindent \textbf{Proof of Proposition~\ref{prop:C-convex-perspective}.} $\;$
      Assume first that~$\bm 0\in\dom (\bm f)$. As~$\bm f$ is proper, star $\set C$-lower semicontinuous and $\set C$-convex, Lemma~\ref{lem:scalarization} implies that~$\bm \lambda^\top \bm f$ is proper, closed and convex for all $\bm \lambda \in \set C^*\backslash\{\bm 0\}$. As~$\bm 0\in\dom (\bm \lambda^\top\bm f)$, Corollary~8.5.2 and Theorem~13.3 by~\citet{Rockafellar1970} imply
      \begin{equation}
      \label{eq:scalarized-limit}
      \lim_{t\downarrow 0}\, t \bm \lambda^\top \bm f(\bm x/t) \; = \; \delta^*_{\dom((\bm \lambda^\top\bm f)^*)}(\bm x)
      \end{equation}
      for all~$\bm \lambda\in \set C^*\backslash\{\bm 0\}$. Note then that the dual cone~$\set C^*$ inherits properness from~$\set C$ \citep[Corollary~1.4.1]{bn13}. This means that~$\set C^*$ is solid and thus contains a basis~$\bm \lambda_1,\ldots, \bm \lambda_{d_{\set C}}$ of~$\R^{d_{\set C}}$. Defining the invertible matrix~$\bm \Lambda=(\bm \lambda_1,\ldots, \bm \lambda_{d_{\set C}})^\top \in \R^{{d_{\set C}}\times {d_{\set C}}}$, we conclude from~\eqref{eq:scalarized-limit} that
      \[
      \lim_{t\downarrow 0}\, t \bm \Lambda \bm f(\bm x/t) = \bm b(\bm x),\quad \text{where}\quad \bm b(\bm x)= \left(\delta^*_{\dom((\bm \lambda_1^\top\bm f)^*)}(\bm x),\ldots, \delta^*_{\dom((\bm \lambda_{d_{\set C}}^\top\bm f)^*)}(\bm x)\right)^\top \in \R^{d_{\set C}},
      \]
      which ensures that~$\lim_{t\downarrow 0} t \bm f(\bm x/t) = \bm \Lambda^{-1} \bm b(\bm x)$ exists. We may thus define the function~$\underline{\bm f}$ through
      \[
      \underline{\bm f}(\bm x, t)=\left\{\begin{array}{ll}
      t \bm f(\bm x/ t) & \text{if }t>0,\\
      \lim_{t\downarrow 0} t \bm f(\bm x/t) & \text{if }t=0.
      \end{array}\right.
      \]
      By construction, we have~$\bm\lambda^\top\underline{\bm f}(\bm x, 0)= \lim_{t\downarrow 0} t \bm\lambda^\top \bm f(\bm x/t)=\delta^*_{\dom((\bm \lambda^\top\bm f)^*)}(\bm x)$ for every~$\bm \lambda \in \set C^*\backslash\{\bm 0\}$, and thus~$\underline{\bm f}$ satisfies property~\emph{(iii)}. This in turn implies that~$\bm\lambda^\top\underline{\bm f}$ coincides with the convex perspective of~$\bm \lambda^\top\bm f$ for every~$\bm \lambda \in \set C^*\backslash\{\bm 0\}$, which is proper, closed and convex by Proposition~\ref{prop:perspective-convex}. Hence, $\underline{\bm f}$ is star $\set C$-lower semicontinuous by definition as well as proper and $\set C$-convex by Lemma~\ref{lem:scalarization}. The function $\underline{\bm f}$ consequently satisfies property~\emph{(i)}. Property~\emph{(ii)} holds by construction.

      If~$\bm 0\notin\dom(\bm f)$, then Corollary~8.5.2 by~\citet{Rockafellar1970} is no longer applicable. As~$\bm f$ is proper by assumption, however, there exists some point~$\bm x_0\in\dom(\bm f)$. Next, define~$\bm g:\mathbb R^{d_{\bm x}} \rightarrow \overline\R{}^{d_{\set C}}$ through $\bm g(\bm x)=\bm f(\bm x-\bm x_0)$ for all~$\bm x\in\mathbb R^{d_{\bm x}}$, and note that~$\bm g$ is proper, closed and convex and that~$\bm 0\in\dom(\bm g)$. By the first part of the proof, we may thus conclude that there exists a function~$\underline{\bm g}: \R^{d_{\bm x}}\times\R_+ \rightarrow \overline \R{}^{d_{\set C}}$ that satisfies properties~\emph{(i)}--\emph{(iii)}. Next, define the function~$\underline{\bm f}$ through~$\underline{\bm f}(\bm x,t)=\underline{\bm g}(\bm x+t\bm x_0,t)$. It is clear that~$\underline{\bm f}$ inherits perperness, star $\set C$-lower semicontinuity and $\set C$-convexity from~$\underline{\bm g}$ and thus satisfies property~\emph{(i)}. By construction, we further have for every~$t>0$ that
      \[
      \underline{\bm f}(\bm x,t0) \; = \; \underline{\bm g}(\bm x+t\bm x_0,t) \; = \; t \bm g(\bm x/t+\bm x_0) \; = \; t \bm f(\bm x/t),
      \]
      where the second equality holds because~$\underline{\bm g}$ satisfies property~\emph{(ii)}, and the third equality follows from the definition of~$\bm g$. This shows that~$\underline{\bm f}$ satisfies property~\emph{(ii)}. Finally, we also have
      \[
      \bm \lambda^\top\underline{\bm f}(\bm x,0) \; = \; \bm \lambda^\top\underline{\bm g}(\bm x,0) \; = \; \delta^*_{\dom((\bm \lambda^\top\bm g)^*)}(\bm x) \; = \; \delta^*_{\dom((\bm \lambda^\top\bm f)^*)}(\bm x),
      \]
      for every~$\bm \lambda \in \set C^*\backslash\{\bm 0\}$, where the second equality holds because~$\underline{\bm g}$ satisfies property~\emph{(iii)}, and the third equality follows from the observation that~$(\bm \lambda^\top\bm g)^*(\bm w)= (\bm \lambda^\top\bm f)^*(\bm w)- \bm w^\top\bm x_0$ for all~$\bm w\in\mathbb R^{d_{\bm x}}$. This reasoning shows that~$\underline{\bm f}$ also satisfies property~\emph{(iii)}. The uniqueness of~$\underline{\bm f}$ is a direct consequence of property~\emph{(iii)} and the observation that the proper cone~$\set C^*$ is solid.
      \hfill \qed

      ~\\[-8mm]

      \noindent \textbf{Proof of Proposition~\ref{prop:solvability-of-POT-convex}.} $\;$
      We first show that the negative objective function of problem~\eqref{eq:dot-cvx-explicit} is proper, closed and convex. Indeed, the convex perspective $-\lambda_{ik} g_i(  \hat{\bm z}_k + \bm v_{ik}/\lambda_{ik})$ defined for~$\lambda_{ik} \ge 0$ is proper, closed and convex by Proposition~\ref{prop:perspective-convex} and by Theorem~12.2 of \citet{Rockafellar1970}, which apply because~$g_i$ obeys assumption \textbf{(G)}. Thus, the negative objective function of~\eqref{eq:dot-cvx-explicit} constitutes a sum of proper, closed and convex functions with different arguments and is therefore also proper, closed and convex. As~$c_\ell$ obeys assumption \textbf{(C)} and $d$ obeys assumption \textbf{(D)}, similar arguments can be used to show that
      the feasible region of~\eqref{eq:dot-cvx-explicit} is closed. To prove that the feasible region is also bounded, note that the first constraint group in~\eqref{eq:dot-cvx-explicit} forces the non-negative variables $\{\lambda_{ik}\}_i$ to reside within a bounded simplex for every~$k\in \set K$. In addition, by Lemma~\ref{lem:growth} there exists a constant~$\delta>0$ such that $d(\hat{\bm z}_k+\bm z, \hat{\bm z}_k)\geq \delta\|\bm z\|_2-1$ for all~$\bm z\in\R^{d_{\bm z}}$ and~$k\in\set K$, and thus we have
      \[
      \sum_{k\in \set{K}}  \sum_{i\in \set{I}_k} \lambda_{ik} d\left(  \hat{\bm z}_k + \frac{\bm v_{ik}}{\lambda_{ik}},  \hat{\bm z}_k \right) \le \epsilon\quad \implies \quad \sum_{k\in \set{K}} \sum_{i\in \set{I}_k} \|\bm v_{ik}\|_2 \le \frac{1+\epsilon}{\delta},
      \]
      where we used the elementary identity~$\sum_{k\in \set{K}} \sum_{i\in \set{I}_k} \lambda_{ik}=1$. Thus, the last constraint group in~\eqref{eq:dot-cvx-explicit} forces the variables~$\{\bm v_{ik}\}_{ik}$ to reside within a bounded set, as well. In conclusion, we have shown that the objective function of~\eqref{eq:dot-cvx-explicit} is upper semicontinuous and that the feasible region is both closed and bounded and therefore compact. Thus, problem~\eqref{eq:dot-cvx-explicit} is indeed solvable.
      \hfill \qed

      ~\\[-8mm]

      \rev{
      \noindent \textbf{Proof of Proposition~\ref{prop:show_me_the_distributions}.} $\;$
      Under the assumptions of the proposition, problem~\eqref{eq:dot-cvx-explicit} is solvable and has the same optimal value as~\eqref{eq:uq-ot}. Even though~\eqref{eq:dot-cvx-explicit} is reminiscent of a restriction of~\eqref{eq:uq-ot} that evaluates the worst-case expected disutility over all $I$-point distributions in~$\mathbb B_\epsilon(\hat \bP)$, the solvability of~\eqref{eq:dot-cvx-explicit} does {\em not} imply that~\eqref{eq:uq-ot} admits a maximizer, see, {\em e.g.}, \citet[Example~2]{med17} or \citet[Example~4]{kmns19}.

      By construction, $\set I^{+}_k$, $\set I^{0}_k := \{ i \in \set I_k \ | \ \lambda_{ik}^\star = 0, \  \bm v_{ik}^\star = \bm 0 \}$ and $\set I^{\infty}_k$ form a partition of $\set I_k$ for every $k \in \set K$. Note that if~$\lambda_{ik}^\star = 0$, then the constraints $\lambda_{ik} c_\ell (\hat{\bm z}_k + \bm v_{ik}/\lambda_{ik} ) \le 0$, $\ell\in\set L$, of problem~\eqref{eq:dot-cvx-explicit} imply via Lemma~\ref{lemma:recession_directions}~\emph{(ii)} that~$\bm v_{ik}^\star$ is a recession direction of the support set~$\set S$. In particular, if~$\set S$ is bounded, this implies that $\bm v_{ik}^\star= \bm 0$. We may thus conclude that~$\set I_k^\infty=\emptyset$ whenever~$\set S$ is bounded. The converse implication does not hold in general.
      
      Assume first that $\set I^{\infty}_k= \emptyset$ for every $k \in \set K$. To see that $\mathbb P^\star$ defined in~\eqref{eq:OT:discrete_distr_optimal} is optimal in~\eqref{eq:uq-ot}, observe that the constraints of~\eqref{eq:dot-cvx-explicit} imply that~$\mathbb P^\star\in\mathbb B_\epsilon(\hat \bP)$ and that the expected disutility~$\mathbb E_{\mathbb P^\star}[ g(\bmt z)]$ is at least as large as the optimal value~$\sum_{k\in \set{K}} \sum_{i\in \set{I}_k} \lambda^\star_{ik} g_i (  \hat{\bm z}_k + \bm v^\star_{ik}/\lambda^\star_{ik})$ of~\eqref{eq:dot-cvx-explicit}. However, as the suprema of~\eqref{eq:uq-ot} and~\eqref{eq:dot-cvx-explicit} match, it is clear that~$\bP^\star$ must be optimal in~\eqref{eq:uq-ot}.

      Assume now that $\set I^{\infty}_k \neq \emptyset$ for some $k \in \set K$. To see that the discrete distributions defined in~\eqref{eq:OT:sequence_optimal} are asymptotically optimal, we first show that~$\mathbb P_n\in\mathbb B_\epsilon(\hat \bP)$ whenever $n \ge |\set I^{\infty}_k|$ for every $k \in \set K$. Indeed, in this case it is easy to see that $\lambda_{ik}(n)\ge 0$ for all~$i\in \set I^{+}_k \cup \set I^{\infty}_k$ and~$k\in\set K$ and that~$\sum_{k\in \set{K}}  \sum_{i\in \set{I}^+_k\cup \set I^{\infty}_k} \lambda_{ik}(n)=1$ because~$\sum_{i\in\set I^+_k} \lambda^\star_{ik}=\hat p_k$ for every~$k\in\set K$. In addition, note that~$\bm z_{ik}(n)\in\set S$ for every~$i\in \set I^{+}_k \cup \set I^{\infty}_k$ and~$k\in\set K$ thanks to the constraints of problem~\eqref{eq:dot-cvx-explicit} and because~$\bm v^\star_{ik}$ is a recession direction of~$\set S$ whenever~$i\in\set I^\infty_k$. In summary, these insights imply that~$\bP_n \in \amb P_0(\set S)$. Finally, moving mass~$\lambda_{ik}(n)$ from~$\hat{\bm z}_{k}$ to~$\bm z_{ik} (n)$ for every~$i\in \set I^{+}_k \cup \set I^{\infty}_k$ and~$k\in\set K$ incurs a total cost of
      \begin{align*}
      & \sum_{k\in \set{K}}  \sum_{i\in \set{I}^+_k\cup \set I^{\infty}_k} \lambda_{ik}(n) d (\bm z_{ik} (n),  \hat{\bm z}_{k}) \\
      &\quad = \; \sum_{k\in \set{K}}  \sum_{i\in \set{I}^+_k} \lambda_{ik}^\star \left(1 - \frac{ |\set I^{\infty}_k| }{n} \right)  d\left(  \hat{\bm z}_k + \frac{\bm v^\star_{ik}}{\lambda_{ik}^\star},  \hat{\bm z}_k \right)
      + \sum_{k\in \set{K}}  \sum_{i\in \set{I}^\infty_k} \frac{\hat p_k}{n} d\left(  \hat{\bm z}_k + n \frac{\bm v^\star_{ik}}{\hat p_k},  \hat{\bm z}_k \right)\\
      &\quad \leq \; \sum_{k\in \set{K}}  \sum_{i\in \set{I}^+_k} \lambda_{ik}^\star d\left(  \hat{\bm z}_k + \frac{\bm v^\star_{ik}}{\lambda_{ik}^\star},  \hat{\bm z}_k \right)
      + \sum_{k\in \set{K}}  \sum_{i\in \set{I}^\infty_k} \lim_{n\rightarrow \infty} \frac{\hat p_k}{n} d\left(  \hat{\bm z}_k + n \frac{\bm v^\star_{ik}}{\hat p_k},  \hat{\bm z}_k \right)\\
      & \quad = \; \sum_{k\in \set{K}}  \sum_{i\in \set{I}_k} \lambda_{ik}^\star d\left(  \hat{\bm z}_k + \frac{\bm v^\star_{ik}}{\lambda_{ik}^\star},  \hat{\bm z}_k \right) \; \leq \; \epsilon,
      \end{align*}
      where the first equality follows from the definitions of~$\lambda_{ik}(n)$ and~$\bm z_{ik}(n)$, and the first inequality holds because the transportation cost~$d(\bm z,\bm z')$ is non-negative and convex in~$\bm z$, which implies that both terms in the second line are non-decreasing in~$n$. The second equality in the above expression exploits our definition of the convex perspective for~$\lambda^\star_{ik}=0$, and the second inequality follows from the constraints of problem~\eqref{eq:dot-cvx-explicit}. Similarly, by using our conventions for the convex perspective, one can show that the asymptotic expected disutility~$\lim_{n\rightarrow \infty}\mathbb E_{\mathbb P_n}[ g(\bmt z)]$ is at least as large as the optimal value~$\sum_{k\in \set{K}} \sum_{i\in \set{I}_k} \lambda^\star_{ik} g_i (  \hat{\bm z}_k + \bm v^\star_{ik}/\lambda^\star_{ik})$ of~\eqref{eq:dot-cvx-explicit}. However, as the suprema of~\eqref{eq:uq-ot} and~\eqref{eq:dot-cvx-explicit} match, it is clear that the distributions~$\bP_n$, $n\in\mathbb N$, must be asymptotically optimal in~\eqref{eq:uq-ot}.
      \hfill \qed}

      \section{Auxiliary Results} \label{app:om}

      \begin{app_prop}[Properties of Conjugate Functions] \label{prop:conjugate-dual}
      The conjugate of a function $f$ is closed and convex. Moreover, if $f$ is closed and convex, then $f^{**} = f$. Finally, a convex function $f$ is proper if and only if $f^*$ is proper.
      \end{app_prop}

      \begin{proof}
      The conjugate $f^*$ is a pointwise supremum of affine functions, and hence it is closed and convex. The other claims  follow from \citet[Theorem~12.2]{Rockafellar1970}.
      \end{proof}



      \begin{app_prop}[Properties of Convex Perspective Functions] \label{prop:perspective-convex}
      If $f:\R^{d_{\bm x}}\rightarrow  \overline{\R}$ is proper, closed and convex, then its convex perspective is also proper, closed and convex.
      \end{app_prop}

      \begin{proof}
      Convexity and properness follow from page~35 of \citet{Rockafellar1970}, while closedness follows from page~67 and Theorem~13.3 of \citet{Rockafellar1970}.
      \end{proof}

      \begin{app_prop}[Conjugates of Perspective Functions] \label{prop:perspective}
      If $f: \R^{d_{\bm x}} \rightarrow  \overline{\R}$  is proper and convex, then for any $t > 0$ we have
      \[
      h^*(\bm w)= \begin{cases}
      tf^*(\bm w/ t) &\textnormal{if }  h(\bm x) = tf(\bm x), \\
      tf^*(\bm w)  & \textnormal{if } h(\bm x) = tf (\bm x/t).
      \end{cases}
      \]
      \end{app_prop}

      \begin{proof}
      The claim follows from Theorem~16.1 of \citet{Rockafellar1970}.
      \end{proof}

      \begin{app_prop}[Conjugates of Sums] \label{prop:inf-conv}
      If $g_1,  \ldots, g_J:\R^{d_{\bm x}}\rightarrow \overline{\R}$ are proper convex functions, then
      \begin{equation} \label{eq:inf_cov}
      \left( \sum_{j\in \set{J}} g_j \right)^* (\bm w) \; \le \; \inf_{\{ \bm{w}_j \}_{j \in \mathcal{J}}} \left\{  \sum_{j\in \set{J}} g_j^*(\bm w_j) \ \Big| \ \sum_{j\in \set{J}} \bm w_j = \bm w \right\} \qquad \forall \bm w\in\R^{d_{\bm x}}.
      \end{equation}
      If $\cap_{j\in \set{J}} {\rm ri}(\dom (g_j)) \ne \emptyset$, then the inequality is tight, and the minimum is attained for every $\bm w$.
      \end{app_prop}

      \begin{proof} For any $j\in\mathcal J$, we denote by~${\rm cl} (g_j)$ the {\em closure} of the function~$g_j$, that is, the largest closed function that resides underneath~$g_j$. By the definition of the conjugate we have
      \begin{equation*}
      \begin{array}{l@{}l}
      \displaystyle \left( \sum_{j\in \set{J}} g_j \right)^* (\bm w)
      \;\; & = \;\;
      \displaystyle \sup_{\bm x} \left\{\bm w^\top \bm x - \sum_{j\in \set{J}} g_j (\bm x) \right\}
      \;\; \le \;\;
      \displaystyle \sup_{\bm x} \left\{\bm w^\top \bm x - \sum_{j\in \set{J}} {\rm cl} (g_j) (\bm x) \right\} \\[6mm]
      & = \;\;
      \displaystyle \left( \sum_{j\in \set{J}} {\rm cl} (g_j) \right)^* (\bm w)
      \;\; \le \;\;
      \displaystyle \inf_{\{ \bm{w}_j \}_{j \in \mathcal{J}}} \left\{ \sum_{j\in \set{J}} g_j^*(\bm w_j) \ \Big| \ \sum_{j\in \set{J}} \bm w_j = \bm w \right\},
      \end{array}
      \end{equation*}
      where the first inequality holds because $g_j(\bm x) \ge {\rm cl} (g_j)(\bm x)$ for all $\bm x \in \R^{d_{\bm x}}$, $j\in \set J$, while the second inequality follows from Theorem~16.4 of \citet{Rockafellar1970}. If $\cap_{j\in \set{J}} {\rm ri}(\dom (g_j)) \ne \emptyset$, then Theorem~16.4 of \citet{Rockafellar1970} further implies that both inequalities become equalities.
      \end{proof}

      Proposition~\ref{prop:inf-conv} asserts that the conjugate of a sum of proper convex functions (the left-hand side of \eqref{eq:inf_cov}) provides a lower bound on the infimal convolution of the conjugates of these functions (the right-hand side of~\eqref{eq:inf_cov}). This lower bound becomes tight if the relative interiors of the domains of the convex functions have a point in common. Under the same condition one can show that the epigraph of the infimal convolution coincides with the Minkowski sum of the epigraphs of the conjugate functions $g_j^*$, $j\in \set J$, see, \eg, \citet[Exercise~1.28~(a)]{rw09}.

      \begin{app_ex}[Conjugates of Sums] The inequality in~\eqref{eq:inf_cov} can be strict. To see this, assume that $\bm x\in\R^2$ and $J=2$, and define $g_1(\bm x)= \delta_{\{1\}}(x_1)$ and $g_2(\bm x)= \delta_{\{-1\}}(x_1)$. Note that both $g_1$ and $g_2$ are proper and convex. In fact, they are even closed. As the domains of $g_1$ and $g_2$ have an empty intersection, we may conclude that $(g_1+g_2)^*(\bm w)=-\infty$ for every $\bm w\in\R^2$. A direct calculation shows that $g^*_1(\bm w)=w_1+ \delta_{\{0\}}(w_2)$ and $g^*_2(\bm w)=-w_1 + \delta_{\{0\}}(w_2)$, which in turn implies that
      \[
      \inf \left\{ g_1^*(\bm w_1)+ g_2^*(\bm w_2) \ \Big| \ \bm w_1+\bm w_2= \bm w \right\}=\left\{\begin{array}{cl}
      -\infty & \text{if } w_2=0, \\ +\infty & \text{otherwise.}
      \end{array} \right.
      \]
      Thus, the gap between the left and the right-hand side in~\eqref{eq:inf_cov} amounts to $\infty$ unless $w_2=0$.
      \end{app_ex}

      \begin{app_prop}[Properties of Partial Conjugates] \label{prop:partial-conjugate-dual}
      For any function $f: \R^{d_{\bm x}} \times \R^{d_{\bm z}} \rightarrow  \overline{\R}$, the partial conjugate $f^{*1}$ is closed and convex (jointly in both arguments) if $-f$ is closed and convex in its second argument. Similarly, the partial conjugate $f^{*2}$ is closed and convex if $-f$ is closed and convex in its first argument.
      \end{app_prop}
      \begin{proof}
      For any fixed $\bm x\in  \R^{d_{\bm x}}$, the functions~$\bm{w}^\top \bm x $ and ~$- f(\bm x, \bm z)$ are closed and convex in~$\bm w$ and~$\bm z$, respectively. Thus, the partial conjugate~$f^{*1} (\bm{w}, \bm z)=\sup_{\bm x\in  \R^{d_{\bm x}}} \left\{ \bm{w}^\top \bm x - f(\bm x, \bm z) \right\}$ is closed and convex jointly in $\bm w$ and $\bm z$ as a pointwise supremum of closed and convex functions. A similar argument can be made for the partial conjugate~$f^{*2}$.
      \end{proof}

      \rev{Below we provide an example of two mutually dual convex optimization problems with a strictly positive duality gap. We also showcase that the presence of a positive duality gap critically depends on the representation of these problems, that is, simple equivalent reformulations of the primal (dual) may change the dual (primal) and eliminate the duality gap.
      
      \begin{app_ex}[Representation-Dependence of Duality Results] \label{ex:duality_gap}
      Consider an instance of problem~\eqref{eq:p-co} adapted from
      Exercise~5.21 by \citet{bv04} with two decision variables $x_1$ and~$x_2$, a convex objective function $f_0(\bm x)= e^{-x_1}$ and a single convex constraint function defined through $ f_1(\bm x) = x_1^2/x_2$ if $x_2\ge 0$ and $ f_1(\bm x) = \infty$ otherwise. The fraction $x_1^2/x_2$ should be interpreted as the convex perspective of $x_1^2$ whenever $x_2\ge 0$.
      Thus, $f_0$ and $f_1$ are both proper and closed.
      Note that this instance of~\eqref{eq:p-co} violates both conditions in Theorem~\ref{prop:strong-duality}. Moreover, any feasible solution satisfies $x_1=0$ and thus attains the optimal value~$1$ of~\eqref{eq:p-co}. A direct calculation reveals that
      \begin{equation*}
      f_0^* (\bm w_0) 
      = \left\{ \begin{array}{cl}  -w_{01} \log(-w_{01}) + w_{01} & \ { \rm if} \  w_{01}\le 0 \text{ and } w_{02}=0, \\ +\infty   & \ \rm{otherwise,}\end{array} \right.
      \end{equation*}
      where we use the standard convention that $0\log(0)=0$, and
      \begin{equation*}
      f_1^* (\bm w_1)
      = \left\{ \begin{array}{cl}  0 & \ { \rm if} \ \frac{1}{4} w_{11}^2+w_{12} \le 0,\\
      +\infty   & \ \rm{otherwise.}
      \end{array} \right.
      \end{equation*}
      The dual problem~\eqref{eq:d-co} can therefore be expressed as
      \begin{equation*}
      \begin{array}{c@{\quad}l@{\quad}l}
      \displaystyle \sup & \displaystyle  w_{01} \log(-w_{01}) - w_{01}\\
      \displaystyle \subj & \displaystyle  \bm w_0+\bm w_1 = \bm 0 \\
      & \displaystyle \frac{w_{11}^2}{4\lambda_1} +w_{12} \le 0, \;\; w_{01} \le 0, \;\; w_{02}=0 \\
      & \bm{w}_0, \bm{w}_1 \ {\rm free} \\
      & \lambda_1 \ge 0.
      \end{array}
      \end{equation*}
      Note that this instance of~\eqref{eq:d-co} violates both conditions in Theorem~\ref{prop:strong-duality}. Any feasible solution satisfies $\bm w_0 = \bm w_1=\bm 0$ and thus attains the optimal value~$0$ of~\eqref{eq:d-co}. We conclude that the duality gap amounts to~$1$. Furthermore, Lemma~\ref{lem:dual_of_d_is_p} allows to recover the primal problem~\eqref{eq:p-co} by dualizing~\eqref{eq:d-co}.
      However, all of these conclusions break down if we simplify~\eqref{eq:p-co} or~\eqref{eq:d-co} by eliminating redundant constraints and variables. For example, as any primal feasible solution satisfies $x_1=0$, problem~\eqref{eq:p-co} is equivalent to the linear program $\inf\{1\mid x_2\ge 0\}$, which admits a Slater point and therefore has a strong dual that is no longer equivalent to~\eqref{eq:d-co}. Similarly, as any dual feasible solution satisfies $\bm w_0 = \bm w_1=\bm 0$, problem~\eqref{eq:d-co} is equivalent to the linear program $\sup\{0\mid \lambda_1\ge 0\}$, which admits a Slater point and therefore has a strong dual that is no longer equivalent to~\eqref{eq:p-co}. Therefore, the dual of the dual may not be equivalent to the primal if one simplifies the dual problem. 
      \end{app_ex}

      Based on the data of the primal problem~\eqref{eq:p-co}, it is often difficult to verify whether the dual problem~\eqref{eq:d-co} admits a Slater point. However, a dual Slater point is guaranteed to exist whenever the primal feasible region is nonempty and bounded. \rev{This result plays an important role in Section~\ref{sec:rco}, where we need to verify that certain dualized embedded optimization problems admit Slater points in order to invoke strong duality for the outer optimization problems.}

      \begin{app_prop}[Sufficient Condition for a Dual Slater Point] \label{prop:bounded-slater}
      If the feasible region of~\eqref{eq:p-co} is nonempty and bounded, then~\eqref{eq:d-co} admits a strict Slater point.
      \end{app_prop}

      The next example shows that the reverse implication of Proposition~\ref{prop:bounded-slater} does not hold in general.

      \begin{app_ex}[Unbounded Dual Feasible Region for Primal with a Slater Point] \label{example:unbounded_dual}
      Consider an instance of problem~\eqref{eq:p-co} with $I=3$, $f_0(x) = x$, $f_1(x) = x-1$, $f_2(x) = 1-x$ and $f_3(x) = -x$. It can be easily verified that~\eqref{eq:p-co} admits a Slater point. We can readily verify that the corresponding dual problem~\eqref{eq:d-co} can be expressed as
      \begin{equation*}
      \begin{array}{l@{\quad}l@{\quad}l}
      \sup &  - \lambda_1 + \lambda_2   \\
      \subj  & w_0 + w_1 + w_2 + w_3 = 0, \;\; w_0  = 1 \\
      &   w_1  = \lambda_1, \;\; w_2 = - \lambda_2, \;\; w_3  \le 0 \\
      &    w_0, w_1, w_2 \textnormal{ free} , \;\; \bm \lambda  \ge \bm 0.
      \end{array}
      \end{equation*}
      Clearly, the feasible region of~\eqref{eq:d-co} is unbounded.
      \end{app_ex}

      The following remark is useful in Section~\ref{sec:rco} when we wish to replace embedded optimization problems with their duals without changing the feasible region of the outer optimization problem.

      \begin{app_rem}[Strict Inequalities] \label{rem:strict-weak-inequalities}
      If~\eqref{eq:p-co} admits a strict Slater point~$\bm x^{\rm S}$, then any feasible solution~$\bm x$ of~\eqref{eq:p-co} can be expressed as the limit of a sequence of strict Slater points $\bm x_n = \frac{1}{n} \bm x^{\rm S} + (1 - \frac{1}{n}) \bm x$,~$n\in \mathbb N$, and its objective function value satisfies $f_0(\bm x) = \liminf_{n \rightarrow \infty} f_0 (\bm x_n)$. Indeed, we have
      \[
      \displaystyle f_0(\bm x)
      \; \le \;
      \liminf_{n\rightarrow \infty} f_0(\bm x_n)
      \; \le \;
      \liminf_{n\rightarrow \infty} \left\{ \frac{1}{n} f_0(\bm x^{\rm S}) + \left( 1 - \frac{1}{n} \right) f_0(\bm x)\right\}
      \; = \;
      f_0(\bm x),
      \]
      where the two inequalities follow from the closedness and the convexity of $f_0$, respectively. Therefore, replacing weak inequalities by strict inequalities in~\eqref{eq:p-co} does not change the infimum of~\eqref{eq:p-co} if~\eqref{eq:p-co} admits a strict Slater point.
      \end{app_rem}
}

      \begin{app_ex}[Non-Convexity of Problem~\eqref{eq:db-ro}] \label{example:nonconvex}
      Consider an instance of problem~\eqref{eq:pw-ro} with $I=d_{\bm x}= d_{\bm z}=1$, $\set Z = \R$, $f_0(x, z_0) = x z_0$ and $f_1(x,z_1) = \frac{1}{2}x^2 + z_1$. Then, we can readily compute
      \[
      f_0^{*1} (w_0 ,z_0)= \begin{cases}
      0 & \text{if $w_0 = z_0$}\\
      \infty & \text{if $w_0 \neq z_0$}
      \end{cases}
      \quad  \text{ and } \quad  f_1^{*1} \left( w_1,z_1 \right)= \frac{w_1^2}{2}  - z_1,
      \]
      which results in the following instance of problem~\eqref{eq:db-ro}.
      \begin{equation*}
      \begin{array}{l@{\quad}l@{\quad}l}
      \sup & \displaystyle - \frac{w_1^2}{2\lambda_1} + \lambda_1 z_1 \\
      \subj  & w_0 + w_1 = 0, \;\; w_0  = z_0 \\
      &    w_0, w_1, z_0,z_1 \textnormal{ free}, \;\; \lambda_1 \ge 0
      \end{array}
      \end{equation*}
      Using the format $(w_0, w_1, z_0, z_1, \lambda_1)$ to denote solutions of~\eqref{eq:db-ro}, it is easy to verify that both $(1, -1, 1, 2, \frac{1}{2} )$  and $( 0,0,0, 0 , 0)$ are feasible in~\eqref{eq:db-ro} with the same objective value $0$. Even though their equally weighted convex combination $(\frac{1}{2}, -\frac{1}{2}, \frac{1}{2}, 1, \frac{1}{4})$ is also feasible, its objective value amounts to $-\frac{1}{4} < \frac{1}{2}\cdot 0 + \frac{1}{2} \cdot 0$. Therefore, the instance of~\eqref{eq:db-ro} at hand is non-convex.
      \end{app_ex}

      \begin{app_prop}[Conjugates of Partial Conjugates] \label{prop:partial-conjugates}
      If $f:\R^{d_{\bm x}} \times \R^{d_{\bm z}} \rightarrow  \overline{\R}$ is closed and convex in its first argument, and $-f$ is closed and convex in its second argument, then  $(f^{*1})^* = (-f)^{*2}$ and $((-f)^{*2})^* =f^{*1}$.
      \end{app_prop}
      \begin{proof}
      The conjugate of $f^{*1}$ with respect to both of its arguments is given by
      \begin{align*}
      (f^{*1})^*(\bm x, \bm y) \; &= \; \sup_{\bm w,\bm z}  \left\{\bm x^\top \bm w + \bm y^\top \bm z - f^{*1}(\bm w,\bm z) \right\} \\
      & = \; \sup_{\bm z} \left\{ \bm y^\top \bm z + \sup_{\bm w} \left\{ \bm x^\top \bm w - f^{*1}(\bm w,\bm z) \right\} \right\} \\
      & = \; \sup_{\bm z} \left\{  \bm y^\top \bm z +  f(\bm x,\bm z)\right\}
      \; = \; (-f)^{*2} (\bm x, \bm y),
      \end{align*}
      where the third equality holds because $f$ is closed and convex in its first argument, which implies that $(f^{*1})^{*1}=f$; see Proposition~\ref{prop:conjugate-dual}. This establishes that $(f^{*1})^*= (-f)^{*2}$.  Since $f^{*1}$ is jointly closed and convex in both of its arguments due to Proposition~\ref{prop:partial-conjugate-dual},  Proposition~\ref{prop:conjugate-dual} further implies that $((-f)^{*2})^* = (f^{*1})^{**}  =f^{*1}$.
      \end{proof}
      
      \begin{app_lem}[Recession Directions]
      \label{lemma:recession_directions}
      The following statements hold.
      \begin{enumerate}[label=(\roman*)]
      \item A vector $\bm v\in \R^{d_{\bm z}}$ is a recession direction for the function $c_\ell$ if and only if $0c_\ell(\bm v/0) \le 0$.
      \item A vector $\bm v\in \R^{d_{\bm z}}$ is a recession direction for the set $\supp$ if and only if $0c_\ell(\bm v/0) \le 0$ for all~$\ell \in \set{L}$.
      \end{enumerate}
      \end{app_lem}
      \begin{proof}
      The result follows from Theorem 8.6 by \citet{Rockafellar1970}. To keep this paper self-contained, however, we provide an alternative proof using our notation. As for assertion \emph{(i)}, assume first that $\bm v$ is a recession direction for $c_\ell$. Thus, for any $\bm z\in \R^{d_{\bm z}}$ with $c_\ell (\bm z) \le 0$ we have
      \begin{align*}
      \quad  0  \; \ge \;  \frac{1}{t} c_\ell (\bm z + t \bm v)  \qquad \forall t > 0  \quad \implies \quad   0 \; & \ge \;  s c_\ell \left( \frac{s \bm z +  \bm v}{s} \right) \qquad \forall s > 0 \\
      \implies \quad   0 \; & \ge \;  \liminf_{s\downarrow 0} s c_\ell \left( \frac{s \bm z +  \bm v}{s} \right) \; \ge \; 0 c_\ell (\bm v / 0).
      \end{align*}

      Assume next that $0c_\ell (\bm v/0) \le 0$, and fix any $\bm z\in \R^{d_{\bm z}}$ with $c_\ell(\bm z)\le 0$. Thus, we have
      \begin{equation*}
      \begin{array}{l@{}l@{\qquad}l}
      c_\ell (\bm z + t \bm v)
      \;\; & = \;\;
      \left( \frac{1}{2} \cdot 2 + \frac{1}{2}  \cdot0\right) c_\ell \left( \frac{\frac{1}{2}  \cdot 2 \bm z + \frac{1}{2} \cdot 2t \bm v}{ \frac{1}{2} \cdot 2 + \frac{1}{2} \cdot 0}\right) & \forall t > 0 \\
      & \le \;\;
      \frac{1}{2}  \left[ 2 c_\ell \left( \frac{2\bm z}{2}\right) \right] + \frac{1}{2} \left[ 0 c_\ell \left( \frac{2t \bm v}{0} \right) \right] =  c_\ell (\bm z) + 0 c_\ell (t\bm v/ 0) \le 0 & \forall t > 0,
      \end{array}
      \end{equation*}
      where the first equality is trivial because $\frac{1}{2} \cdot 2 +\frac{1}{2} \cdot 0 = 1$, and the inequality follows from the convexity of the convex perspective established in Proposition~\ref{prop:perspective-convex}. The second equality exploits the properness of~$c_l^*$  and the positive homogeneity of support functions of nonempty sets. The last inequality holds because $c_\ell(\bm z) \le 0$ by assumption and because $0c_\ell (\bm v / 0) \le 0$ implies that $0c_\ell (t\bm v/ 0) \le 0$. As the above reasoning applies for any $t > 0$, we conclude that $\bm v $ is indeed a recession direction for~$c_\ell$.

      Assertion \emph{(ii)} follows from assertion \emph{(i)} and the observation that $\bm v$ is a recession direction for $\set Z$ if and only if $\set Z$ is nonempty and $\bm v$ is a recession direction for every $c_\ell$,  $\ell \in \set{L}$.
      \end{proof}

      The following three remarks discuss various generalizations of the main theorems of Section~\ref{sec:rco}.
      
      \begin{app_rem}[Equivalence of~\eqref{eq:db-ro} and~\eqref{eq:db-cvx-ro} without a Strict Slater Point]
      \label{rem:positive-lambda}
      Proposition~\ref{prop:d-b=d-b-cvx-ro}~\emph{(ii)} remains valid if~\eqref{eq:db-ro} admits a feasible solution $(  \{ \bm w_i, {\bm z}_i\}_i,  \bm \lambda )$ with $\bm \lambda> \bm 0$ instead of a strict Slater point. Indeed, the only property of a strict Slater point needed in the proof is that $\bm \lambda>\bm 0$.
      \end{app_rem}
	
      \begin{app_rem}[Strong Duality for~\eqref{eq:pw-ro} and~\eqref{eq:db-ro} without a Strict Slater Point]
      \label{rem:positive-lambda-sd}
      Theorem~\ref{thm:p-w=d-b-ro}~(iii) remains valid if~\eqref{eq:db-ro} admits a Slater point $(  \{ \bm w_i, {\bm z}_i\}_i,  \bm \lambda )$ with $\bm \lambda> \bm 0$ instead of a strict Slater point. Similarly as for Proposition~\ref{prop:d-b=d-b-cvx-ro}~(ii), the only property of $(  \{ \bm w_i, {\bm z}_i\}_i,  \bm \lambda )$ required in the proof, beyond it being a Slater point, is that $\bm \lambda>\bm 0$.
      \end{app_rem}

      \begin{app_rem}[Heterogeneous Uncertainty Sets]
      \label{rem:uncertain Z_i}
      All results of Section~\ref{sec:rco} extend in a straightforward manner to situations in which the objective function and the
      constraints of problem~\eqref{eq:pw-ro} are equipped with individual uncertainty sets $\set Z_i$, $i\in \set I_0$, all of which satisfy assumption \textbf{(C)}.
       \end{app_rem}

    {
      \begin{app_ex}[Random Matrix Theory]\label{example:random_matrix_theory} The techniques developed in Section~\ref{sec:extensions} allow us to analyze the spectral properties of random matrices governed by an ambiguous distribution. For example, they enable us to compute the worst-case conditional value-at-risk (CVaR) at level $\varepsilon\in(0,1)$ of~the (negative) largest eigenvalue $-\lambda_{\text{\rm max}}(\tilde{\bm Z})$ of a random matrix~$\tilde{\bm Z}$ in the proper convex cone $\mathbb S^{d_{\bm Z}}_+$ of positive semidefinite matrices within $\mathbb R^{d_{\bm Z}\times d_{\bm Z}}$. We assume that the distribution of~$\tilde{\bm Z}$ belongs to
      \[
	        \mathcal P=\left\{ \mathbb P\in\mathcal P_0(\mathbb S^{d_{\bm Z}}_+) \;\left|\; \mathbb E_{\mathbb P}[\tilde {\bm Z}]\preceq_{\mathbb S^{d_{\bm Z}}_+} \overline{\bm M},~\mathbb E_{\mathbb P}[\tilde {\bm Z}^{-1}]\preceq_{\mathbb S^{d_{\bm Z}}_+} \underline{\bm M} \right.\right\}.
	  \]
      Here, $\bm Z^{-1}$ is a shorthand for the function $\bm F(\bm Z)=\bm Z^{-1}$ if $\bm Z\succ_{\mathbb S^{d_{\bm Z}}_{+}}\bm 0$ and $\bm F(\bm Z)=+\bm \infty_{\mathbb S^{d_{\bm Z}}_+}$ otherwise. This function is proper, $\mathbb S^{d_{\bm Z}}_+$-convex and star $\mathbb S^{d_{\bm Z}}_+$-lower semicontinuous; see also Example~\ref{ex:C-convex-functions} in Appendix~\ref{app_before_A}. In addition, $\lambda_{\text{\rm max}}(\bm Z)$ is proper, closed and convex in the usual sense. By using Jensen's inequality, one can verify that $\mathcal P$ is nonempty if and only if the generalized moment bounds $\overline{\bm M},\underline{\bm M}\in\mathbb S^{d_{\bm Z}}_+$ satisfy $\underline{\bm M}^{-1} \preceq_{\mathbb S^{d_{\bm Z}}_+} \overline{\bm M}$. By the definition of the CVaR due to \cite{ref:rockafellar2000optimization} and by Sion's minimax theorem \citep{s58}, the worst-case CVaR of $-\lambda_{\text{\rm max}}(\tilde{\bm Z})$ satisfies
      {\rm\[
	        \sup_{\mathbb P\in\mathcal P} \mathbb P\text{-CVaR}_\varepsilon (-\lambda_{\text{max}}(\tilde{\bm Z})) = 
	        \inf_{x\in\mathbb R} x+\frac{1}{\varepsilon} \sup_{\mathbb P\in\mathcal P} \mathbb E_{\mathbb P}\left[ \max\{-\lambda_{\text{max}}(\tilde{\bm Z})-x,0\} \right].
	   \]}The worst-case expectation in the above expression constitutes an instance of the generalized uncertainty quantification problem~\eqref{eq:puq-conic} that satisfies the conditions~\textbf{(C$_{\rm g}$)}, \textbf{(G$_{\rm g}$)}, \textbf{(H$_{\rm g}$)} and \textbf{(S$_{\rm g}$)}. By Theorem~\ref{thm:strong-duality-all-problems-conic}, it can be reformulated as a tractable convex minimization problem, and thus the worst-case CVaR can be computed efficiently. We emphasize that this instance of~\eqref{eq:puq-conic} is beyond the reach of existing methods in distributionally robust optimization.
      \end{app_ex}}
      
      \begin{app_lem}[Conjugates of Powers of Norms] \label{lem:p-norm}
      Assume that $\|\cdot\|$ and $\|\cdot\|_*$ are mutually dual norms on~$\R^{d_{\bm z}}$ and that $p,q \in [1,+\infty]$ satisfy $\frac{1}{p} + \frac{1}{q} = 1$. Then, the following statements hold.
      \begin{itemize}
      \item[(i)] The conjugate of $h(\bm z) = \frac{1}{p} \|\bm z\|^p $ is given by $h^*(\bm y) = \frac{1}{q} \| \bm y \|^q_*$. Here, we interpret $\frac{1}{p} \|\bm z\|^p$ as the indicator function of the closed unit ball around~$\bm 0$ with respect to~$\|\cdot\|$ if $p=+\infty$ and~$\frac{1}{q} \| \bm y \|^q_*$ as the indicator function of the closed unit ball around~$\bm 0$ with respect to~$\|\cdot\|_*$ if $q=+\infty$.
      \item[(ii)] The first partial conjugate of $d(\bm z, \bm z') = \| \bm z - \bm z' \|^p$ is given by $d^{*1}(\bm y, \bm z') = \bm y^\top\bm z' +   \varphi(q) \left\| \bm y \right\|^q_*$, where  $\varphi(q)=(q-1)^{(q-1)}/q^q$. Here, we interpret $\| \bm z - \bm z' \|^p$ as the indicator function of the closed unit ball around~$\bm z'$ with respect to~$\|\cdot\|$ if $p=+\infty$ and $\varphi(q) \left\| \bm y \right\|^q_*$ as the indicator function of the closed unit ball around~$\bm 0$ with respect to~$\|\cdot\|_*$ if $q=+\infty$.
      \end{itemize}
      \end{app_lem}
      \begin{proof}
      Assume first that $p\in (1,+\infty)$. For any fixed~$\bm z, \bm y \in\R^{d_{\bm z}}$ we then have
      \begin{align*}
      \bm z^\top \bm y - \frac{1}{p} \|\bm z\|^p \; \le \; \|\bm z\| \|\bm y\|_* - \frac{1}{p} \|\bm z\|^p \; \leq \; \max_{t\ge 0}~ t\|\bm y\|_*-\frac{1}{p}t^p \; = \; \frac{1}{q} \|\bm y\|^{q}_*,
      \end{align*}
      where the first inequality follows from the definition of the dual norm, and the equality holds because the maximum over~$t\ge 0$ is attained at~$t^\star=\|\bm y\|^{q-1}_*$. Both inequalities in the above expression are tight if we set $\bm z$ to~$\bm z^\star=(\|\bm y\|^{q-1}_*/\|\bm y\|) \bm y$. Indeed, the first inequality is tight because~$\bm z^\star$ is parallel to~$\bm y$, and the second one is tight because~$\|\bm z^\star\|=t^\star$. Therefore, we have
      \[ h^*(\bm y) \; = \; \sup_{\bm z} \left\{ \bm z^\top \bm y - \frac{1}{p} \|\bm z\|^p \right\} \; = \; \frac{1}{q} \|\bm y\|^{q}_*. \]
      Standard limit arguments show the claim for $p\in\{1,+\infty\}$, and thus assertion~{\em (i)} follows.

      Assume now again that~$p\in(1,+\infty)$. By the definition of partial conjugates, we then have
      \begin{align*}
      d^{*1}(\bm y, \bm z') \; & = \; \sup_{\bm z} \left\{ \bm y^\top \bm z  - \| \bm z - \bm z' \|^p \right\} \; = \; \bm y^\top \bm z' +  p \sup_{\bm z} \left\{ \left( \frac{\bm y}{p} \right)^\top \bm z  - \frac{1}{p} \| \bm z\|^p  \right\}\\
      & = \; \bm y^\top \bm z' +   \frac{p}{q} \left\| \frac{\bm y}{p}\right\|_*^q \; = \; \bm y^\top \bm z' +   \varphi(q) \left\|\bm y\right\|^q_* ,
      \end{align*}
      where the second and the third equality follow from the variable substitution $\bm z\leftarrow \bm z-\bm z'$ and from assertion~{\em (i)}, respectively, while the last equality follows from elementary algebra. Standard limit arguments can again be used to prove the claim for $p\in\{1,+\infty\}$, and thus assertion~{\em (ii)} follows.
      \end{proof}

      \begin{app_lem}[Growth of Non-Negative Convex Functions] \label{lem:growth}
      If $f:\R^{d_{\bm z}}\rightarrow[0,+\infty]$ is closed and convex with $f(\bm z)=0$ if and only if~$\bm z=\bm 0$, then there is~$\delta>0$ with $f(\bm z)\geq \delta \|\bm z\|_2 -1$ for all~$\bm z\in \R^{d_{\bm z}}$.
      \end{app_lem}
      \begin{proof}
      Assume for the sake of argument that there exists no~$\delta>0$ with the advertised properties. In this case, for every~$n\in\mathbb N$ there exists~$\bm z_n\in\R^{d_{\bm z}}$ such that $f(\bm z_n) <\frac{1}{n} \|\bm z_n\|_2-1$. 
      As the unit sphere in~$\R^{d_{\bm z}}$ is compact, there further exists a subsequence~$\bm z_{n_k}$, $k\in\mathbb N$, and a vector~$\bm v\in \R^{d_{\bm z}}$ that satisfies~$\lim_{k\rightarrow \infty} \bm z_{n_k}/\|\bm z_{n_k}\|_2=\bm v$. By construction, we thus have $\|\bm v\|_2=1$ and
      \[
      f(\bm v) \; = \; \liminf_{k\rightarrow\infty} f\left( \frac{\bm z_{n_k}}{\|\bm z_{n_k}\|_2}\right) \; \leq \; \liminf_{k\rightarrow\infty} \left(1-\frac{1}{\|\bm z_{n_k}\|_2}\right) f(\bm 0) + \frac{1}{\|\bm z_{n_k}\|_2} f(\bm z_{n_k}) \; = \; 0,
      \]
      where the first equality and the inequality follow from the lower semicontinuity and the convexity of~$f$, respectively, while the second equality holds because~$f(\bm 0)=0$ and because $f(\bm z_{n_k})/\|\bm z_{n_k}\|_2\leq 1/n_k$ by the construction of~$\bm z_{n_k}$. As $f$ is non-negative, the above reasoning implies that~$f(\bm v)=0$, which in turn implies that~$\bm v=\bm 0$. However, this conclusion contradicts our earlier observation that $\|\bm v\|_2=1$. Hence, our hypothesis was false, and the claim follows.
      \end{proof}

      \begin{app_lem}[Slater Points]
      \label{lem:slater-points}
      Assume that $\mathcal{X} = \{ \bm x \in \R^{d_{\bm x}} \ | \ f_i (\bm x) \le 0 \;\; \forall i \in \mathcal{I}, \ h_j (\bm x) = 0 \;\; \forall j \in \mathcal{J} \}$ is a convex set defined in terms of convex inequality constraint functions $f_i$, $i\in\set I$, and affine equality constraint functions $h_j$, $j\in\set J$. If $\set X$ admits a Slater point, then any $\bm x\in {\rm ri}(\set X)$ is a Slater point.
      \end{app_lem}
      \begin{proof}
      Let $\bm x^{\rm S}$ be a Slater point for $\set X$, which exists by assumption. Select any $\bm x\in {\rm ri}(\set X)$, and assume that $\bm x\neq \bm x^{\rm S}$ for otherwise the claim is trivial. As both $\bm x$ and $\bm x^{\rm S}$ are elements of the convex set $\set X$, all points of the form $\theta\bm x+(1-\theta)\bm x^{\rm S}$ for some $\theta\in \mathbb R$ belong to the affine hull of $\set X$. In addition, as $\bm x\in {\rm ri}(\set X)$, we may thus conclude that there exists $\varepsilon>0$ such that $\theta\bm x+(1-\theta)\bm x^{\rm S}\in\set X$ for all $\theta\in [0,1+\varepsilon]$. Setting $\theta=1+\varepsilon$, we thus find $f_i((1+\varepsilon)\bm x-\varepsilon \bm x^{\rm S})\leq 0$ and consequently
      \[
      f_i(\bm x) \; = \; f_i\left(\frac{1}{1+\varepsilon}\left((1+\varepsilon)\bm x-\varepsilon \bm x^{\rm S}\right) +\frac{\varepsilon}{1+\varepsilon} \bm x^{\rm S} \right) \; \leq \; \frac{1}{1+\varepsilon} f_i\left((1+\varepsilon)\bm x-\varepsilon \bm x^{\rm S} \right) +\frac{\varepsilon}{1+\varepsilon} f_i\left(\bm x^{\rm S}\right) \; < \; 0
      \]
      for all $i\in\set I$ such that $f_i$ is nonlinear, where the first inequality exploits the convexity of $f_i$, and the second inequality holds because~$f_i(\bm x^{\rm S})< 0$ by the definition of a Slater point.

      By using similar arguments as in the first part of the proof, one can show that there exists~$\varepsilon >0$ such that $(1+\varepsilon)\bm x-\varepsilon \bm x^{\rm S}\in \dom (f_i)$ for all $i\in\set I$ and $(1+\varepsilon)\bm x-\varepsilon \bm x^{\rm S}\in \dom (h_j)$ for all $j\in\set J$. As $\bm x^{\rm S}\in{\rm ri} (\dom (f_i))$ for all $i\in\set I$ and $\bm x^{\rm S}\in{\rm ri} (\dom (h_j)$ for all $j\in\set J$ by the definition of a Slater point, the line segment principle by \citet[Proposition~1.3.1]{b09} then implies that the point~$\bm x$ on the line segment between $\bm x^{\rm S}$ and $\theta\bm x+(1-\theta)\bm x^{\rm S}$ belongs to ${\rm ri} (\dom (f_i))$ for all $i\in\set I$ and to ${\rm ri} (\dom (h_j)$ for all $j\in\set J$. Thus, $\bm x$ is indeed a Slater point.
      \end{proof}

{
      \begin{app_rem}[Solvability of~\eqref{eq:uq-ot} under Superlinear Transportation Costs]
      \label{rem:superlinear}
      Assume as usual that the finite convex program~\eqref{eq:dot-cvx} admits a Slater point with $\lambda_{ik} > 0$ for all~$i\in\set I_k$ and~$k\in\set K$ and that the transportation cost~$d(\bm z,\hat{\bm z}_k)$ grows superlinearly in~$\bm z$ for every~$k\in\set K$. If~$( \{ \lambda_{ik}^\star, \bm v_{ik}^\star\}_{ik})$ is a maximizer of~\eqref{eq:dot-cvx-explicit} and $\set I^\infty_k\neq\emptyset$ for some~$k\in\set K$, then we have  
      \[
      \lambda^\star_{ik} d\left(  \hat{\bm z}_k + \frac{\bm v^\star_{ik}}{\lambda^\star_{ik}},  \hat{\bm z}_k \right)
      \; = \;
      0 d\left(  \hat{\bm z}_k + \frac{\bm v^\star_{ik}}{0},  \hat{\bm z}_k \right)
      \; = \;
      \lim_{\lambda_{ik}\downarrow 0} \lambda_{ik} d\left(  \hat{\bm z}_k + \frac{\bm v^\star_{ik}}{\lambda_{ik}},  \hat{\bm z}_k \right)
      \; = \;
      \infty
      \]
      for every $i\in\set I^\infty_k$, where the second equality follows from our conventions about perspective functions, while the third equality holds because the transportation cost grows superlinearly in the first argument. Thus, $( \{ \lambda_{ik}^\star, \bm v_{ik}^\star\}_{ik})$ violates the transportation budget constraint, which contradicts our assumption that it is a maximizer of the feasible problem~\eqref{eq:dot-cvx-explicit}. Hence, $\set I^\infty_k$ must be empty for every~$k\in\set K$, which in turn implies via the above discussion that~\eqref{eq:uq-ot} is solvable.
      \end{app_rem}}
      
{
      \begin{app_ex}[Shaping the Transportation Cost]\label{ex:shaping_the_costs}
      The transportation cost function~$d(\bm z,\bm z')$ 
      can be used to incorporate structural distributional information into the uncertainty quantification problem~\eqref{eq:uq-ot}. For example, if it is known that $\tilde{\bm z}$ is supported on the non-negative orthant $\mathbb R_+^{d_{\bm z}}$ and is unlikely to have small components, then one can set~$d(\bm z,\bm z')= \sum_{n=1}^{d_{\bm z}} (z_n-z_n')^2/z_n$ if $z_n>0$ for every $n=1,\ldots,d_{\bm z}$ and~$d(\bm z,\bm z')=+\infty$ otherwise. This transportation cost function satisfies condition~\textbf{(D)}. In addition, $d(\bm z,\bm z')$ tends to $+\infty$ as $\bm z$ approaches the boundary of~$\mathbb R_+^{d_{\bm z}}$. Thus, it is expensive to move probability mass to areas of the support set that are expected to have a low probability. One can show that the first partial conjugate of this transportation cost function is given by $d^{*1}( \bm y, \bm z') =\sum_{n=1}^{d_{\bm z}} 2z'_n(1-\sqrt{1-y_n})$ if $y_n\leq 1$ for every $n=1,\ldots,d_{\bm z}$ and by $d^{*1}( \bm y, \bm z') =+\infty$ otherwise. As another example, if it is known that the atoms of the nominal distribution represent random samples from the unknown true distribution that are corrupted by isotropic noise with variance $\gamma^2$, then one can set the transportation cost to the Huber loss function $d( \bm z, \bm z') = \frac{1}{2}\| \bm z  -  \bm z'\|_2^2$ if $\|\bm z - \bm z' \|_2 \le \gamma$ and $d( \bm z, \bm z') = \gamma \|\bm z- \bm z' \|_2 - \frac{\gamma^2}{2}$ otherwise. This transportation cost function satisfies condition~\textbf{(D)}. In addition, it ensures that the cost of moving probability mass over short distances~$\leq\gamma$ is small but increases linearly over longer transportation distances. One can show that the first partial conjugate of this transportation cost function is given by $d^{*1}( \bm y, \bm z') = \bm y^\top \bm z' + \frac{1}{2} \| \bm y \|_2^2$ if $\|\bm y\|_2\leq \gamma$ and by $d^{*1}( \bm y, \bm z') =+\infty$ otherwise. The results of this section imply that the uncertainty quantification problem~\eqref{eq:uq-ot} can be reformulated as a finite convex minimization problem of the form~\eqref{eq:pot-cvx} or~\eqref{eq:dot-cvx} under either of these transportation cost functions. These reformulations are new and beyond the scope of existing methods of distributionally robust optimization.
      \end{app_ex}}

      \end{document}